\documentclass[11pt]{article}

\usepackage{graphicx} 
\usepackage{amsmath,mathtools,amsthm,fancyhdr,amssymb,bbm, enumerate,float, mathrsfs}
\usepackage{tcolorbox}
\usepackage{appendix}
\tcbuselibrary{breakable}
\usepackage{listings}
\usepackage{color, comment}
\usepackage{accents}
\usepackage{hyperref}
\usepackage{indentfirst}
\usepackage{tablefootnote}
\hypersetup{
colorlinks,
linkcolor=blue,
anchorcolor=blue,
citecolor=blue
}
\allowdisplaybreaks[4]

\numberwithin{equation}{section}
\mathtoolsset{showonlyrefs,showmanualtags} 

\newcommand{\R}{\mathbb{R}}
\newcommand{\Hi}{\mathcal{H}}
\newcommand{\TMod}{\widetilde{\text{Mod}}}
\newcommand{\HMod}{\widehat{\text{Mod}}}
\newcommand{\Ni}{\mathcal{N}}
\newcommand{\Z}{\mathcal{Z}}
\newcommand{\E}{\mathcal{E}}
\newcommand{\T}{\mathcal{T}}
\newcommand{\tT}{\tilde{T}}
\newcommand{\tQ}{\tilde{Q}}
\newcommand{\tS}{\tilde{S}}
\newcommand{\tPsi}{\tilde{\Psi}}
\newcommand{\hPsi}{\hat{\Psi}}
\newcommand{\M}{\mathcal{M}}
\newcommand{\U}{\mathcal{U}}
\newcommand{\V}{\mathcal{V}}
\newcommand{\F}{\mathcal{F}}
\newcommand{\Si}{\mathcal{S}}
\newcommand{\lam}{\lambda}
\newcommand{\eps}{\epsilon}
\newcommand{\N}{\mathbb{N}}
\newcommand{\A}{\mathcal{A}}
\newcommand{\Oi}{\mathcal{O}}
\newcommand{\Li}{\mathcal{L}}
\newcommand{\B}{\mathcal{B}}
\newcommand{\J}{\mathcal{J}}
\newcommand{\Rmnum}[1]
{\uppercase\expandafter{\romannumeral#1}}
\newcommand{\dr}[1]{\left(#1\right)}

\newcommand{\scl}[2]{\left\langle #1\,,\, #2\right\rangle}
\newcommand{\brc}[1]{\left\{ #1 \right\}}
\newcommand{
\step
}[1]{\medskip\noindent\textbf{\underline{Step #1}:}}
\newcommand{
\case
}[1]{\medskip\noindent\underline{Case #1}:}
\newcommand{\smallcase}[1]{\medskip\noindent(#1)}

\newtheorem{proposition}{Proposition}[section]
\newtheorem{lemma}{Lemma}[section]
\newtheorem{theorem}{Theorem}[section]

\theoremstyle{definition}
\newtheorem{definition}{Definition}[section]
\newtheorem{remark}{Remark}[section]

\usepackage[
    left=20mm,
    right=20mm,
    top=25mm,
    bottom=25mm
]{geometry}

\title{On the Stability of Type II Blowup for the Keller-Segel System in High Dimensions}
\date{}
 \author{Thomas Y. Hou\footnote{Applied and Computational Mathematics, Caltech, Pasadena, CA. Email: \href{hou@cms.caltech.edu}{hou@cms.caltech.edu}.},\; 
 Xiang Qin\footnote{Applied and Computational Mathematics, Caltech, Pasadena, CA. Email: \href{xqin2@caltech.edu}{xqin2@caltech.edu}.},\; 
 Peicong Song\footnote{Applied and Computational Mathematics, Caltech, Pasadena, CA. Email: \href{psong2@caltech.edu}{psong2@caltech.edu}.},\;
 Zirui Wang\footnote{Department of Mathematics, Brown University, Providence, RI. Email: \href{zirui_wang1@brown.edu}{zirui\_wang1@brown.edu}.}
 }

\date{\vspace{-5ex}}

\begin{document}

\maketitle
\begin{abstract}
    We study finite-time blowup for the parabolic--elliptic Keller--Segel
system on $\R^d$ in dimensions $d\geq11$, where the problem is mass
supercritical. For every
integer $l\geq2$, we construct smooth radially symmetric solutions whose radial
mass variable concentrates the normalized stationary state $Q$ at a quantized
scale. Each blowup regime can be realized by solutions with nonnegative population density throughout their classical lifespan. More precisely, near the blowup time $T$,
\begin{equation*}
    u(t,r)=\frac{1}{\lambda^2(t)}\left[Q\left(\frac{r}{\lambda(t)}\right)
    +\epsilon\left(t,\frac{r}{\lambda(t)}\right)\right],
    \qquad
    \lambda(t)=c(T-t)^{\frac{l}{\gamma(d)}}(1+o(1)),
\end{equation*}
where $c>0$ and
$\gamma(d)=\frac12\bigl(d-2-\sqrt{(d-2)(d-10)}\bigr)$. The remainder
$\epsilon$ converges to zero in local $L^\infty$ norms and in a range of high-order homogeneous Sobolev
norms. Since $2l>\gamma(d)$, the concentration scale is strictly smaller than
the parabolic scale $\sqrt{T-t}$, and the resulting blowup is of
type~\Rmnum{2}. The $l$-th regime has exactly $l-1$ unstable radial modulation
directions and is stable within a codimension-$(l-1)$ class of suitably regular
radial initial data. The proof combines a generalized-kernel expansion driven
by the algebraic tail of $Q$, modulation analysis, coercive weighted high-order
energy estimates, and a finite-dimensional topological argument. This yields a
quantized hierarchy of stationary-state concentration rates for the
high-dimensional Keller--Segel flow.
\end{abstract}

\section{Introduction}
\label{sec:intro}

\subsection{Problem setup}\label{subsec:problem-setup}
We study the parabolic--elliptic Keller--Segel system on $\R^d$, $d\geq 11$:
\begin{equation}\label{equation:keller-segel}
    \left\{\begin{aligned}
        &\partial_t\rho=\Delta\rho-\nabla\cdot(\rho\nabla c),\\
        &\Delta c+\rho=0,\\
        &\rho(\cdot,0)=\rho_0\geq0.
    \end{aligned}\right.
\end{equation}
Here $\rho\geq0$ is the population density and $c$ is the quasistatic chemoattractant potential. The model originates in \cite{Patlak1953,KellerSegel1970}; its parabolic--elliptic reduction was introduced by J\"ager--Luckhaus \cite{JagerLuckhaus1992}, and comprehensive reviews are given in \cite{Horstmann2003,Horstmann2004}. Under the concentration scaling
\[
    \rho_\lambda(x,t)=\lambda^{-2}\rho\left(\frac{x}{\lambda},\frac{t}{\lambda^2}\right),
\]
the $L^{d/2}$ norm is invariant while the mass is multiplied by $\lambda^{d-2}$. Thus the problem is mass critical in $d=2$ and mass supercritical for $d\geq3$.

For radial solutions, writing $r=|x|$, we set
\[
    u(r,t)=\frac{1}{r^d}\int_0^r\rho(s,t)s^{d-1}ds,
    \qquad c_r=-ru,
    \qquad \rho=du+ru_r.
\]
Then \eqref{equation:keller-segel} is equivalent to the local scalar equation
\begin{equation}\label{equation:1d-keller-segel}
    \partial_t u=\partial_{rr}u+\frac{d+1}{r}\partial_r u+u(r\partial_r u+du).
\end{equation}
It inherits the scaling $u_\lambda(r,t)=\lambda^{-2}u(r/\lambda,t/\lambda^2)$. Regarding $u$ as a radial function on $\R^d$, a finite-time blowup at $T$ is of type~\Rmnum{1} when
\[
   \limsup_{t\uparrow T}(T-t)\|u(t)\|_{L^\infty(\R^d)}<+\infty,
\]
and of type~\Rmnum{2} otherwise; for results in the original variables we use the analogous convention for $\rho$. For every admissible integer $l$, we construct a quantized type~\Rmnum{2} regime and prove that it persists for a codimension-$(l-1)$ class of radial initial data.

\subsection{Background and related work}

\paragraph{Blowup for the Keller--Segel system.}
The distinction between type~\Rmnum{1} and type~\Rmnum{2} is strongly dimension dependent. In $d=2$, every finite-time blowup in the standard classical finite-mass class is necessarily of type~\Rmnum{2} \cite{YukiNaito2008}. The stationary state $\bar U(x)=8(1+|x|^2)^{-2}$ has mass $8\pi$, the threshold between global subcritical dynamics and finite-time collapse in the finite-second-moment free-energy class \cite{Blanchet2006,Blanchet2008}. The stable one-point regime above threshold concentrates exactly one copy of $\bar U$ and obeys
\begin{equation*}
\begin{aligned}
    \rho(t,x)\sim\frac{1}{\lambda^2(t)}
    \bar U\left(\frac{x-x^*(t)}{\lambda(t)}\right),\quad
    \lambda(t)=2e^{-\frac{2+\gamma_{\rm E}}{2}}\sqrt{T-t}\,
    e^{-\sqrt{\frac{|\log(T-t)|}{2}}}
    \bigl(1+o_{t\uparrow T}(1)\bigr),
\end{aligned}
\end{equation*}
where $\gamma_{\rm E}$ is the Euler--Mascheroni constant. Thus $\lambda(t)\ll\sqrt{T-t}$. This mechanism was constructed by Herrero--Vel\'azquez and Vel\'azquez \cite{Herrero1996,Velazquez2002}; Rapha\"el--Schweyer proved radial stability \cite{Raphael2014}, and Collot--Ghoul--Masmoudi--Nguyen obtained the exact constant and nonradial stability \cite{Ann.PDE22,CPAM2022}. Their spectral theory also produces unstable quantized rates, while radial universality and more elaborate multi-bubble dynamics are studied in \cite{Mizoguchi2022,BuseghinCollot2026,buseghin2023existencefinitetimeblowup,collot2024}.

For every $d\geq3$, there are backward self-similar type~\Rmnum{1} solutions. The known radial profile family is countably infinite for $3\leq d\leq9$ and contains at least two profiles for $d\geq10$ \cite{Herrero1998,Senba2005,Brenner1999,nguyen2026infinitely,collot2024stabilitytypeiselfsimilar}. Stability results of those type~\Rmnum{1} blowups can be found in \cite{Glogić2024,collot2024stabilitytypeiselfsimilar,li2025nonradialstabilityselfsimilarblowup}. Mizoguchi--Senba proved that for $3\leq d\leq9$, a nonnegative radial, radially nonincreasing solution is necessarily of type~\Rmnum{1} whenever its blowup set is not all of $\R^d$; in particular, this applies to finite-mass data \cite{MizoguchiSenba2011TypeI,Souplet2019}. This restriction on the data is essential: there are shrinking spherical-layer type~\Rmnum{2} blowups for $d\geq 3$ \cite{COLLOT2023collapsingRing} and collapsing ring type~\Rmnum{2} blowup for $d = 3$ \cite{ hou2026axisymmetrictypeiiblowup,delPino_JFA26}. For the centered stationary-state mechanism relevant here, Mizoguchi and Senba constructed positive radial type~\Rmnum{2} solutions for $d\geq11$ \cite{senba2006fast,MizoguchiSenba2007}. We resolve this branch into a quantized, codimensionally stable hierarchy. The remaining $d=10$ case is its repeated-root endpoint: we derive a formal logarithmic law, but leave its rigorous construction open. Thus our result and the formal endpoint law give a unified dimension-dependent type~\Rmnum{1}/type~\Rmnum{2} picture for radially centered stationary-state concentration, with rigorous construction in $d=10$ as the sole remaining gap.

\paragraph{Tail-driven type~\Rmnum{2} blowup in supercritical equations.}
A common pattern is seen most clearly by considering together the focusing semilinear heat and nonlinear Schr\"odinger equations
\begin{equation*}
    \partial_t v=\Delta v+|v|^{p-1}v,
    \qquad
    i\partial_t\psi+\Delta\psi+|\psi|^{p-1}\psi=0.
\end{equation*}
They share the stationary equation $\Delta Q+Q^p=0$ and the thresholds
\begin{equation*}
    p_S=\frac{d+2}{d-2},\qquad
    p_{\rm JL}(d)=
    \begin{cases}
        +\infty, & 3\leq d\leq10,\\[1mm]
        1+\dfrac{4}{d-4-2\sqrt{d-1}}, & d\geq11.
    \end{cases}
\end{equation*}
For generic $p>p_{\rm JL}$, the normalized positive state has a two-term algebraic tail $Q(r)=c_\infty r^{-a}-c_1r^{-\gamma_p}+o(r^{-\gamma_p})$, where $a=2/(p-1)$ and $\gamma_p>a$ is the smaller indicial root; write $\alpha_p=\gamma_p-a$. The slow scaling resonance generated by this tail yields, up to phase for NLS, the shared quantized pattern
\begin{equation*}
    w(t,r)\sim \lambda_l(t)^{-a}Q\left(\frac{r}{\lambda_l(t)}\right),
    \qquad
    \lambda_l(t)\sim c_l(T-t)^{l/\alpha_p},
    \qquad 2l>\alpha_p.
\end{equation*}
For the heat flow, this hierarchy was developed by Herrero--Vel\'azquez and made rigorous and classified under radial, nonresonance hypotheses in \cite{HerreroVelazquez1994Heat,Mizoguchi2004Heat,Mizoguchi2007,MatanoMerle2009}; resonant values, notably the Lepin exponent $p_L=1+6/(d-10)$, exhibit modified regimes \cite{Seki2020Lepin}. Nonnegative radial type~\Rmnum{2} blowup is excluded for $p_S<p<p_{\rm JL}$ \cite{MatanoMerle2004Heat}. Merle--Rapha\"el--Rodnianski obtained the analogous hierarchy for energy-supercritical NLS in $d\geq11$ for sufficiently large generic odd powers, transferring robust modulation and adapted-energy methods from the critical NLS theory \cite{MerleRaphaelRodnianski2015}. Parallel families occur for supercritical harmonic-map heat flow, wave maps, Yang--Mills heat flow, and semilinear wave equations \cite{GhoulIbrahimNguyen2019,GhoulIbrahimNguyen2018,Yi2025,BensouilahDuongGhoul2025,Collot2018}. Nonradial and geometric heat constructions are also known \cite{Collot2017,delpino2020newtypeiifinite,DelPinoMussoWeiZhou2026Tube}.

A striking endpoint feature is the replacement of the pure power law by logarithmic corrections. At $p=p_{\rm JL}$ the tail roots coalesce, and Seki constructed logarithmically corrected type~\Rmnum{2} heat blowup \cite{Seki2018JL}; the analogous NLS endpoint remains open. A different marginal endpoint is the seven-dimensional corotational harmonic-map heat flow, where the stable law $\lambda(t)\sim\sqrt{T-t}/|\log(T-t)|$ is now rigorous \cite{Ghoul2026}. Our repeated-root endpoint $d=10$ formally gives $\lambda(t)\sim c(T-t)^{l/4}|\log(T-t)|^{-l/[4(l-2)]}$ for $l>2$ and $\lambda(t) \sim\sqrt{T-t}\, \exp( -\frac12\sqrt{|\log(T-t)|}  )$ for $l=2$. Logarithmic rates also recur in critical problems: the planar Keller--Segel law above, the mass-critical NLS log--log regime \cite{Perelman2001,MerleNLSGAFA2003,MerleNLSAnnals2005,MerleNLSJAMS06}, and the energy-critical semilinear and harmonic-map heat flows \cite{Schweyer2012,RaphaelSchweyer2013,RaphaelSchweyer2014}. These examples reflect marginal scaling dynamics, but their mechanisms are model dependent. The present $d=10$ prediction originates from the coalescing stationary-tail and the degeneracy of the modulation dynamics.

\subsection{Main result}
\begin{theorem}[Type II blowup solutions to \eqref{equation:1d-keller-segel}]\label{theorem:main-theorem}
    Let $d\geq 11$, let $\gamma$ be defined in \eqref{defeq:gamma-Gamma} and fix an integer $l$ such that
    \begin{equation}
        2l>\gamma,\quad\text{or equivalently }l\geq 2.
    \end{equation}
    Then there exists a smooth radially symmetric solution $u$ to \eqref{equation:1d-keller-segel} that blows up in finite time $T>0$ and admits the decomposition
    \begin{equation}
        u(t,r)=\frac{1}{\lambda^2(t)}\left(Q\left(\frac{r}{\lambda(t)}\right)+\eps_{tot}\left(t,\frac{r}{\lambda(t)}\right)\right).
    \end{equation}
    Moreover, the scaling parameter satisfies
    \begin{equation}\label{equation:blowup-speed}
        \lam(t) = c(u_0)(T-t)^{\frac{l}{\gamma}}(1+o_{t\rightarrow T}(1)),\quad c(u_0)>0,
    \end{equation}
    where
\begin{equation}
u_0=u(0,\cdot)
\end{equation}
denotes the initial datum. Furthermore, we have
\begin{enumerate}
    \item [(i)] The blow-up regime is stable within a codimension-$(l-1)$ class of radial initial data.
    \item [(ii)] For every $B>0$, we have
    \begin{equation}\label{equation:stability-L-infinity-loc}
        \lim_{t\rightarrow T}\|\eps_{tot}\|_{L^{\infty}(|y|\leq B)}=0.
    \end{equation}
    \item[(iii)] We have
    \begin{equation}
        \lim_{t\rightarrow T}\|\nabla^{\sigma}\eps_{tot}\|_{L^2}=0,\quad \forall\sigma\in\left[\frac{d}{2}+2,\mathfrak{s}\right]
    \end{equation}
    where $\mathfrak{s}=\mathfrak{s}(l)$ is a sufficiently large Sobolev exponent satisfying
\begin{equation}
\mathfrak{s}(l)\rightarrow+\infty
\qquad\text{as }l\rightarrow+\infty.
\end{equation}
    
\end{enumerate}
\end{theorem}
\begin{remark}[Nonnegative density]
The solutions in Theorem~\ref{theorem:main-theorem} can be chosen with nonnegative
density $\rho(t,r)=du(t,r)+r\partial_r u(t,r)$ throughout
$0\le t<T$. Indeed, Lemma~\ref{lemma:asymptotic-Q} gives
\[
R_Q(y):=dQ(y)+yQ'(y)
       =(d-2)Q(y)+\Lambda Q(y)
       \gtrsim (1+y^2)^{-1}>0.
\]
For $\eta>0$ sufficiently small, the admissibility bounds in
Lemma~\ref{lemma:construction-Tk} and Proposition~\ref{proposition:approximate-profile}, together with the localization
\eqref{equation:construct-Q-local}, imply
\[
\sup_{y\ge0}
\frac{\big|(d+y\partial_y)(\widetilde Q_b-Q)(y)\big|}
     {R_Q(y)}
=o_{b_1\to0}(1),
\qquad |b_k|\lesssim b_1^k.
\]
We may take $\epsilon_0=0$ in Definition~\ref{definition:initial-data}. Since
$b_1(s_0)\asymp s_0^{-1}$, choosing $s_0$ sufficiently large
then gives $\rho_0\ge R_Q/2>0$, uniformly over the initial
unstable parameters used in Proposition~\ref{proposition:exist-solution-trapped}. Thus the
topological selection can be performed within this
positive-density family. Finally, \eqref{equation:keller-segel} yields
$\partial_t\rho-\Delta\rho+\nabla c\cdot\nabla\rho=\rho^2\ge0$,
so the parabolic maximum principle on every time interval
$[0,T']$, $T'<T$, preserves $\rho\ge0$.
\end{remark}

\subsection{Notations}
For each $d\geq11$, we let $\gamma, \Gamma$ to be the two roots of the equation $x^2-(d-2)x+2(d-2)=0$. In particular we have
\begin{equation}\label{defeq:gamma-Gamma}
    \gamma = \gamma(d)=\frac{d-2-\sqrt{(d-2)(d-10)}}{2},\quad\Gamma=\Gamma(d)=\frac{d-2+\sqrt{(d-2)(d-10)}}{2}.
\end{equation}
We note that $\gamma(d)$ is decreasing in $d$, with $\gamma(11) = 3$ and $\gamma(+\infty) = 2$. Hence $\gamma(d)\in (2,3]$ for $d\geq 11$. Then we define
\begin{equation}\label{defeq:h-delta}
    h=\left\lfloor\frac{1}{2}\left(\frac{d}{2}-\gamma\right)\right\rfloor\in\N,\quad \delta=\frac{1}{2}\left(\frac{d}{2}-\gamma\right)-h\in(0,1)
\end{equation}
where $\lfloor x \rfloor$ denotes the greatest integer below $x$. The integer $h$ records the critical weighted integrability threshold associated with the decay of $\Lambda Q$, so that the energy hierarchy naturally starts at the level $\E_{2(h+1)}$.

Let $L\gg 1$ be a sufficiently large even integer, we set
\begin{equation}\label{defeq:m-def}
    m=L+h+1.
\end{equation}
The largeness of $L$ is used to close the bootstrap for the top order
energy $\E_{2m}$. It ensures the sufficient decay needed to strictly improve all the energy
and stable-mode bootstrap bounds in Proposition~\ref{proposition:reduce-finite-dimension}.

For $b_1>0$ we define
\begin{equation}\label{defeq:B0-B1}
    B_0=\frac{1}{\sqrt{b_1}},\quad B_1=B_0^{1+\eta}
\end{equation}
where $\eta$ is chosen as a small enough parameter such that $0<\eta\ll1$.

We also introduce a nonincreasing cutoff function $\chi$ such that $\text{supp}\chi\subset[0,2]$ and $\chi\equiv1$ on $[0,1]$. For $M>0$ we define
\begin{equation}\label{defeq:chi-M}
    \chi_M(y)=\chi\left(\frac{y}{M}\right).
\end{equation}

For simplicity, we write
\begin{equation}\label{defeq;integrate-f}
    \int f:=\int_{0}^{+\infty}y^{d-1}f(y)dy
\end{equation}
for a function $f$ in $\R^d$ with radial symmetry. We note this definition is equivalent to integrating $f$ in $\R^d$. We also define the inner product for radial symmetric functions $f$ and $g$ as
\begin{equation}\label{defeq:inner-product}
    \langle f,g\rangle=\int fg=\int_{0}^{+\infty}y^{d-1}fgdy.
\end{equation}

Finally we introduce the repeatedly used differential operator $\Lambda$ as
\begin{equation}\label{defeq:Lambda}
    \Lambda f=y\partial_yf+2f.
\end{equation}

\subsection{Strategy of the proof}

At a structural level, the proof follows the modulation--energy--topology scheme developed for energy-supercritical harmonic-map heat flow in \cite{GhoulIbrahimNguyen2019}: one constructs an approximate finite-dimensional manifold, controls the exact flow near it by modulation and adapted energies, and finally selects the unstable parameters by a topological argument. The main task is to adapt this framework to the distinct linear and nonlinear structures of the Keller--Segel equation.\\[3pt]
\noindent\textbf{Step 1: Construction of the approximate profile.}
We work in the renormalized variables \eqref{def:change-of-variable-origin} and expand around the stationary state $Q$. Its far-field behavior
\[
    Q(y)=2y^{-2}-a_0y^{-2-\gamma}+O(y^{-4-\gamma})
\]
is decisive: the leading $y^{-2}$ tail is annihilated by the scaling operator, so the subleading exponent $\gamma$ governs the resonant dynamics. A first Keller--Segel-specific difficulty is that the linearized operator $\Li$ is not self-adjoint in the standard radial $L^2$ space. Motivated by the entropy symmetrization in \cite{Raphael2014}, we use the factorization $\Li=\B^*\A$ and construct the positive implicit weight $\M=\T_Q/\Lambda Q$, with respect to which $\Li$ is symmetric; see Lemma~\ref{lemma:structure-of-L} and \eqref{equation:adjoint-L}. This factorization also gives an explicit inverse of $\Li$ and hence the generalized kernel
\[
    T_k=(-1)^k\Li^{-k}\Lambda Q,\qquad
    T_k(y)\sim c_k^\infty y^{2k-2-\gamma},\qquad
    \Lambda T_k-(2k-\gamma)T_k=O(y^{2k-4-\gamma}).
\]
The last cancellation permits the recursive construction of $Q_b$ in Proposition~\ref{proposition:approximate-profile}. Since the profiles $T_k$ eventually grow at infinity, $Q_b$ is localized at the scale $B_1$ in Proposition~\ref{proposition:localize-Qb}. Cancelling its leading error yields the finite-dimensional system
\begin{equation*}
    (b_k)_s+(2k-\gamma)b_1b_k-b_{k+1}=0,\qquad
    -\frac{\lambda_s}{\lambda}=b_1,
\end{equation*}
whose explicit orbit $b_k^e=c_ks^{-k}$ in \eqref{defeq:bek-in-dynamic-system} produces the quantized rate $\lambda(t)\sim c(T-t)^{l/\gamma}$.
\\ [3pt]
\textbf{Step 2: Modulation and bootstrap.}
We decompose the exact solution as $v=\widetilde Q_b+\eps$ and determine $(\lambda,b_1,\ldots,b_L)$ by the weighted orthogonality conditions \eqref{relation:orthogonality-condition}. The compactly supported dual profile $\Phi_M$ is chosen so that \eqref{equation:cancellation-identity} isolates the generalized-kernel coordinates. We then bootstrap both the distance of $b$ from the explicit orbit and the hierarchy of weighted energies in Definition~\ref{definition:bootstrap-assumption}. Projecting the equation for $\eps$ gives the modulation system. Because the slow tails make the rough estimate for the top parameter $b_L$ too large, it is sharpened by adding a localized flux of the radiation at the scale $B_0$; this is the corrected coordinate appearing in Lemma~\ref{lemma:modulation-bound-improve}.
\\ [3pt]
\textbf{Step 3: Weighted energy estimates.}
The infinite-dimensional remainder is measured by
\[
    \E_{2k}=\int \M|\Li^k\eps|^2.
\]
Passing to the original variables, as in \eqref{defeq:w-eps-original-variable}, removes the nonperturbative scaling transport and exposes the dissipation furnished by the weighted factorization. The varying scale of the nonconstant weight nevertheless creates a sign-indefinite quadratic term in the top-order identity. To control it, we add lower-order time-dependent corrections to the Lyapunov functional and exploit the sharp almost-homogeneity cancellation
\[
    \frac{\Lambda\J+2\J}{\J}=O\left(\frac{1}{1+y^2}\right),\qquad \J=\M^{-1},
\]
proved in \eqref{equation:cancellation-J}. A second model-specific issue is the derivative quadratic nonlinearity $\Ni(\eps)=\eps(y\partial_y\eps+d\eps)$, which cannot be treated as a derivative-free semilinear remainder. Adapted coercivity and separate interpolation estimates near the origin and at infinity prevent derivative loss. These ingredients yield the Lyapunov inequalities of Proposition~\ref{proposition:energy-control}.
\\[3pt]
\textbf{Step 4: Topological closure.}
The estimates strictly improve all bootstrap bounds except the $l-1$ unstable coordinates. Proposition~\ref{proposition:reduce-finite-dimension} gives strict outward crossing.
If all initial points of $C=[-1,1]^{l-1}$ exited, their first exits would
define a continuous map $\Upsilon:C\to\partial C$.
The initial-face sign persists, so $\Upsilon(\xi)\neq-\xi$ on $\partial C$.
Brouwer applied to $-\Upsilon$ contradicts this.
For each fixed admissible radiation and remaining parameters, selecting
the unstable coordinates therefore yields a trapped solution.
This is the codimension statement; integrating the modulation law gives
the rate.

\begin{remark}[The critical case $d=10,\gamma=4$] In dimension $d=10$, the asymptotic behavior of $Q$ at the infinity admits an additional logarithmic correction as in Lemma~\ref{lemma:appendix-profile-Q}:
\begin{equation}
    Q(y) = \frac{2}{y^2} -c_{Q,1}\frac{\log y}{y^6} +c_{Q,2}y^{-6}+o(y^{-6}), \qquad c_{Q,1}>0,
\end{equation}
and hence
\begin{equation}
    T_k(y) = c_{T,k,1} y^{2k-6}\log y+c_{T,k,2}y^{2k-6}+o(y^{2k-6}). 
\end{equation}
Consequently, 
\begin{equation}
    \Lambda T_k = \left( 2k-4+\frac{1}{\log y} +o\left(\frac{1}{|\log y|}\right) \right)T_k.
\end{equation}
Evaluating this correction at the parabolic scale $B_0=b_1^{-1/2}$ (see \cite{Ghoul2026,Raphael2014}) formally leads to the refined modulation system
\begin{equation}\label{equation:dynamic-equation-log-correction}
    (b_k)_s+ \left( 2k-4+\frac{2}{|\log b_1|} \right)b_1b_k-b_{k+1}=0.
\end{equation}

For $l>2$, the ansatz 
\begin{equation}
    b_k(s)=\frac{c_k}{s^k}+\frac{d_k}{s^k\log s}+o\left(\frac{1}{s^k\log s}\right),\quad 1\leq k\leq l,\quad c_{l+1}=d_{l+1}=0
\end{equation}
gives $c_1=l/(2l-4),d_1=-2l/(2l-4)^2$, which leads to
\begin{equation}
    \lambda(t) = c(T-t)^{\frac l4} |\log(T-t)|^{-\frac{l}{4(l-2)}} \bigl(1+o_{t\to T}(1)\bigr).
\end{equation}

 The case $l=2$ is genuinely more critical and cannot be obtained by substituting $l=2$ into the preceding formula. Using the dynamic equation \eqref{equation:dynamic-equation-log-correction} by setting $b_3=0$, we can formally derive that
 \begin{equation}
     b_1(s)=\frac{\log s}{s}-\frac{\log\log s}{s}+O\dr{\frac{1}{s}},\quad b_2(s)=-2\left(\frac{\log s}{s}\right)^2+4\frac{\log s\log\log s}{s^2}+O\dr{\frac{|\log s|}{s^2}},
 \end{equation}
and hence predicts the stronger correction 
\begin{equation}
    \lambda(t) = c\sqrt{T-t}\, \exp\left( -\frac12\sqrt{|\log(T-t)|} \right) \bigl(1+o_{t\to T}(1)\bigr).
\end{equation}
 In particular, this is still a type-II regime. We emphasize that, in the $l>2$ setting, the logarithmic correction originates from the logarithmic term in the second-order asymptotic expansion of the steady state. This mechanism differs from \cite{Ghoul2026}, where the critical relation $2l=\gamma$ causes the leading-order coefficient in the modulation system to vanish. In the case $l=2$, both effects interact: the leading algebraic modulation dynamics degenerates, and the logarithmic correction inherited from the steady-state tail becomes the dominant contribution. This formally leads to a stretched-exponential correction to the self-similar rate, of the same qualitative form as the stable blow-up regime for the two-dimensional Keller-Segel system, see \cite{Raphael2014}.

 We do not treat the critical case d=10 in this paper, as its rigorous treatment requires a more delicate analysis of the asymptotic expansion and the associated modulation equations. 
 \end{remark}

\section{Construction of approximate blowup profile}\label{sec:profile}
In this section we construct the approximate solution to \eqref{equation:1d-keller-segel} using the same approach as in \cite{GhoulIbrahimNguyen2019}. This approach can also be found in \cite{Raphael2014,RaphaelSchweyer2013,RaphaelRodnianski2012,HillairetRaphael2012,MerleRaphaelRodnianski2015}. The key in our construction is the fact that the linearized operator $\Li$ around $Q$ admits an explicit inversion formula for $\Li^{-1}$.

Let us begin by introducing the change of variables,
\begin{equation}\label{def:change-of-variable-origin}
    u(r,t)=\frac{1}{\lam^2(t)}v(y,s),\quad y=\frac{r}{\lam(t)},\quad \frac{ds}{dt}=\frac{1}{\lam^2(t)},
\end{equation}
which leads to the dynamic rescaling equation
\begin{equation}\label{equation:dynamic-rescaling-eq}
    \partial_sv=\partial_{yy}v+\frac{d+1}{y}\partial_yv+v(y\partial_yv+dv)+\frac{\lambda_s}{\lambda}\Lambda v.
\end{equation}
Let $Q$ denote the unique steady state
solution (up to scaling) of the original equation \eqref{equation:1d-keller-segel}, as well as the leading part of the solution of \eqref{equation:dynamic-rescaling-eq}. In particular,
\begin{equation}\label{equation:steady-state}
    0=Q''+\frac{d+1}{y} Q'+Q(yQ'+dQ),\quad Q(0)=1,\quad Q'(0)=0.
\end{equation}
Our goal is to construct an approximate solution of \eqref{equation:dynamic-rescaling-eq} near $Q$. A natural approach is to linearize \eqref{equation:dynamic-rescaling-eq} around the steady state $Q$, which gives rise to the linearized operator $\Li$. We next collect the main properties of $\Li$ that will be used throughout the analysis. Before doing so, we introduce a class of admissible functions which encodes the prescribed asymptotic behavior at both the origin and infinity.
\begin{definition}[Admissible function]\label{definition:admisible-function}
    We say that a function $f\in C^{\infty}(\R_+)$ is admissible of degree $(p_1,p_2)\in\N\times\R$ if:
    \begin{enumerate}
        \item [(i)] $f$ admits a Taylor expansion of all orders at the origin,
        \begin{equation}
            f(y)=\sum_{k=p_1}^pc_ky^k+O(y^{p+1}).
        \end{equation}
        \item[(ii)] For $y\geq1$ and for all $k\in \N$, we have
        \begin{equation}
        |\partial_y^kf(y)|\lesssim y^{p_2-k}.
        \end{equation}
    \end{enumerate}
    We also say that $f$ is admissible of exact degree $(p_1,p_2)$ if in addition,
    \begin{equation}
        c_{p_1}\neq0,\quad\text{and}\quad f(y) = c_{\infty}y^{p_2}(1+o_{y\to\infty}(1)),\;\;c_\infty\neq 0.
    \end{equation}
    In this case we denote $f\sim (p_1,p_2)$.
\end{definition}
Let us now derive the main properties of $\Li$ in the following subsection.

\subsection{Structure of the linearized operator.} The structure of $\Li$ is at the heart of both the construction of the approximate solution and the derivation of the high-order Sobolev energy estimates. We start by analyzing the asymptotic behaviour of the steady state $Q$.
\begin{lemma}[Asymptotic behaviour of $Q$]\label{lemma:asymptotic-Q} Let $d\geq11$. There exists a unique solution $Q$ to \eqref{equation:steady-state}, which admits the following asymptotic behavior. For any $k\in\N$,
\begin{enumerate}
    \item [(i)] (asymptotic behaviour of $Q$)
    \begin{equation}
        Q(y)=\left\{
    \begin{aligned}
         & 1+\sum_{i=1}^kc_iy^{2i}+O(y^{2k+2}), && \text{as $y\rightarrow0$}\\
       & y^{-2}(2-a_0y^{-\gamma}+O(y^{-\gamma-2})), && \text{as $y\rightarrow+\infty$},
    \end{aligned}
    \right. 
    \end{equation}
    where $a_0>0$ and $\gamma$ is defined in \eqref{defeq:gamma-Gamma}.
    \item [(ii)] (degeneracy)
    \begin{equation}
        \Lambda Q>0,\quad \Lambda Q(y)=\left\{
    \begin{aligned}
         & 2+\sum_{i=1}^kc_i'y^{2i}+O(y^{2k+2}), && \text{as $y\rightarrow0$}\\
       & a_0\gamma y^{-2}(y^{-\gamma}+O(y^{-\gamma-2})), && \text{as $y\rightarrow+\infty$}.
    \end{aligned}
    \right. 
    \end{equation}
\end{enumerate}   
\end{lemma}
\begin{proof}
    The proof is left to Appendix~\ref{app:steady-state}.
\end{proof}
The structure of the linearized operator $\Li$ is given by the following lemma:
\begin{lemma}[Factorization of $\Li$]\label{lemma:structure-of-L}
    Let $d\geq11$ and define the first-order operators
    \begin{equation}\label{def:operator-A-B}
        \begin{gathered}
            \A\eps=-\Lambda Q\partial_y\left(\frac{\eps}{\Lambda Q}\right),\\
            \B^*\eps=\frac{1}{y^{d-1}\T_Q}\partial_y\left(y^{d-1}\T_Q\eps\right),
        \end{gathered}
    \end{equation}
    where $\T_Q$ is defined from the equation    \begin{equation}\label{def:TQ-equation}
        \left(\ln\frac{y^{d-1}\T_Q}{\Lambda Q}\right)'=\frac{d+1}{y}+yQ.
    \end{equation}
    Then it follows
    \begin{equation}\label{def:factorize-L}
    \begin{aligned}
        \Li\epsilon&=-\left[\partial_{yy}\eps+\left(\frac{d+1}{y}+yQ\right)\partial_y\eps+(y\partial_yQ+2dQ)\eps\right]\\
        &=- \frac{1}{y^{d-1}\T_Q}\partial_y\left(y^{d-1}\T_Q\Lambda Q\partial_y\left(\frac{\eps}{\Lambda Q}\right)\right)=\B^*\A\eps.
    \end{aligned}
    \end{equation}
    We can also write the operators as:
\begin{equation}\label{defeq:linear-operator-in-V}
    \begin{gathered}
        \Li\eps=\B^*\A\eps=-\partial_{yy}\eps+yV^1(y)\partial_y\eps+V^2(y)\eps,\\
        \B^*\eps=\partial_y\eps+yV^B(y)\eps,\quad\A\eps=-\partial_y\eps+yV^A(y)\eps.
    \end{gathered}
\end{equation}
\end{lemma}
\begin{remark}
    We further introduce the weight function
    \begin{equation}\label{def:MQ}
        \M_Q=\frac{\T_Q}{\Lambda Q}.
    \end{equation}
    From \eqref{def:TQ-equation} we obtain
    \begin{equation}\label{equation:integral-formula-M}
        \M_Q=Cy^2\exp\left(\int_{0}^{y}xQ(x)dx\right),
    \end{equation}
    which leads to $\M_Q>0$ for $y>0$.
    We also note the operators $\A$ and $\B^*$ are formally adjoint in the sense that
    \begin{equation}\label{equation:adjoint-A-B}
         \scl{\M_Q\A u}{w}=  \scl{\M_Q\B^* w}{u}     ,
    \end{equation}
   with the inner product defined in \eqref{defeq:inner-product}. This also shows
    \begin{equation}\label{equation:adjoint-L}
        \scl{\M_Q\Li u}{w}=  \scl{\M_Q\Li w}{u}. 
    \end{equation}
In what follows, for convenience we write
   \begin{equation}\label{def:coeef-M-J}
    \M := \M_Q,\quad \J:=\M^{-1}.
   \end{equation}
\end{remark}
We summarize below the asymptotic behaviour of the coefficient functions, which will be used repeatedly in the subsequent estimates.
\begin{lemma}[Admissibility of the functions]\label{lemma:asymp-coeficients}
    The steady state $Q$ and the related coefficient functions satisfy the following admissibility properties, with the exact corresponding admissible orders listed below.
    \begin{equation}\label{relation:admissibility-Q}
        \begin{gathered}
            Q\sim(0,-2),\quad\Lambda Q\sim(0,-\gamma-2)\\
            \T_Q\sim(2,2-\gamma),\quad\M_Q\sim(2,4),
        \end{gathered}
    \end{equation}
    and
    \begin{equation}\label{relation:admissibility-V}
        \begin{gathered}
            \left(yV^1+\frac{d+1}{y}\right)\sim(1,-1),\quad V^2\sim(0,-2)\\
            V^A\sim(0,-2),\quad y^2V^B\sim(0,0).
        \end{gathered}
    \end{equation}
    Furthermore we can compute the asymptotic behaviour
    \begin{equation}\label{property:decay-V}
        \begin{aligned}
            &V^A(y)=\left\{
    \begin{aligned}
         & c_1'+O(y^2), && \text{as $y\rightarrow0$}\\
       & -\frac{\gamma+2}{y^2}+O(y^{-4}), && \text{as $y\rightarrow+\infty$},
    \end{aligned}
    \right. \\
    &V^B(y)=\left\{
    \begin{aligned}
         & \frac{d+1}{y^2}+O(1), && \text{as $y\rightarrow0$}\\
       & \frac{d+1-\gamma}{y^2}+O(y^{-4}), && \text{as $y\rightarrow+\infty$}.
    \end{aligned}
    \right.
        \end{aligned}
    \end{equation}
\end{lemma}
\begin{proof}
    The exact admissible orders of $Q$ and $\Lambda Q$ follow directly from Lemma~\ref{lemma:asymptotic-Q}. Integrating \eqref{def:TQ-equation} and using the fact that $\gamma>2$, we obtain
    \begin{equation}\label{equation:expression-infinity-M}
        \ln\left(y^{d-1}\M_Q\right)=\ln\left(\frac{y^{d-1}\T_Q}{\Lambda Q}\right)=(d+3)\ln y+\ln\left(\frac{\T_Q(1)}{\Lambda Q(1)}\right)+O\left(\frac{1}{y^2}\right),\quad y\rightarrow+\infty
    \end{equation}
    and
    \begin{equation}
        \ln\left(\frac{\T_Q}{y^2\Lambda Q}\right)=\ln\left(\frac{\T_Q(1)}{\Lambda Q(1)}\right)-\int_y^1xQdx,\quad y\rightarrow0.
    \end{equation}
    This yields the asymptotic behavior of $\T_Q$ and $\M_Q$. To show the asymptotic behaviour of $V^1,V^2,V^A$ and $V^B$, we obtain from \eqref{def:factorize-L} and \eqref{defeq:linear-operator-in-V} 
    \begin{equation}
        \begin{gathered}
            V^1=-\frac{d+1}{y^2}-Q,\quad V^2=-(y\partial_y Q+2dQ)\\
            V^A=\frac{\partial_y\Lambda Q}{y\Lambda Q},\quad V^B=\frac{\partial_y(y^{d-1}\T_Q)}{y^d\T_Q}.
        \end{gathered}
    \end{equation}
    Therefore the asymptotic behaviour follows from \eqref{relation:admissibility-Q}.
\end{proof}
\begin{remark}
    Since $\J=\M^{-1}$ and the asymptotic expression of $\M$ in \eqref{equation:expression-infinity-M}, we obtain $\M=cy^4+O(y^2)$ as $y\rightarrow+\infty$ with a constant $c$. Therefore the following cancellation holds as $y\rightarrow+\infty$:
    \begin{equation}\label{equation:cancellation-J}
    \frac{\Lambda \J+2\J}{\J}=O\left(\frac{1}{1+y^2}\right).
    \end{equation}
    This degeneracy will allow us to measure in a sharp way the size of tails at infinity, and will be used in our energy estimation.
\end{remark}
From \eqref{def:operator-A-B}, we see that the kernels of $\A$ and $\B^*$ are explicit:
\begin{equation}\label{relation:kernel-of-A-B}
    \begin{gathered}
        \A w=0,\quad\text{if and only if}\quad w\in\text{Span}(\Lambda Q),\\
        \B^* w=0,\quad\text{if and only if}\quad w\in\text{Span}\left(\frac{1}{y^{d-1}\T_Q}\right).
    \end{gathered}
\end{equation}
Thus we are able to compute $\Li^{-1}$ in an elementary two-step procedure by
inverting $\A$ and $\B^*$. In particular, this leads to the following.
\begin{lemma}[Inversion of $\Li$]\label{lemma:inversion-L}
    Let $f$ be a $C^{\infty}$ radially symmetric function. We define $w$ by
    \begin{equation}\label{relation:inverse-f}
         w :=-\Lambda Q\int_0^y\frac{\A_w(x)}{\Lambda Q(x)}dx,\quad\A_w :=\frac{1}{y^{d-1}\T_Q}\int_0^yf(x)\T_Q(x)x^{d-1}dx,
    \end{equation}
    Then it follows $\Li w=f$, and we define
    \begin{equation}
        w:=\Li^{-1}f.
    \end{equation}
\end{lemma}
We remark that ``the inverse of $\Li$" is not canonically defined, as $\Li$ has nontrivial kernels. However, here by convention we use $\Li^{-1}$ to represent the procedure \eqref{relation:inverse-f}.
\begin{remark}
    Given these properties we can compute the kernels of $\Li$. We can write two kernels as
    \begin{equation}\label{relation:kernel-Li}
        \Li w=0,\quad\text{if and only if}\quad w\in\text{Span}(\Lambda Q,\Gamma_Q),
    \end{equation}
    where we have
    \begin{equation}
        \A(\Gamma_Q)=\frac{1}{y^{d-1}\T_Q},\quad \Gamma_Q(y)=\Lambda Q\int_y^{+\infty}\frac{dx}{x^{d-1}\T_Q\Lambda Q}.
    \end{equation}
    Hence $\Gamma_Q$ admits the following asymptotic behavior,
    \begin{equation}
        \Gamma_Q(y)=\left\{
    \begin{aligned}
         & \frac{c_{0,\Gamma_Q}}{y^d}(1+O(y^2)), && \text{as $y\rightarrow0$}\\
       & c_{1,\Gamma_Q}y^{-d+\gamma}(1+O(y^{-2})), && \text{as $y\rightarrow+\infty$}.
    \end{aligned}
    \right. 
    \end{equation}
\end{remark}
We note $\Li$ naturally acts on the class of admissible functions in the following way:
\begin{lemma}[Action of $\Li$ and $\Li^{-1}$ on admissible functions]\label{lemma:act-L-admissible}
    Let $f$ be an admissible function of degree $(p_1,p_2)$. Then it follows: 
    \begin{enumerate}
        \item [(i)] $\Lambda f$ is admissible of degree $(p_1,p_2)$.
        \item [(ii)] If $p_1$ is even, then $\Li f$ is admissible of degree $(\max\{0,p_1-2\},p_2-2)$.
        \item [(iii)] If $p_2>-\gamma-4$, then $\Li^{-1}f$ is admissible of degree $(p_1+2,p_2+2)$.
    \end{enumerate}
\end{lemma}
\begin{proof}
    Part (i) and (ii) are direct consequences of Definition~\ref{definition:admisible-function}. Part (iii) can be shown by simply computing over the integration formula of $\Li^{-1}$ \eqref{relation:inverse-f}.
\end{proof}
Then we construct the series $T_k$ and present its properties.
\begin{lemma}[Construction of the kernel of $\Li^k$]\label{lemma:construction-Tk}
    Consider the profiles
    \begin{equation}\label{defeq:definition-Tk}
        T_k=(-1)^k\Li^{-k}\Lambda Q,\quad k\in\N.
    \end{equation}
    Then it holds:
    \begin{enumerate}
        \item [(i)] $T_k$ is admissible of degree $(2k,2k-2-\gamma)$.
        \item [(ii)] $\Lambda T_k-(2k-\gamma)T_k$ is admissible of degree $(2k,2k-4-\gamma)$.
    \end{enumerate}
\end{lemma}
\begin{proof}
    Part (i) is straightforward from the admissibility of $\Lambda Q$ in \eqref{relation:admissibility-Q} and Lemma~\ref{lemma:act-L-admissible}.

    For part (ii) we prove it by induction. For $k=1$, we explicitly compute using Lemma~\ref{lemma:inversion-L} to get
    \begin{equation}
        \partial_y^iT_1(y)=c_{T,i}y^{-\gamma-i}+O(y^{-2-\gamma-i}),\quad i\in\N,\quad y\rightarrow+\infty.
    \end{equation}
    Therefore we proceed by noticing $\partial_y^i\Lambda f=\Lambda\partial_y^if+i\partial_y^if$,
    \begin{equation}
        \partial_y^i[\Lambda T_1-(2-\gamma) T_1]=\Lambda\partial_y^iT_1-(2-\gamma-i)\partial_y^iT_1=O(y^{-\gamma-i-2}).
    \end{equation}
    Hence we deduce that $\Lambda T_1-(2-\gamma)T_1$ is admissible of degree $(2,-\gamma-2)$. We now assume the claim holds for $k\geq1$, that is $\Lambda T_k-(2k-\gamma)T_k$ is admissible of degree $(2k,2k-4-\gamma)$. We shall use the following identity:
\begin{equation}
\Li(\Lambda\eps)=\Lambda(\Li\eps)+2\Li\eps-y(\Lambda V^1)\partial_y\eps-(\Lambda V^2)\eps,
\end{equation}
which follows from a direct computation. We now use definition \eqref{defeq:definition-Tk} to write
\begin{equation}\label{equation:induction-relation-lamTk}
   \begin{aligned}
       \Li(\Lambda T_{k+1}-(2k+2-\gamma)T_{k+1})=-\Lambda T_k+(2k-\gamma)T_k-y(\Lambda V^1)\partial_yT_{k+1}-(\Lambda V^2)T_{k+1}.
   \end{aligned}
\end{equation}
We notice that $T_{k+1}$ is admissible of degree $(2k+2,2k-\gamma)$. Then from \eqref{relation:admissibility-V} we derive that
\begin{equation}
    -y(\Lambda V^1)\partial_yT_{k+1}-(\Lambda V^2)T_{k+1}=O\left(y^{2k+2}\right),\quad y\rightarrow0,
\end{equation}
and
\begin{equation}
    -y(\Lambda V^1)\partial_yT_{k+1}-(\Lambda V^2)T_{k+1}=O\left(y^{2k-2\gamma-2}\right),\quad y\rightarrow+\infty.
\end{equation}
Together with the induction hypothesis, we deduce that the right-hand side of \eqref{equation:induction-relation-lamTk} is admissible of degree $(2k,2k-4-\gamma)$. The conclusion then follows by part (iii) of Lemma~\ref{lemma:act-L-admissible}.
\end{proof}

\subsection{Construction of blowup profile} In this subsection we use the structure of the linear operators to construct the approximate blowup profile. To begin, we first introduce the concept of homogeneous admissible function.
\begin{definition}[Homogeneous admissible function]\label{def:homogeneous-admissible}
    For $j=(j_1,\dots,j_L)\in\N^L$, we say that a function $f(b,y)$ with $b=(b_1,\dots,b_L)$ is homogeneous of degree $(p_1,p_2,p_3)\in\N\times\R\times\N$ if it is a finite linear combination of monomials
    \begin{equation}
        \tilde{f}(y)\prod_{k=1}^Lb_k^{j_k},
    \end{equation}
    where $\tilde{f}(y)$ is admissible of degree $(p_1,p_2)$ and 
    \begin{equation}
        \sum_{k=1}^Lkj_k=p_3.
    \end{equation}
    We set
    \begin{equation}
        \deg(f):=(p_1,p_2,p_3).
    \end{equation}
\end{definition}
We are now able to construct our approximate blowup profile.
\begin{proposition}[Approximate profile]\label{proposition:approximate-profile}
    Let $L\gg 1$ be an integer and $M>0$ be a large enough universal constant. Then there exists a small enough universal constant $b^*(M,L)$ such that the following holds true. For a $C^1$ map
    \begin{equation}
        b=(b_1,\dots,b_L):[s_0,s_1]\mapsto[-b^*,b^*]^L,
    \end{equation}
    with a priori bounds
    \begin{equation}
        0<b_1<b^*,\quad |b_k|\lesssim b_1^k,\quad 2\leq k\leq L.
    \end{equation}
    Then there exists a profile
    \begin{equation}\label{equation:construct-Q}
        Q_{b(s)}(y)=Q(y)+\sum_{k=1}^Lb_k(s)T_k(y)+\sum_{k=2}^{L+2}S_k(b,y)=Q(y)+\Theta_{b(s)}(y)
    \end{equation}
    satisfying the following properties:
    \begin{enumerate}
        \item [(i)] Approximate equation. The constructed profile satisfies
        \begin{equation}\label{equation:profile-equation-approximate}
            \partial_sQ_b-\partial_{yy}Q_b-\frac{d+1}{y}\partial_yQ_b-Q_b(y\partial_yQ_b+dQ_b)+b_1\Lambda Q_b=\Psi_b+\textnormal{Mod}(t)
        \end{equation}
        with the modulation defined as
        \begin{equation}\label{equation:Mod-term-1}
            \textnormal{Mod}(t)=\sum_{k=1}^L[(b_k)_s+(2k-\gamma)b_1b_k-b_{k+1}]\left[T_k+\sum_{j=k+1}^{L+2}\frac{\partial S_j}{\partial b_k}\right],
        \end{equation}
        where we use the convention $b_j=0$ for $j\geq L+1$.
        \item [(ii)] Estimate on the profiles. We have $S_1=0$ and the profiles $S_k$ for $2\leq k\leq L+2$ are homogeneous with
        \begin{equation}
            \begin{gathered}
                \deg(S_k)=(2k,2k-4-\gamma,k)\\
                \frac{\partial S_k}{\partial b_j}=0,\quad 2\leq k\leq j\leq L.
            \end{gathered}
        \end{equation}
        \item [(iii)] Estimate on the error term $\Psi_b$. For all $0\leq k\leq L$, we have the following estimate by recalling the definition of $h,\delta,m,\eta,B_1$ in \eqref{defeq:h-delta}, \eqref{defeq:m-def} and \eqref{defeq:B0-B1}:
        \begin{equation}\label{estimate:psib-bound}
            \int_{y\leq2B_1}\M|\Li^{h+k+1}\Psi_b|^2+\int_{y\leq 2B_1}\frac{\M|\Psi_b|^2}{1+y^{4(h+k+1)}}\lesssim b_1^{2k+4+2(1-\delta)-C_L\eta}
        \end{equation}
        with a constant $C_L$ depends only on $L$
        .We also have the improved local bound
        \begin{equation}\label{estimate:psib-local-bound}
            \int_{y\leq2M}\M|\Li^{h+k+1}\Psi_b|^2\lesssim M^Cb_1^{2L+6}.
        \end{equation}
    \end{enumerate}
\end{proposition}
\begin{proof}
    We proceed in three steps. First we rewrite the formula of the error term.

    \step{1} Expansion of $\Psi_b$. We derive from \eqref{equation:construct-Q}, \eqref{equation:profile-equation-approximate} and \eqref{equation:Mod-term-1}
    \begin{align}
         &\partial_sQ_b-\partial_{yy}Q_b-\frac{d+1}{y}\partial_yQ_b-Q_b(y\partial_yQ_b+dQ_b)+b_1\Lambda Q_b\\
            &=b_1\Lambda Q+\partial_s\Theta_b-\partial_{yy}\Theta_b-(\frac{d+1}{y}+yQ)\partial_y\Theta_b-(y\partial_yQ+2dQ)\Theta_b+b_1\Lambda \Theta_b\\
            &\quad -[Q_b(y\partial_yQ_b+dQ_b)-Q(y\partial_yQ+dQ)-yQ\partial_y\Theta_b-(y\partial_yQ+2dQ)\Theta_b]\\
            &:= A_1+A_2.
    \end{align}
    Using the expression of $\Theta_b$ \eqref{equation:construct-Q} and the definition of $T_k$ in \eqref{defeq:definition-Tk}, we obtain
    \begin{align}
        A_1&=b_1\Lambda Q+\sum_{k=1}^L[(b_k)_sT_k+b_k\Li T_k+b_1b_k\Lambda T_k]+\sum_{k=2}^{L+2}[\partial_sS_k+\Li S_k+b_1\Lambda S_k]\\ &=\sum_{k=1}^L[(b_k)_sT_k-b_{k+1}T_k+b_1b_k\Lambda T_k]+\sum_{k=2}^{L+2}[\partial_sS_k+\Li S_k+b_1\Lambda S_k]\\ &=\sum_{k=1}^L[(b_k)_s-b_{k+1}+(2k-\gamma)b_1b_k]T_k+\sum_{k=1}^L[\Li S_{k+1}+\partial_sS_k+b_1b_k[\Lambda T_k-(2k-\gamma)T_k]+b_1\Lambda S_k]\\
    &\quad+[\mathcal{L}S_{L+2}+\partial_sS_{L+1}+b_1\Lambda S_{L+1}]+[\partial_s S_{L+2}+b_1\Lambda S_{L+2}].
    \end{align}
    Since we have
    \begin{equation}
        \partial_sS_k=\sum_{j=1}^L(b_j)_s\frac{\partial S_k}{\partial b_j}=\sum_{j=1}^L[(b_j)_s+(2j-\gamma)b_1b_j-b_{j+1}]\frac{\partial S_k}{\partial b_j}-\sum_{j=1}^L[(2j-\gamma)b_1b_j-b_{j+1}]\frac{\partial S_k}{\partial b_j},
    \end{equation}
    we can write
    \begin{equation}
        A_1=\text{Mod}(t)+\sum_{k=1}^{L+1}[\Li S_{k+1}+E_k]+E_{L+2},
    \end{equation}
    where the terms $E_k$ are defined as for $k=1,\dots,L$,
    \begin{equation}\label{defeq:ek-in-a1}
         E_k=b_1b_k[\Lambda T_k-(2k-\gamma)T_k]+b_1\Lambda S_k-\sum_{j=1}^{k-1}[(2j-\gamma)b_1b_j-b_{j+1}]\frac{\partial S_k}{\partial b_j},
    \end{equation}
    for $k=L+1,L+2$,
    \begin{equation}
         E_k=b_1\Lambda S_k-\sum_{j=1}^{L}[(2j-\gamma)b_1b_j-b_{j+1}]\frac{\partial S_k}{\partial b_j}.
    \end{equation}
    To compute the nonlinear term $A_2$, we first observe
    \begin{equation}\label{defeq:A2}
        A_2=-\Theta_by\partial_y\Theta_b-d\Theta_b^2.
    \end{equation}
    Therefore we can write
    \begin{equation}
         A_2=\sum_{i=2}^{L+2}P_i+R,
    \end{equation}
    where
    \begin{equation}\label{defeq:Pi-in-A2}
        \begin{aligned}
            P_i=&\sum_{j=1}^{i-1}c_{1,i,j}b_jb_{i-j}T_jT_{i-j}+c_{2,i,j}b_jT_jS_{i-j}+c_{3,i,j}S_jS_{i-j}\\
    &+\sum_{j=1}^{i-1}c_{4,i,j}b_jb_{i-j}T_jy\partial_yT_{i-j}+c_{5,i,j}b_j(T_jy\partial_yS_{i-j}+y\partial_yT_jS_{i-j})+c_{6,i,j}S_jy\partial_yS_{i-j}
        \end{aligned}
    \end{equation}
    and
    \begin{equation}\label{defeq:R-in-A2}
        \begin{aligned}
            R=&\sum_{i=L+3}^{2L+4}\left(\sum_{j=1}^{i-1}c_{1,i,j}b_jb_{i-j}T_jT_{i-j}+c_{2,i,j}b_jT_jS_{i-j}+c_{3,i,j}S_jS_{i-j}\right)\\
    &+\sum_{i=L+3}^{2L+4}\left(\sum_{j=1}^{i-1}c_{4,i,j}b_jb_{i-j}T_jy\partial_yT_{i-j}+c_{5,i,j}b_j(T_jy\partial_yS_{i-j}+y\partial_yT_jS_{i-j})+c_{6,i,j}S_jy\partial_yS_{i-j}\right).
        \end{aligned}
    \end{equation}
    Here we set $T_j=0$ for $j\geq L+1$ and $S_j=0$ for $j\geq L+3$. We mention here that all terms involving powers of $b$ greater than or equal to $L+3$ are collected in $R$, which therefore constitutes a perturbative higher-order remainder. In conclusion, we have
    \begin{equation}\label{equation:rewrite-psib}
        \Psi_b=\sum_{k=1}^{L+1}\left[\mathcal{L}S_{k+1}+E_k+P_{k+1}\right]+E_{L+2}+R.
    \end{equation}

    \step{2} Construction of $S_k$. From \eqref{equation:rewrite-psib} we construct the profile $S_k$ as
    \begin{equation}\label{defeq:construct-inverse-Sk}
        S_1=0,\quad S_k=-\Li^{-1}F_k,\quad2\leq k\leq L+2,
    \end{equation}
    where
    \begin{equation}
        F_k=E_{k-1}+P_{k},\quad2\leq k\leq L+2.
    \end{equation}
    If $F_k$ is homogeneous with
    \begin{equation}\label{relation:degree-fk}
        \deg(F_k)=(2k-2,2k-6-\gamma,k),\quad 2\leq k\leq L+2
    \end{equation}
    and
    \begin{equation}\label{relation:smaller-fk}
        \frac{\partial F_k}{\partial b_j}=0,\quad 2\leq k\leq j\leq L+2,
    \end{equation}
    then it follows from Lemma~\ref{lemma:act-L-admissible} that $S_k$ is homogeneous of degree $(2k,2k-4-\gamma,k)$ for $2\leq k\leq L+2$. This leads to part (ii) of Proposition~\ref{proposition:approximate-profile}. We now give the proof of \eqref{relation:degree-fk} and \eqref{relation:smaller-fk}. Let us prove by induction on $k$.

    For $k=2$, we compute from \eqref{defeq:ek-in-a1} and \eqref{defeq:Pi-in-A2} to get
    \begin{equation}
        F_2 = E_1+P_2=b_1^2(\Lambda T_1-(2-\gamma)T_1)+b_1^2(c_{1,1,1}T_1^2+c_{4,1,1}T_1y\partial_y T_1).
    \end{equation}
    Then from Lemma~\ref{lemma:construction-Tk} we obtain $\Lambda T_1-(2-\gamma)T_1$ is admissible of degree $(2,-2-\gamma)$ and $c_{1,1,1}T_1^2+c_{4,1,1}T_1y\partial_y T_1$ is admissible of degree $(4,-2\gamma)$. Hence we can deduce that $F_2$ is homogeneous with degree $\deg(F_2)=(2,-2-\gamma,2)$.

    Now we suppose the claim holds for $k$ and we aim to prove the case for $k+1$. Using again \eqref{defeq:ek-in-a1} and \eqref{defeq:Pi-in-A2} we have for $k=1,\dots,L$,
    \begin{equation}
        F_{k+1}=E_k+P_{k+1}=b_1b_k[\Lambda T_k-(2k-\gamma)T_k]+b_1\Lambda S_k-\sum_{j=1}^{k-1}[(2j-\gamma)b_1b_j-b_{j+1}]\frac{\partial S_k}{\partial b_j}+P_{k+1}.
    \end{equation}
    From Lemma~\ref{lemma:construction-Tk} and induction hypotheses we obtain
    \begin{equation}
        \deg\left(b_1b_k[\Lambda T_k-(2k-\gamma)T_k]+b_1\Lambda S_k\right)=(2k,2k-4-\gamma,k+1).
    \end{equation}
    The induction hypotheses also shows $\partial S_k/\partial b_j$ is homogeneous of degree $(2k,2k-4-\gamma,k-j)$, which leads to
    \begin{equation}
        \deg\left(b_1b_k[\Lambda T_k-(2k-\gamma)T_k]+b_1\Lambda S_k-\sum_{j=1}^{k-1}[(2j-\gamma)b_1b_j-b_{j+1}]\frac{\partial S_k}{\partial b_j}\right)=(2k,2k-4-\gamma,k+1).
    \end{equation}
    For $P_{k+1}$ we recall the definition in \eqref{defeq:Pi-in-A2}. Using again Lemma~\ref{lemma:construction-Tk} and induction hypotheses, we can get the terms
    \begin{equation}
        T_jT_{k+1-j}, T_jS_{k+1-j},S_jS_{k+1-j},T_jy\partial_yT_{k+1-j},T_jy\partial_yS_{k+1-j},y\partial_yT_jS_{k+1-j},S_jy\partial_yS_{k+1-j}
    \end{equation}
    are admissible of degree $(2k+2,2k-2-2\gamma)$. Therefore
    \begin{equation}
        \deg(P_{k+1})=(2k+2,2k-2-2\gamma,k+1),
    \end{equation}
    and this concludes the case for $k+1$ when $k=1,\dots,L$. For $k=L+1,L+2$, we can follow the same procedure and conclude the proof of part (ii).

    \step{3} We now give the local bounds for error term $\Psi_b$. From \eqref{equation:rewrite-psib} and \eqref{defeq:construct-inverse-Sk} we can rewrite $\Psi_b$ as
    \begin{equation}
        \Psi_b=E_{L+2}+R.
    \end{equation}
    From step 2 we derive $E_{L+2}$ is homogeneous of degree $(2L+4,2L-\gamma,L+3)$. Then we use Lemma~\ref{lemma:construction-Tk}, part (ii) of Proposition~\ref{proposition:approximate-profile} and the definition of $R$ \eqref{defeq:R-in-A2} to show that $R$ has leading degree $(2L+6,2L+2-2\gamma,L+3)$. Therefore we obtain
    \begin{equation}
        \deg(\Psi_b)=(2L+4,2L-\gamma,L+3),
    \end{equation}
    and this leads to the following estimate for all $0\leq k\leq L$
    \begin{align}
        \int_{y\leq2B_1}\mathcal{M}|\mathcal{L}^{h+k+1}\Psi_b|^2&\lesssim b_1^{2L+6}\int_{y\leq2B_1}|y^{2L-\gamma-2(h+k+1)}|^2y^{d-1+4}dy\\
        &\lesssim b_1^{2L+6}\int_{y\leq2B_1}y^{4(L-k+\delta)-1}dy\\
        &\lesssim b_1^{(2L+6)-2(L-k+\delta)(1+\eta)}\lesssim b_1^{2k+4+2(1-\delta)-C_L\eta},
    \end{align}
    where we use Lemma~\ref{lemma:act-L-admissible} and the relation $d-2\gamma-4h=4\delta$. Similarly, the control of the term $\M|\Psi_b|^2/(1+y^{4(h+k+1)})$ is obtained with the exact same lines as above, which proves \eqref{estimate:psib-bound}. Finally, the local bound \eqref{estimate:psib-local-bound} follows exactly from the homogeneous degree of $\Psi_b$.
\end{proof}
In the following proposition we localize the profile $Q_b$ to avoid the growth of tails in the region $y\geq2B_1\gg B_0$.
\begin{proposition}[Localization of profile]\label{proposition:localize-Qb}
    We assume the a priori bound
    \begin{equation}\label{relation:assumption-b1s}
        |(b_1)_s|\lesssim b_1^2
    \end{equation}
    and all assumptions in Proposition~\ref{proposition:approximate-profile}. Then we construct the localized profile
    \begin{equation}\label{equation:construct-Q-local}
        \tQ_{b(s)}(y)=Q(y)+\sum_{k=1}^Lb_k\tT_k+\sum_{k=2}^{L+2}\tS_k,\quad \tT_k=\chi_{B_1}T_k,\quad\tS_k=\chi_{B_1}S_k        
    \end{equation}
    which satisfies the equation
    \begin{equation}\label{equation:profile-equation-approximate-local}
        \partial_s\tilde{Q}_b-\partial_{yy}\tilde{Q}_b-\frac{d+1}{y}\partial_y\tilde{Q}_b-\tilde{Q}_b(y\partial_y\tilde{Q}_b+d\tilde{Q}_b)+b_1\Lambda \tilde{Q}_b=\tilde{\Psi}_b+\chi_{B_1}\textnormal{Mod}(t).
    \end{equation}
    Then the error term $\tPsi_b$ satisfies the following bounds:
    \begin{enumerate}
        \item [(i)] Sobolev bounds. For all $0\leq k\leq L-1$,
        \begin{equation}\label{estimate:cut-psib-bound-1}
        \begin{aligned}
            \int \M|\Li^{h+k+1}\tPsi_b|^2+\int\frac{\M|\A\Li^{h+k}\tPsi_b|^2}{1+y^2}+\int\frac{\M|\Li^{h+k}\tPsi_b|^2}{1+y^4}+\int\frac{\M|\tPsi_b|^2}{1+y^{4(h+k+1)}}\\
            \lesssim b_1^{2k+2+2(1-\delta)-C_L\eta},
        \end{aligned}
        \end{equation}
        and
        \begin{equation}\label{estimate:cut-psib-bound-2}
            \begin{aligned}
            \int \M|\Li^{h+L+1}\tPsi_b|^2+\int\frac{\M|\A\Li^{h+L}\tPsi_b|^2}{1+y^2}+\int\frac{\M|\Li^{h+L}\tPsi_b|^2}{1+y^4}+\int\frac{\M|\tPsi_b|^2}{1+y^{4(h+L+1)}}\\
            \lesssim b_1^{2L+2+2(1-\delta)(1+\eta)}.
        \end{aligned}
        \end{equation}

        \item [(ii)] Local bounds. For $M\leq B_1$ and $0\leq k\leq L$ we have
        \begin{equation}\label{estimate:cut-psib-bound-local-1}
            \int_{y\leq 2M}\M|\Li^{h+k+1}\tPsi_b|^2\lesssim M^Cb_1^{2L+6},
        \end{equation}
        and
        \begin{equation}\label{estimate:cut-psib-bound-local-2}
            \int_{y\leq2B_0}\M|\Li^{h+k+1}\tPsi_b|^2+\int_{y\leq2B_0}\frac{\M|\tPsi_b|^2}{1+y^{4(h+k+1)}}\lesssim b_1^{2k+4+2(1-\delta)-C_L\eta}.
        \end{equation}
    \end{enumerate}
\end{proposition}
\begin{proof}
    We first compute from \eqref{equation:construct-Q-local}
    \begin{align}
        \partial_s\tilde{Q}_b&-\partial_{yy}\tilde{Q}_b-\frac{d+1}{y}\partial_y\tilde{Q}_b-\tilde{Q}_b(y\partial_y\tilde{Q}_b+d\tilde{Q}_b)+b_1\Lambda \tilde{Q}_b\\
        &=\chi_{B_1}\left[\partial_sQ_b-\partial_{yy}Q_b-\frac{d+1}{y}\partial_yQ_b-Q_b(y\partial_yQ_b+dQ_b)+b_1\Lambda Q_b\right]\\
        &\quad+\Theta_b\left[\partial_s\chi_{B_1}-\left(\partial_{yy}\chi_{B_1}+\frac{d+1}{y}\partial_y\chi_{B_1}\right)+b_1\Lambda\chi_{B_1}\right]-2\partial_y\chi_{B_1}\partial_y\Theta_b+b_1(1-\chi_{B_1})\Lambda Q\\
        &\quad-\left[\tQ_b(y\partial_y\tQ_b+d\tQ_b)-Q(y\partial_yQ+dQ)-\chi_{B_1}(Q_b(y\partial_yQ_b+dQ_b)-Q(y\partial_yQ+dQ))\right].
    \end{align}
    Recalling \eqref{equation:profile-equation-approximate}, we can write
    \begin{equation}
        \tPsi_b=\chi_{B_1}\Psi_b+\hPsi_b,
    \end{equation}
    where
    \begin{align}
        \hPsi_b&=\underbrace{b_1(1-\chi_{B_1})\Lambda Q}_{\hPsi_b^1}\\
        &\quad-\underbrace{\left[\tQ_b(y\partial_y\tQ_b+d\tQ_b)-Q(y\partial_yQ+dQ)-\chi_{B_1}(Q_b(y\partial_yQ_b+dQ_b)-Q(y\partial_yQ+dQ))\right]}_{\hPsi_b^2}.\\
        &\quad +\underbrace{\Theta_b\left[\partial_s\chi_{B_1}-\left(\partial_{yy}\chi_{B_1}+\frac{d+1}{y}\partial_y\chi_{B_1}\right)+b_1\Lambda\chi_{B_1}\right]-2\partial_y\chi_{B_1}\partial_y\Theta_b}_{\hPsi_b^3}.
    \end{align} 
    For the estimate  of the term $\chi_{B_1}\Psi_b$, we follow the exact same steps as in the proof of \eqref{estimate:psib-bound} and \eqref{estimate:psib-local-bound}. This leads to the bounds in \eqref{estimate:cut-psib-bound-1}, \eqref{estimate:cut-psib-bound-2}, \eqref{estimate:cut-psib-bound-local-1} and \eqref{estimate:cut-psib-bound-local-2}. Since all terms are supported in the region $B_1\leq y\leq2B_1$ except $\hPsi_b^1$, the local bounds \eqref{estimate:cut-psib-bound-local-1} and \eqref{estimate:cut-psib-bound-local-2} follow from \eqref{estimate:psib-bound} and \eqref{estimate:psib-local-bound}. We are therefore left to estimate the term $\hPsi_b$. Using the asymptotic behavior of $(1-\chi_{B_1})\Lambda Q$ from Lemma~\ref{lemma:asymptotic-Q} and Lemma \ref{lemma:act-L-admissible}, we obtain
    \begin{equation}
        \int\M|\Li^{h+k+1}\hPsi_b^1|^2\lesssim \int_{y\geq B_1}\frac{y^4y^{d-1}}{y^{4(h+k+1)+2(\gamma+2)}}dy\lesssim b_1^{2k+2+2(1-\delta)(1+\eta)+2k\eta}
    \end{equation}
    for all $0\leq k\leq L$, where we use the identity $d-2\gamma-4h=4\delta$ and the definition of $B_1$ in \eqref{defeq:B0-B1}. For the nonlinear term $\hPsi_b^2$, we can write
    \begin{equation}\label{equation:derive-cut-error-nonlinear-rewrite}
    \begin{aligned}
        &Q_b(y\partial_yQ_b+dQ_b)-Q(y\partial_yQ+dQ)\\
        &=\Theta_b(y\partial_y\Theta_b+d\Theta_b)+yQ\partial_y\Theta_b+(y\partial_yQ+2dQ)\Theta_b.
    \end{aligned}
    \end{equation}
    Note again $\hPsi_b^2$ is supported in the region $B_1\leq y\leq2B_1$, thus we only need to estimate \eqref{equation:derive-cut-error-nonlinear-rewrite} on this support. By recalling the degree of admissibility of $T_k$ and $S_k$ in Lemma~\ref{lemma:construction-Tk} and Proposition~\ref{proposition:approximate-profile}, we estimate for $y\geq B_1$ and $\forall j\in\N$,
    \begin{equation}
        |\partial_y^j\Theta_b|\lesssim \sum_{k=1}^Lb_1^ky^{2k-2-\gamma-j}\mathbf{1}_{y\geq B_1}.
    \end{equation}
    Therefore from \eqref{equation:derive-cut-error-nonlinear-rewrite} and the degree of admissibility of $Q$ in Lemma~\ref{lemma:asymptotic-Q}, we obtain
    \begin{equation}
        |\partial_y^j\hPsi_b^2|\lesssim \sum_{k=1}^Lb_1^ky^{2k-4-\gamma-j}\mathbf{1}_{B_1\leq y\leq2B_1}\lesssim\frac{b_1}{y^{\gamma+j+2}}\sum_{k=1}^Lb_1^{-(k-1)\eta}\mathbf{1}_{B_1\leq y\leq2B_1}.
    \end{equation}
    This leads to the following bound for all $0\leq k\leq L$,
    \begin{equation}
        \begin{aligned}
            \int\M|\Li^{h+k+1}\hPsi_b^2|^2&\lesssim b_1^2\sum_{j=1}^Lb_1^{-2(j-1)\eta}\int_{B_1\leq y\leq2B_1}\frac{y^4y^{d-1}}{y^{4(h+k+1)+4+2\gamma}}dy\\
            &\lesssim b_1^{2k+2+2(1-\delta)(1+\eta)}\sum_{j=1}^Lb_1^{(2k-2j+2)\eta}.
        \end{aligned}
    \end{equation}
    To control $\hPsi_b^3$, we first observe
    \begin{equation}
        |\partial_s\chi_{B_1}|\lesssim\frac{|(b_1)_s|}{b_1}\frac{y}{B_1}\mathbf{1}_{B_1\leq y\leq2B_1}\lesssim b_1\mathbf{1}_{B_1\leq y\leq2B_1},
    \end{equation}
    where we use the assumption \eqref{relation:assumption-b1s}. This yields for all $0\leq k\leq L$,
    \begin{equation}
        \begin{aligned}
            \int\M|\Li^{h+k+1}\hPsi_b^3|^2&\lesssim b_1^2\sum_{j=1}^Lb_1^{2j}\int_{B_1\leq y\leq2B_1}\frac{y^4y^{d-1}}{y^{4(h+k+1)+4+2\gamma-4j}}dy\\
            &\lesssim b_1^{2k+2+2(1-\delta)(1+\eta)}\sum_{j=1}^Lb_1^{(2k-2j)\eta}.
        \end{aligned}
    \end{equation}
    Gathering the bounds above, we get
    \begin{equation}
        \int\M|\Li^{h+k+1}\hPsi_b|^2\lesssim b_1^{2k+2+2(1-\delta)(1+\eta)}\sum_{j=1}^Lb_1^{(2k-2j)\eta}\lesssim b_1^{2k+2+2(1-\delta)(1+\eta)+2\eta(k-L)}.
    \end{equation}
    Finally, the bound of the terms
    \begin{equation}
       \frac{\M|\A\Li^{h+k}\tPsi_b|^2}{1+y^2},\quad\frac{\M|\Li^{h+k}\tPsi_b|^2}{1+y^4},\quad\frac{\M|\tPsi_b|^2}{1+y^{4(h+k+1)}}
    \end{equation}
    is obtained along the exact same lines as above. This concludes the proof.
\end{proof}

\subsection{Dynamical system} From the construction of the approximate profile, we can obtain the dynamic system for $b=(b_1,\dots,b_L)$ by setting the $\text{Mod}(t)$ to zero:
\begin{equation}\label{equation:dynamic-system-b}
    (b_k)_s+(2k-\gamma)b_1b_k-b_{k+1}=0,\quad1\leq k\leq L.
\end{equation}
The behavior of this system determines the corresponding
blowup rate. Systems of this form also arise in a number of related
blowup problems and have been studied extensively; see, for example,
\cite{MerleRaphaelRodnianski2015,GhoulIbrahimNguyen2019}. We now recall the result from
\cite{GhoulIbrahimNguyen2019}.
\begin{lemma}[Solution of the dynamic system]
    Let $\frac{1}{2}\gamma<l\ll L$, $l\in\N$ and the sequence
    \begin{equation}\label{defeq:ck-in-dynamic-system}
    \begin{gathered}
        c_1=\frac{l}{2l-\gamma},\quad c_{k+1}=-\frac{\gamma(l-k)}{2l-\gamma}c_k,\quad 1\leq k\leq l-1,\\
        c_{k+1}=0,\quad k\geq l.
    \end{gathered}
    \end{equation}
    Then the explicit choice
    \begin{equation}\label{defeq:bek-in-dynamic-system}
        b_k^e(s)=\frac{c_k}{s^k},\quad 1\leq k\leq L
    \end{equation}
    forms an exact solution to \eqref{equation:dynamic-system-b}. Moreover, if we linearize around this solution as
    \begin{equation}\label{equation:perturb-bk}
        b_k(s)=b^e_k(s)+\frac{\U_k(s)}{s^k}
    \end{equation}
    with the notation $\U=(\U_1,\dots,\U_l)^T$. Then for $1\leq k\leq l-1$, we have
    \begin{equation}\label{equation:linearized-dynamic-system}
        \begin{gathered}
            (b_k)_s+(2k-\gamma)b_1b_k-b_{k+1}=\frac{1}{s^{k+1}}\left[s(\U_k)_s-(A_l\U)_k+O(|\U|^2)\right],\\
            (b_l)_s+(2l-\gamma)b_1b_l=\frac{1}{s^{l+1}}\left[s(\U_l)_s-(A_l\U)_l+O(|\U|^2)\right],
        \end{gathered}
    \end{equation}
    where
    \begin{equation}
        A_l=(a_{i,j})_{1\leq i,j\leq l}\quad \text{with}\quad\left\{
    \begin{aligned}
         & a_{1,1}=\frac{\gamma(l-1)}{2l-\gamma}-(2-\gamma)c_1, && \\
       & a_{i,i}=\frac{\gamma(l-i)}{2l-\gamma}, && 2\leq i\leq l,\\
       & a_{i,i+1}=1, && 1\leq i\leq l-1,\\
       &a_{i,1}=-(2i-\gamma)c_i,&&2\leq i\leq l,\\
       &a_{i,j}=0,&& \text{otherwise}.
    \end{aligned}
    \right. 
    \end{equation}
    Thus $A_l$ is diagonalizable:
    \begin{equation}\label{defeq:diag-Al}
        A_l=P_l^{-1}D_lP_l,\quad D_l=\text{diag}\left\{-1,\frac{2\gamma}{2l-\gamma},\frac{3\gamma}{2l-\gamma},\dots,\frac{l\gamma}{2l-\gamma}\right\}.
    \end{equation}
\end{lemma}

\section{The bootstrap argument}\label{sec:bootstrap}
In this section, we set up the bootstrap arguments at the heart of the proof of Theorem~\ref{theorem:main-theorem}. We will first set up an explicit equation of the linearization of the problem. We then specify the class of initial data and introduce the preliminary bootstrap bounds for the resulting dynamical flow. Finally, we derive the modulation equations that will be used in the estimates of the higher-order weighted Sobolev norms in Section~\ref{sec:energy-control}.

\subsection{Setup of the equation}
Let the-space time renormalized variables be
\begin{equation}
     y=\frac{r}{\lambda(t)},\quad s(t)=s_0+\int_0^t\frac{1}{\lambda^2(\tau)}d\tau.
\end{equation}
We use the notation:
\begin{equation}
    f_{\lambda}(t,r)=\frac{1}{\lambda^2}f(s,y)
\end{equation}
which shows
\begin{equation}\label{relation:f-f-lambda}
    \partial_tf_{\lambda}=\frac{1}{\lambda^2}\left[\partial_sf-\frac{\lambda_s}{\lambda}\Lambda f\right]_{\lambda}.
\end{equation}
Recall \eqref{def:change-of-variable-origin}, we have
\begin{equation}
    u(t,r)=v_{\lambda}(s,y).
\end{equation}
Now we decompose the renormalized solution $v$ as
\begin{equation}\label{defeq:decompose-v-u}
    v(s,y)=\left(\tQ_{b(s)}+\eps\right)(s,y),\quad u(t,r)=\frac{1}{\lambda^2(t)}\left(\tQ_{b(t)}+\eps\right)\left(t,\frac{r}{\lambda(t)}\right)
\end{equation}
where $\tQ_{b(s)}$ is constructed in Proposition~\ref{proposition:localize-Qb} and the modulation parameters
\begin{equation}
    \lambda(s)>0,\quad b(s)=(b_1(s),\dots,b_L(s))
\end{equation}
are determined from the $L+1$ orthogonality conditions:
\begin{equation}\label{relation:orthogonality-condition}
    \langle \epsilon,\M\Li^k\Phi_{M}\rangle=0,\quad 0\leq k\leq L.
\end{equation}
Here we choose $\Phi_M$ as
\begin{equation}\label{defeq:PhiM}
    \Phi_M=\sum_{k=0}^L c_{k,M}\mathcal{L}^{k}(\chi_M\Lambda Q),
\end{equation}
with
\begin{equation}\label{defeq:c_k_M}
    c_{0,M}=1,\quad c_{k,M}=(-1)^{k+1}\frac{\sum_{j=0}^{k-1}c_{j,M}\langle \M\mathcal{L}^{j}(\chi_{M}\Lambda Q),T_k\rangle}{\langle\chi_{M}\M\Lambda Q,\Lambda Q\rangle},\quad 1\leq k\leq L.
\end{equation}
This leads to the nondegeneracy
\begin{equation}
    \langle \M\Phi_{M},\Lambda Q\rangle=\langle\chi_M\M\Lambda Q,\Lambda Q\rangle\gtrsim M^{d-2\gamma},
\end{equation}
and the cancellation
\begin{equation}
    \langle \M\Phi_M,T_k\rangle =0.
\end{equation}
In particular, we have
\begin{equation}\label{equation:cancellation-identity}
    \langle\M\Li^iT_k,\Phi_M\rangle=(-1)^k\langle\chi_M\M\Lambda Q,\Lambda Q\rangle\delta_{i,k},\quad 0\leq i,k\leq L.
\end{equation}

Using dynamic rescaling equation \eqref{equation:dynamic-rescaling-eq}, we obtain
\begin{equation}\label{equation:dynamic-eps}
    \partial_s\eps-\frac{\lam_s}{\lambda}\Lambda \eps+\Li\eps=-\tPsi_b-\HMod-\Hi(\epsilon)+\Ni(\epsilon)\equiv \F,
\end{equation}
where we have the small linear term
\begin{equation}\label{defeq:small-linear-H}
    \Hi(\eps)=-\eps(y\partial_y(\tQ_b-Q)+2d(\tQ_b-Q))-(\tQ_b-Q)y\partial_y\eps,
\end{equation}
and the nonlinear term
\begin{equation}\label{defeq:nonlinear-N}
    \mathcal{N}(\epsilon)=\epsilon(y\partial_y\epsilon+d\epsilon),
\end{equation}
and the modulation term is defined as
\begin{equation}\label{equation:def-tMod}
    \HMod=-\left(\frac{\lambda_s}{\lambda}+b_1\right)\Lambda\tQ_b+\chi_{B_1}\text{Mod}.
\end{equation}

We also need to consider the perturbation in the original variables
\begin{equation}\label{defeq:w-eps-original-variable}
     w(t,r)=\epsilon_{\lambda}(s,y),\quad u(t,r)=\frac{1}{\lambda^2(t)}\tQ_{b(t)}\left(\frac{r}{\lambda(t)}\right)+w(t,r).
\end{equation}
Then it follows from \eqref{relation:f-f-lambda} and \eqref{equation:dynamic-eps} that $w$ satisfies
\begin{equation}\label{equation:dynamic-w-origin}
    \partial_t w+\mathcal{L}_{\lambda}w=\frac{1}{\lambda^2}\mathcal{F}_{\lambda},
\end{equation}
with the rescaled linearized operators in the original variables
\begin{equation}\label{defeq:linear-operator-lambda}
    \begin{gathered}
        \Li_{\lam}=\B^*_{\lam}\A_{\lam}=-\partial_{rr}+rV^1_{\lam}\partial_r+V^2_{\lam},\\
    \B^*_{\lam}=\partial_r+rV^B_{\lam},\quad\A_{\lam}=-\partial_r+rV^A_{\lam}.
    \end{gathered}
\end{equation}
Here the weighted functions are defined in \eqref{defeq:linear-operator-in-V}.

\subsection{Bootstrap argument}
We use the smooth radial family $u_0=Q_{b_0}+\varepsilon_0$ of Definition~3.1, with
$\varepsilon_0\in C^\infty\cap L^\infty$, $\|\varepsilon_0\|_{\dot H^\sigma}\ll 1$
for $d/2\leq \sigma\leq m$, and finite $\mathcal{E}_{2k}$, $h+1\leq k\leq m$.
The local flow is taken in the infinite-mass class
\[
    \|u\|_X
    :=
    \sup_{r\geq 0}
    \left(
        \langle r\rangle^2 |u(r)|
        +
        \langle r\rangle^3 |u_r(r)|
    \right)
    <\infty.
\]
Indeed, on $\mathbb{R}^{d+2}$, \eqref{equation:1d-keller-segel} reads
$u_t=\Delta u+u x\cdot\nabla u+du^2$.
The heat-kernel bounds with polynomial weights give a contraction for its Duhamel
formula in $C([0,t_1];X)$ (the one-derivative bound is
$O((t-\tau)^{-1/2})$), hence uniqueness and continuous dependence.
Commuting the equation and applying weighted energy estimates preserves the finite
$\mathcal{E}_{2k}$ and gives continuous dependence in these norms.
Lemma~\ref{lemma:inter-bound-origin} and Lemma~\ref{lemma:inter-bound-far} and the bootstrap bound $\|u\|_X$ on every finite rescaled-time
interval; \eqref{estimate:modulation-bound-L-1} keeps $\lambda$ bounded above and away from zero there.
The flow therefore continues while the bootstrap holds. Thus the decomposition defined in \eqref{defeq:decompose-v-u}
\begin{equation}\label{equation:decompose-u-smalltime}
    u(t,r)=\frac{1}{\lambda^2(t)}\left(\tQ_{b(t)}+\eps\right)\left(t,\frac{r}{\lambda(t)}\right)
\end{equation}
is uniquely defined on a small time interval $[0,t_1]$. In fact, we can use the implicit function theorem to show the existence of the decomposition in \eqref{equation:decompose-u-smalltime}. We can compute explicitly
\begin{equation}
    \left.\frac{\partial}{\partial\lambda}\left(\tQ_{b(t)}\right)_{\lam},\frac{\partial}{\partial b_1}\left(\tQ_{b(t)}\right)_{\lam},\dots,\frac{\partial}{\partial b_L}\left(\tQ_{b(t)}\right)_{\lam}\right|_{\lam=1,b=0}=(-\Lambda Q,T_1,\dots,T_L),
\end{equation}
which implies the nondegeneracy of the Jacobian
\begin{equation}
    \left|\left\langle \frac{\partial}{\partial(\lam,b_j)}\left(\tQ_{b(t)}\right)_{\lam},\M\Li^i\Phi_M \right\rangle_{1\leq j\leq L,0\leq i\leq L}\right|_{\lam=1,b=0}=\left|\langle\chi_M\M\Lambda Q,\Lambda Q\rangle\right|^{L+1}>0.
\end{equation}
We then want to show that the perturbation $\eps$ remains small in the energy topology. In particular, we will consider the following weighted energy norms:
\begin{equation}\label{defeq:weighted energy}
    \E_{2k}=\int\M|\Li^k\eps|^2,\quad h+1\leq k\leq m.
\end{equation}
We now describe our choice of the set of initial data.
\begin{definition}[Description of the open set of initial data]\label{definition:initial-data}
    Define
    \begin{equation}\label{defeq:definition-Vl}
        \V=P_l\U
    \end{equation}
    where $\U_k=s^kb_k-c_k,\;1\leq k\leq l$ with $c_k$ defined in \eqref{defeq:ck-in-dynamic-system}, and $P_l$ refers to the diagonalization of $A_l$ in \eqref{defeq:diag-Al}. We let $\Oi$ be the open set of radial initial data of the form
    \begin{equation}
        u_0(r)=\left(\tQ_{b_0}+\eps_0\right)\left(r\right)
    \end{equation}
    satisfying the orthogonality condition \eqref{relation:orthogonality-condition} with the initial rescaled time $s_0\geq1$. That is, at the time $s_0$ we have
    \begin{equation}
        u(s_0,r)=u_0,\quad b(s_0)=b_0
    \end{equation}
    and the normalization
    \begin{equation}
        \lam(s_0)=1.
    \end{equation}
    Moreover, the initial data satisfy the following
    \begin{enumerate}
        \item [(i)] Smallness of the initial perturbation of the modes. For the unstable modes
        \begin{equation}
            |s_0^{\frac{\eta}{2}(1-\delta)}\V_k(s_0)|<1,\quad 2\leq k\leq l,
        \end{equation}
        and for the stable modes
        \begin{equation}
            |s_0^{\frac{\eta}{2}(1-\delta)}\V_1(s_0)|<1,\quad |b_k(s_0)|<s_0^{-\frac{100l(2k-\gamma)}{2l-\gamma}},\quad l+1\leq k\leq L.
        \end{equation}
        \item [(ii)] Smallness of the initial data.
        \begin{equation}
            \sum_{k=h+1}^{m}\E_{2k}(s_0)<s_0^{-\frac{100Ll}{2l-\gamma}}.
        \end{equation}
    \end{enumerate}
\end{definition}
Under the initial smallness condition, we can now assume the bootstrap bounds on $[s_0,\infty)$.
\begin{definition}[Bootstrap assumption]\label{definition:bootstrap-assumption}
    For a large constant $K\geq1$, we assume the bootstrap for $s\geq s_0$
    \begin{align}
    &|\V_k(s)|
    \leq 100s^{-\frac{\eta}{2}(1-\delta)},
    && 1\leq k\leq l,
    \\
    &|b_k(s)|
    \leq s^{-k},
    && l+1\leq k\leq L,
    \\
    &\E_{2m}(s)
    \leq K s^{-(2L+2(1-\delta)(1+\eta))},
    \\
    &\E_{2k}(s)
    \leq
    \left\{
    \begin{aligned}
        &Ks^{-\frac{l}{2l-\gamma}(4k-d)},\\
        &s^{-2(k-h-1)-2(1-\delta)+K\eta},
    \end{aligned}
    \right.
    &&
    \begin{aligned}
        &h+1\leq k\leq l+h,\\
        &l+h+1\leq k\leq m-1.
    \end{aligned}\label{defeq:bootstrap-assumption}
\end{align}
\end{definition}
\begin{remark}
    For a fixed time $s$, we can define $\Si_K(s)$ as the set of all $(b_1(s),\dots,b_L(s),\eps(s))$ such that all bounds in \eqref{defeq:bootstrap-assumption} hold.
\end{remark}
\begin{remark}
    The rest of this section, together with Section~\ref{sec:energy-control}, is devoted to deriving the main dynamical estimates needed to close the bootstrap argument. As is customary in a bootstrap scheme, unless otherwise stated, we assume throughout these sections that the bootstrap assumptions hold.
\end{remark}

\subsection{Modulation equations} We begin by deriving the modulation bounds for parameters $(b,\lam)$ driven by the orthogonality \eqref{relation:orthogonality-condition}. We will also repeatedly use the coercivity of the operator $\Li$ in Lemma~\ref{lemma:coercive-bound-eps}. We first present the estimates for $\Phi_M$ defined in \eqref{defeq:PhiM}.
\begin{lemma}[Estimates for $\Phi_M$]\label{lemma:estimate-PhiM}
    We have the following estimates
    \begin{equation}
        \begin{gathered}
            |c_{k,M}|\lesssim M^{2k},\quad 0\leq k\leq L\\
        \int\M|\Phi_M|^2\lesssim M^{d-2\gamma},\quad \int\M|\Li\Phi_M|^2\lesssim M^{d-2\gamma-4}.
        \end{gathered}
    \end{equation}
\end{lemma}
\begin{proof}
    We prove by induction on $k$. Since $c_{0,M}=1$, the case holds for $k=0$. Assume the bound holds for $k$, now we consider $k+1$. We recall the definition \eqref{defeq:c_k_M} to compute
    \begin{equation}
        \begin{aligned}
            |c_{k+1,M}|&\lesssim\frac{1}{M^{d-2\gamma}}\sum_{j=0}^k M^{2j}\int|\chi_M\M\Lambda Q\Li^j(T_{k+1})|\\
            &\lesssim\frac{1}{M^{d-2\gamma}}\sum_{j=0}^k M^{2j}\int_{y\leq M}\frac{y^{d-1}y^4}{y^{\gamma+2}}y^{2k-2j-\gamma}dy\lesssim M^{2(k+1)},
        \end{aligned}
    \end{equation}
    where we use the degree of admissibility of $\Lambda Q, T_k$ in Lemma~\ref{lemma:asymp-coeficients} and Lemma~\ref{lemma:construction-Tk}. Therefore we have
    \begin{equation}
        \int\M|\Phi_M|^2\lesssim\int\M|\chi_M\Lambda Q|^2+\sum_{j=1}^L|c_{j,M}|^2\int\M|\Li^j(\chi_M\Lambda Q)|^2\lesssim M^{d-2\gamma},
    \end{equation}
    and
    \begin{equation}
        \int\M|\Li\Phi_M|^2\lesssim\sum_{j=0}^L|c_{j,M}|^2\int\M|\Li^{j+1}(\chi_M\Lambda Q)|^2\lesssim M^{d-2\gamma-4}.
    \end{equation}
\end{proof}
Now we compute the modulation equations.
\begin{lemma}[Control of the modulation parameters]\label{lemma:modulation-bound-rough}
    There following bounds hold:
    \begin{equation}\label{estimate:modulation-bound-L-1}
        \sum_{k=1}^{L-1}|(b_k)_s+(2k-\gamma)b_1b_k-b_{k+1}|+\left|b_1+\frac{\lambda_s}{\lambda}\right|\lesssim b_1^{L+1+(1-\delta)(1+\eta)},
    \end{equation}
    and
    \begin{equation}\label{estimate:modulation-bound-L-rough}
        |(b_L)_s+(2L-\gamma)b_1b_L|\lesssim \frac{\sqrt{\mathcal{E}_{2m}}}{M^{2\delta}}+b_1^{L+1+(1-\delta)(1+\eta)} .
    \end{equation}
\end{lemma}
\begin{proof}
    We first take the inner product of \eqref{equation:dynamic-eps} with $\Li^L\Phi_M$ and use the orthogonality condition \eqref{relation:orthogonality-condition} to obtain
    \begin{equation}
        \langle\M\HMod,\Li^L\Phi_M\rangle=-\langle\M\tilde{\Psi}_b,\Li^L\Phi_M\rangle-\langle\M\Li^{L+1}\eps,\Phi_M\rangle-\left\langle-\frac{\lambda_s}{\lambda}\Lambda\eps+\mathcal{H}\eps-\mathcal{N}\eps,\M\Li^L\Phi_M\right\rangle.
    \end{equation}
    We define
    \begin{equation}
        D(t)=\left|b_1+\frac{\lambda_s}{\lambda}\right|+\sum_{k=1}^L|(b_k)_s+(2k-\gamma)b_1b_k-b_{k+1}|,
    \end{equation}
    with $b_{L+1}=0$. From the definition \eqref{equation:Mod-term-1} and \eqref{equation:def-tMod}, we can compute by recalling the cancellation \eqref{equation:cancellation-identity}
    \begin{equation}
        \langle\M\widehat{\text{Mod}},\Li^L\Phi_M\rangle=(-1)^L\langle\M\Lambda Q,\Phi_M\rangle[(b_L)_s+(2L-\gamma)b_1b_L]+O(M^Cb_1D(t)),
    \end{equation}
    where we use the fact that $\Phi_M$ is compactly supported in $y\leq2M$. Then we estimate the error term from \eqref{estimate:cut-psib-bound-local-1} and Lemma~\ref{lemma:estimate-PhiM} to obtain
    \begin{equation}
        |\langle\M\tilde{\Psi}_b,\Li^L\Phi_M\rangle|\lesssim \left(\int_{y\leq2M}\M|\Li^{L}\tPsi_b|^2\right)^{\frac{1}{2}}\left(\int_{y\leq2M}\M|\Phi_M|^2\right)^{\frac{1}{2}}\lesssim M^Cb_1^{L+3}.
    \end{equation}
    We now apply the coercivity bounds in Lemma~\ref{lemma:coercive-bound-eps} with the bootstrap assumptions to estimate
    \begin{equation}
        |\langle\M\Li^{L+1}\eps,\Phi_M\rangle|\lesssim M^{2h}\left(\int\M\frac{|\Li^{L+1}\eps|^2}{1+y^{4h}}\right)^{\frac{1}{2}}\left(\int\M|\Phi_M|^2\right)^{\frac{1}{2}}\lesssim M^{2h+\frac{d}{2}-\gamma}\sqrt{\E_{2m}},
    \end{equation}
    and
    \begin{equation}
        \left|\left\langle-\frac{\lambda_s}{\lambda}\Lambda\eps+\mathcal{H}\eps-\mathcal{N}\eps,\M\Li^L\Phi_M\right\rangle\right|\lesssim M^Cb_1(\E_{2m}+\sqrt{\E_{2m}}+D(t))\lesssim M^Cb_1(\sqrt{\E_{2m}}+D(t)).
    \end{equation}
    Gathering the above bounds we obtain
    \begin{equation}\label{estimate:proof-modulation-bL}
        |(b_L)_s+(2L-\gamma)b_1b_L|\lesssim \frac{\sqrt{\mathcal{E}_{2m}}}{M^{2\delta}}+b_1^{L+1+(1-\delta)(1+\eta)}+M^Cb_1D(t).
    \end{equation}
    
    To proceed the terms $b_k$ for $1\leq k\leq L-1$, we take the inner product of \eqref{equation:dynamic-eps} with $\Li^k\Phi_M$ and use the orthogonality condition \eqref{relation:orthogonality-condition} to obtain the following for $1\leq k\leq L-1$,
    \begin{equation}
        \langle\M\HMod,\Li^k\Phi_M\rangle=-\langle\M\tilde{\Psi}_b,\Li^k\Phi_M\rangle-\left\langle-\frac{\lambda_s}{\lambda}\Lambda\eps+\mathcal{H}\eps-\mathcal{N}\eps,\M\Li^k\Phi_M\right\rangle.
    \end{equation}
    Estimating by following the steps for $b_L$, we have
    \begin{equation}\label{estimate:proof-modulation-bk}
        |(b_k)_s+(2k-\gamma)b_1b_k-b_{k+1}|\lesssim b_1^{L+1+(1-\delta)(1+\eta)}+M^Cb_1(\sqrt{\E_{2m}}+D(t)).
    \end{equation}
    Similarly we take the inner product of \eqref{equation:dynamic-eps} with $\Phi_M$ to imply
    \begin{equation}\label{estimate:proof-modulation-lam}
        \left|b_1+\frac{\lambda_s}{\lambda}\right|\lesssim b_1^{L+1+(1-\delta)(1+\eta)}+M^Cb_1(\sqrt{\E_{2m}}+D(t)).
    \end{equation}
    Taking the sum of \eqref{estimate:proof-modulation-bL}, \eqref{estimate:proof-modulation-bk} and \eqref{estimate:proof-modulation-lam} we get
    \begin{equation}
        D(t)\lesssim M^C\sqrt{\E_{2m}}+b_1^{L+1+(1-\delta)(1+\eta)},
    \end{equation}
    which concludes the proof.
\end{proof}
The bounds for $b_L$ above shows
\begin{equation}
      |(b_L)_s+(2L-\gamma)b_1b_L|\lesssim  b_1^{L+(1-\delta)(1+\eta)}
\end{equation}
from the bootstrap assumptions. This is not good enough to close the expected one
\begin{equation}
    |(b_L)_s+(2L-\gamma)b_1b_L|\ll b_1^{L+1}.
\end{equation}
This implies a term which oscillates in time should be considered to improve the bound. In particular we claim the following lemma.
\begin{lemma}[Improved control for $b_L$]\label{lemma:modulation-bound-improve}
    We have the following
    \begin{equation}\label{estimate:modulation-bound-improve}
        \left|(b_L)_s+(2L-\gamma)b_1b_L+\frac{d}{ds}\left\{\frac{\langle\M\Li^L\eps,\chi_{B_0}\Lambda Q\rangle}{\langle\M\Lambda Q,\chi_{B_0}\Lambda Q\rangle}\right\}\right|\lesssim\frac{1}{B_0^{2\delta}}[C(M)\sqrt{\E_{2m}}+b_1^{L+1+(1-\delta)-C_L\eta}].
    \end{equation}
\end{lemma}
\begin{proof}
    We first commute \eqref{equation:dynamic-eps} with $\Li^L$ and take the inner product with $\chi_{B_0}\Lambda Q$ to get
    \begin{align}
        &\langle\M\Lambda Q,\chi_{B_0}\Lambda Q\rangle\left\{\frac{d}{ds}\left[\frac{\langle\M\Li^L\eps,\chi_{B_0}\Lambda Q\rangle}{\langle\M\Lambda Q,\chi_{B_0}\Lambda Q\rangle}\right]-\langle\M\Li^L\eps,\chi_{B_0}\Lambda Q\rangle\frac{d}{ds}\left[\frac{1}{\langle\M\Lambda Q,\chi_{B_0}\Lambda Q\rangle}\right]\right\}\\
        &=-\langle\M\Li^{L}\widehat{\text{Mod}},\chi_{B_0}\Lambda Q\rangle+\langle\M\Li^{L}(\mathcal{N}\eps-\mathcal{H}\eps),\chi_{B_0}\Lambda Q\rangle+\frac{\lambda_s}{\lambda}\langle\M\Li^{L}\Lambda\eps,\chi_{B_0}\Lambda Q\rangle\\
        &\quad-\langle\M\Li^{L}\tilde{\Psi}_b,\chi_{B_0}\Lambda Q\rangle+\langle\M\Li^L\eps,\partial_s(\chi_{B_0})\Lambda Q\rangle-\langle\M\Li^{L+1}\eps,\chi_{B_0}\Lambda Q\rangle. \label{equation:proof-refine-bL-innerproduct}
    \end{align}
    The $\HMod$ term in the right-hand side of \eqref{equation:proof-refine-bL-innerproduct} contains the modulation equation for $b_L$. Specifically we write from \eqref{equation:Mod-term-1} and \eqref{equation:def-tMod}
    \begin{align}
        &|\langle\M\Li^{L}\widehat{\text{Mod}},\chi_{B_0}\Lambda Q\rangle-(-1)^L\scl{\M\Lambda Q}{\chi_{B_0}\Lambda Q}[(b_L)_s+(2L-\gamma)b_1b_L]|\\
            &\lesssim\sum_{k=1}^L|(b_k)_s+(2k-\gamma)b_1b_k-b_{k+1}|\left|\scl{\M\sum_{j=k+1}^{L+2}\frac{\partial\tS_j}{\partial b_k}}{\Li^L(\chi_{B_0}\Lambda Q)}\right|\\
            &\quad+\left|\frac{\lam_s}{\lam}+b_1\right|\left|\scl{\M\Lambda \tilde{\Theta}_b}{\Li^L(\chi_{B_0}\Lambda Q)}\right|.\label{equation:proof-tmod-expansion}
    \end{align}
    Here we use the fact that $\Li(\Lambda Q)=0,\Li^LT_k=0$ for $1\leq k\leq L-1$, and $\Li^LT_L=(-1)^L\Lambda Q$.
    Recalling the admissibility of degree in our construction in Proposition~\ref{proposition:approximate-profile}, we derive the bounds for $y\sim B_0$:
    \begin{equation}
        |\Lambda \Theta_b|\lesssim b_1y^{-\gamma},\quad\sum_{j=k+1}^{L+2}\left|\frac{\partial S_j}{\partial b_k}\right|\lesssim\sum_{j=k+1}^{L+2}b_1^{j-k}y^{2(j-2)-\gamma}\lesssim b_1y^{2k-2-\gamma}.
    \end{equation}
    Together with the bounds in Lemma~\ref{lemma:modulation-bound-rough}, we derive the bound for the right-hand side of \eqref{equation:proof-tmod-expansion}
    \begin{align}
        &\sum_{k=1}^L|(b_k)_s+(2k-\gamma)b_1b_k-b_{k+1}|\left|\scl{\M\sum_{j=k+1}^{L+2}\frac{\partial\tS_j}{\partial b_k}}{\Li^L(\chi_{B_0}\Lambda Q)}\right|\\
        &\qquad\qquad+\left|\frac{\lam_s}{\lam}+b_1\right|\left|\scl{\M\Lambda \tilde{\Theta}_b}{\Li^L(\chi_{B_0}\Lambda Q)}\right|\\
        &\lesssim (C(M)\sqrt{\E_{2m}}+b_1^{L+1+(1-\delta)(1+\eta)})b_1\int_{B_0\leq y\leq2B_0}\frac{y^4y^{2L-2-\gamma}y^{d-1}}{y^{2L+2+\gamma}}dy\\
        &\lesssim (C(M)\sqrt{\E_{2m}}+b_1^{L+1+(1-\delta)(1+\eta)})b_1B_0^{d-2\gamma}.
    \end{align}

    Now we turn to the bounds for all remaining terms in \eqref{equation:proof-refine-bL-innerproduct}. We recall from the asymptotic behaviour of $\Lambda Q$ in Lemma~\ref{lemma:asymptotic-Q} that
    \begin{equation}\label{estimate:M-LambdaQ-LambdaQ}
        B_0^{d-2\gamma}\lesssim\scl{\M\Lambda Q}{\chi_{B_0}\Lambda Q}\lesssim B_0^{d-2\gamma}.
    \end{equation}
    For the second term in the left-hand side of \eqref{equation:proof-refine-bL-innerproduct}, we estimate
    \begin{equation}\label{estimate:M-Lieps-LambdaQ}
        |\langle\M\Li^L\eps,\chi_{B_0}\Lambda Q\rangle|\lesssim B_0^{2h+2}|\scl{\M\Lambda Q}{\chi_{B_0}\Lambda Q}|^{\frac{1}{2}}\left(\int\frac{\M|\Li^L\eps|^2}{1+y^{4h+4}}\right)^{\frac{1}{2}}\lesssim B_0^{\frac{d}{2}-\gamma+2h+2}\sqrt{\E_{2m}}
    \end{equation}
    by using the coercivity bounds. Therefore we can write
    \begin{equation}
        \begin{aligned}
            \left|\langle\M\Li^L\eps,\chi_{B_0}\Lambda Q\rangle\frac{d}{ds}\left[\frac{1}{\langle\M\Lambda Q,\chi_{B_0}\Lambda Q\rangle}\right]\right|&\lesssim\frac{|\langle\M\Li^L\eps,\chi_{B_0}\Lambda Q\rangle|}{\langle\M\Lambda Q,\chi_{B_0}\Lambda Q\rangle^2}\left|\frac{(b_1)_s}{b_1}\right|\int_{B_0\leq y\leq2B_0}\M|\Lambda Q|^2\\
        &\lesssim b_1\frac{B_0^{\frac{d}{2}-\gamma+2h+2}\sqrt{\E_{2m}}}{B_0^{2d-4\gamma}}B_0^{d-2\gamma}\lesssim \frac{\sqrt{\E_{2m}}}{B_0^{2\delta}},
        \end{aligned}
    \end{equation}
    where we used the relation $d=4h+4\delta+2\gamma$ and the definition of $B_0$.

    For other terms in the right-hand side of \eqref{equation:proof-refine-bL-innerproduct}, we first use Cauchy-Schwarz and the fact $\Li(\Lambda Q)=0$ to obtain
    \begin{equation}
        \begin{aligned}
            |\langle\M\Li^L\eps,\partial_s(\chi_{B_0})\Lambda Q\rangle|&\lesssim\left|\frac{(b_1)_s}{b_1}\right|\left(\int_{B_0\leq y\leq2B_0}(1+y^{4h+4})\M|\Lambda Q|^2\right)^{\frac{1}{2}}\left(\int\frac{\M|\Li^L\eps|^2}{1+y^{4h+4}}\right)^{\frac{1}{2}}\\
            &\lesssim b_1 B_0^{\frac{d}{2}-\gamma+2h+2}\sqrt{\E_{2m}}\lesssim B_0^{\frac{d}{2}-\gamma+2h}\sqrt{\E_{2m}},
        \end{aligned}
    \end{equation}
    \begin{equation}
        |\langle\M\Li^{L+1}\eps,\chi_{B_0}\Lambda Q\rangle|\lesssim\left(\int(1+y^{4h})\M|\chi_{B_0}\Lambda Q|^2)\right)^{\frac{1}{2}}\left(\int\frac{\M|\Li^{L+1}\eps|^2}{1+y^{4h}}\right)^{\frac{1}{2}}\lesssim B_0^{\frac{d}{2}-\gamma+2h}\sqrt{\E_{2m}}
    \end{equation}
    and
    \begin{align}
        \left|\frac{\lambda_s}{\lambda}\langle\M\Li^{L}\Lambda\eps,\chi_{B_0}\Lambda Q\rangle\right|&\lesssim b_1\left(\int(1+y^{4(L+h)+2})\M|\Li^L(\chi_{B_0}\Lambda Q)|^2\right)^{\frac{1}{2}}\left(\int\frac{\M|\partial_y\eps|^2}{1+y^{4(L+h)+2}}\right)^{\frac{1}{2}}\\
            &\quad+b_1\left(\int(1+y^{4(L+h)+2})\M|\Li^L(\chi_{B_0}\Lambda Q)|^2\right)^{\frac{1}{2}}\left(\int\frac{\M|\eps|^2}{1+y^{4(L+h+1)}}\right)^{\frac{1}{2}}\\
            &\lesssim B_0^{\frac{d}{2}-\gamma+2h}\sqrt{\E_{2m}}.
    \end{align}
    For the error term we estimate by using \eqref{estimate:cut-psib-bound-local-2}:
    \begin{equation}
        \begin{aligned}
        |\langle\M\Li^{L}\tilde{\Psi}_b,\chi_{B_0}\Lambda Q\rangle|&\lesssim\left(\int(1+y^{4(L+h+1)})\M|\Li^L(\chi_{B_0}\Lambda Q)|^2\right)^{\frac{1}{2}}\left(\int\frac{\M|\tPsi_b|^2}{1+y^{4(L+h+1)}}\right)^{\frac{1}{2}}\\
        &\lesssim B_0^{\frac{d}{2}-\gamma+2h+2}b_1^{L+2+(1-\delta)-C_L\eta}.
        \end{aligned}        
    \end{equation}
    Similarly we can estimate
    \begin{align}
        |\langle\M\Li^{L}(\mathcal{N}\eps-\mathcal{H}\eps),\chi_{B_0}\Lambda Q\rangle|&\lesssim\int\M|\Hi\eps\Li^L(\chi_{B_0}\Lambda Q)|+\int\M|\Ni\eps\Li^L(\chi_{B_0}\Lambda Q)|\\
        &\lesssim \left(\int\frac{\M|\Hi\eps|^2}{1+y^{4m-4}}\right)^{\frac{1}{2}}\left(\int(1+y^{4m-4})\M|\Li^L(\chi_{B_0}\Lambda Q)|^2\right)^{\frac{1}{2}}\\
        &\qquad+\left(\int\frac{\M|\Ni\eps|^2}{1+y^{4m}}\right)^{\frac{1}{2}}\left(\int(1+y^{4m})\M|\Li^L(\chi_{B_0}\Lambda Q)|^2\right)^{\frac{1}{2}}\\
        &\lesssim B_0^{\frac{d}{2}-\gamma+2h-1}\sqrt{\E_{2m}}+b_1B_0^2B_0^{\frac{d}{2}-\gamma+2h-1}\sqrt{\E_{2m}}\lesssim B_0^{\frac{d}{2}-\gamma+2h}\sqrt{\E_{2m}}.
    \end{align}
    Dividing \eqref{equation:proof-refine-bL-innerproduct} by
$(-1)^L\scl{\M\Lambda Q}{\chi_{B_0}\Lambda Q}$
and collecting the above estimates, we obtain
\eqref{estimate:modulation-bound-improve}.
\end{proof}

\section{Energy control}\label{sec:energy-control}
In this section, we derive the monotonicity formula needed to close the bootstrap argument. 
This energy estimate is the core of our analysis and also the most technical part of the paper. 
Our approach is inspired by the ideas developed in \cite{Raphael2014,RaphaelSchweyer2014,MerleRaphaelRodnianski2015,GhoulIbrahimNguyen2019}. 
The main technical difficulty lies in the treatment of the nonlinear terms, which possess additional structure and require the most subtle estimates in our computation.
\begin{proposition}[Lyapunov monotonicity of the weighted energy]\label{proposition:energy-control}
    There holds
\begin{equation}\label{equation:energy-control-1}
\begin{aligned}
    &\frac{d}{dt}\brc{\frac{\E_{2m}}{\lam^{4m+2-d}}[1+O(b_1^{\eta(1-\delta)})]}\\
    &\leq\frac{b_1}{\lam^{4m+4-d}}\left[\left(2+\frac{C}{M^{2\delta}}\right)\E_{2m}+C\left(b_1^{L+(1-\delta)(1+\eta)}\sqrt{\E_{2m}}+b_1^{2L+2(1-\delta)(1+\eta)}\right)\right].
\end{aligned}
\end{equation}
Furthermore for all $h+1\leq k\leq m-1$ we have
\begin{equation}\label{equation:energy-control-2}
\begin{aligned}
    &\frac{d}{dt}\brc{\frac{\E_{2k}}{\lam^{4k+2-d}}[1+O(b_1)]}\\
    &\leq\frac{b_1}{\lam^{4k+4-d}}\left[2\E_{2k}+C\left(b_1^{k-h-1+(1-\delta)-C\eta}\sqrt{\E_{2k}}+b_1^{2(k-h-1)+2(1-\delta)-C\eta}\right)\right].
\end{aligned}
\end{equation}
\end{proposition}

\begin{proof}
We only prove \eqref{equation:energy-control-1}, since
\eqref{equation:energy-control-2} follows from the same argument. The only
additional difficulty arises at the top order. Indeed, for
$h+1\leq k\leq m-1$, the rough modulation estimates of
Lemma~\ref{lemma:modulation-bound-rough} are sufficient to control all the
modulation terms. At the top order $k=m$, however, these estimates are too
weak to close the bootstrap argument. We therefore introduce a suitable
time-dependent correction and use the refined modulation estimate from
Lemma~\ref{lemma:modulation-bound-improve}. This produces an additional
contribution $C\E_{2m}/M^{2\delta}$ in \eqref{equation:energy-control-1}, which can be absorbed by choosing $M$
sufficiently large.

The proof of \eqref{equation:energy-control-1} is divided into the following nine steps.

\step{1} Identity for weighted energy. We consider the following derivatives for perturbation $\eps=\eps(y,s)$ satisfying equation \eqref{equation:dynamic-eps} for any $k\in\N$:
    \begin{equation} 
    \begin{gathered}
        \eps_{2k+1}=\A\epsilon_{2k},\quad \eps_{2k+2}=\Li\eps_{2k},\quad w_{2k+1}=\A_{\lam}w_{2k},\quad w_{2k+2}=\Li_{\lam}w_{2k},\\
        \eps_0=\eps,\quad w_0=w,
    \end{gathered}
\end{equation}
here $w=w(r,t)$ is defined in \eqref{defeq:w-eps-original-variable} and the operators $\Li,\Li_{\lam},\A,\A_{\lam}$ are defined in \eqref{def:operator-A-B}, \eqref{def:factorize-L} and \eqref{defeq:linear-operator-lambda}. It follows directly that
\begin{equation}
    \Li=\B^*\A,\quad\Li_{\lam}=\B_{\lam}^*\A_{\lam}.
\end{equation}
To derive the energy identity, we recall the dynamic rescaling equation \eqref{equation:dynamic-w-origin} and commute with $\Li_{\lam}^{m-1}$ and $\A_{\lam}\Li_{\lam}^{m-1}$ to obtain
\begin{align}
\label{equation:commute-energy}
\begin{gathered}
\partial_tw_{2m}=-\mathcal{L}_{\lambda}w_{2m}+\mathcal{L}_{\lambda}^m\left(\frac{1}{\lambda^2}\mathcal{F}_{\lambda}\right)+[\partial_t,\mathcal{L}_{\lambda}^m]w,\\
\partial_tw_{2m-1}=-\mathcal{A}_{\lambda}w_{2m}+\mathcal{A_{\lambda}}\mathcal{L}^{m-1}_{\lambda}\left(\frac{1}{\lambda^2}\mathcal{F}_{\lambda}\right)+[\partial_t,\mathcal{A}_{\lambda}\mathcal{L}^{m-1}_{\lambda}]w.
\end{gathered}
\end{align}
Now we note the scaling
\begin{equation}
    \int\J^{-1}_{\lam}|w_{2m}|^2=\frac{1}{\lam^{4m+2-d}}\int\J^{-1}|\eps_{2m}|^2=\frac{1}{\lam^{4m+2-d}}\E_{2m}
\end{equation}
with the definition of $\J$ in \eqref{def:coeef-M-J}.
Hence it is sufficient to compute the energy identity for $w_{2m}$ and in particular we have:
\begin{align}
    \frac{1}{2}\frac{d}{dt}(\J_{\lam}^{-1}w_{2m},w_{2m})=&(\partial_t w_{2m},\J_{\lam}^{-1}w_{2m})+\frac{1}{2}\int\partial_t(\J_{\lam}^{-1})w_{2m}^2\\
        =&-\int\J_{\lam}^{-1}|\A_{\lam}w_{2m}|^2-\frac{1}{2}\int\frac{b_1(\Lambda\J)_{\lam}}{\lam^2\J_{\lam}^{2}}w_{2m}^2+\frac{1}{2}\int\frac{(\Lambda\J)_{\lam}}{\lam^2\J_{\lam}^{2}}\left(\frac{\lam_s}{\lam}+b_1\right)w_{2m}^2\\
        &+(\J_{\lam}^{-1}w_{2m},\frac{1}{\lam^2}\Li_{\lam}^m\F_{\lam})+(\J_{\lam}^{-1}w_{2m},[\partial_t,\Li_{\lam}^m]w).\label{equation:derive-energy}
\end{align}
In \eqref{equation:derive-energy} the second quadratic term has no definite sign and has the worst bound. The key technical point is to introduce a suitable oscillation in time and to exploit the dissipative structure. Recalling \eqref{equation:commute-energy} we proceed as follows:
\begin{align}
    \frac{d}{dt}\left(\int\frac{b_1(\Lambda\J)_{\lam}}{\lam^2\J_{\lam}^{2}}w_{2m}w_{2m-2}\right)=&-\int\frac{b_1(\Lambda\J)_{\lam}}{\lam^2\J_{\lam}^{2}}w_{2m}^2+\int\frac{b_1(\Lambda\J)_{\lam}}{\lam^2\J_{\lam}^{2}}w_{2m}\left[\Li_{\lam}^{m-1}\left(\frac{1}{\lam^2}\F_{\lam}\right)+[\partial_t,\Li^{m-1}_{\lam}]w\right]\\
    &+\int\frac{b_1(\Lambda\J)_{\lam}}{\lam^2\J_{\lam}^{2}}w_{2m-2}\left[\Li_{\lam}^{m}\left(\frac{1}{\lam^2}\F_{\lam}\right)+[\partial_t,\Li^{m}_{\lam}]w\right]\\
    &-\int\frac{b_1(\Lambda\J)_{\lam}}{\lam^2\J_{\lam}^{2}}(\Li_{\lam}w_{2m})w_{2m-2}+\int w_{2m}w_{2m-2}\frac{d}{dt}\left(\frac{b_1(\Lambda\J)_{\lam}}{\lam^2\J_{\lam}^2}\right),\label{estimate:time-oscillate-1}
\end{align}
and the following holds
\begin{equation}\label{estimate:time-oscillate-mid}
    \begin{aligned}
    &-\int\frac{b_1(\Lambda\J)_{\lam}}{\lam^2\J_{\lam}^{2}}(\Li_{\lam}w_{2m})w_{2m-2}\\
    &=\frac{2b_1}{\lam^2}(\J_{\lambda}^{-1}w_{2m},w_{2m})-\frac{b_1}{\lambda^2}\int \mathcal{J}_{\lambda}^{-1}\mathcal{A}_{\lambda}w_{2m}\mathcal{A}_{\lambda}\left(\frac{(\Lambda \mathcal{J})_{\lambda}+2\mathcal{J}_{\lambda}}{\mathcal{J}_{\lambda}}w_{2m-2}\right).
\end{aligned}
\end{equation}
We also compute
\begin{equation}\label{estimate:time-oscillate-2}
    \begin{aligned}
    \frac{d}{dt}\left(\int\frac{b_1}{\lam^2\J_{\lambda}}w_{2m-1}^2\right)=&\int w_{2m-1}^2\frac{d}{dt}\left(\frac{b_1}{\lam^2\J_{\lam}}\right)-2\frac{b_1}{\lam^2}\int \J_{\lam}^{-1}w_{2m-1}\A_{\lam}w_{2m}\\
    &+\frac{2b_1}{\lam^2}\int \J_{\lam}^{-1}w_{2m-1}\left[\A_{\lam}\Li^{m-1}_{\lam}\left(\frac{1}{\lam^2}\F_{\lam}\right)+[\partial_t,\A_{\lam}\Li^{m-1}_{\lam}]w\right].
\end{aligned}
\end{equation}
Now we compute the leading order quadratic terms in \eqref{estimate:time-oscillate-1} and \eqref{estimate:time-oscillate-2} as
\begin{align}
    \frac{d}{dt}\left(\frac{b_1(\Lambda\J)_{\lam}}{\lam^2\J_{\lam}^2}\right)&=\frac{d}{dt}\left(b_1\frac{\Lambda \J}{\J^2}(\frac{r}{\lam(t)})\right)\\
    &=\frac{1}{\lam^2}\left[
    b_{1s}\frac{\Lambda \J}{\J^2}-b_1\frac{\lam_s}{\lam}y\partial_y\left(\frac{\Lambda \J}{\J^2}\right)\right](y)\\
    &=\frac{1}{\lam^2\J}\left[-2
    b_{1s}+8b_1\frac{\lam_s}{\lam}+O\left(\frac{|b_{1s}|+b_1|\frac{\lam_s}{\lam}|}{1+y^2}\right)\right],\label{estimate:dt-leading-1}
\end{align}
and
\begin{equation}\label{estimate:dt-leading-2}
    \begin{aligned}
        \frac{d}{dt}\left(\frac{b_1}{\lam^2\J_{\lam}}\right)&=\frac{d}{dt}\left(\frac{b_1}{\J(r/\lam)}\right)=\frac{1}{\lam^2}\left[\frac{b_{1s}}{\J}+b_1\frac{\frac{\lam_s}{\lam}y\partial_y\J}{\J^2}\right]\\
    &=\frac{1}{\lam^2\J}\left[
    b_{1s}-4b_1\frac{\lam_s}{\lam}+O\left(\frac{|b_{1s}|+b_1|\frac{\lam_s}{\lam}|}{1+y^2}\right)\right].
    \end{aligned}
\end{equation}
Combining \eqref{estimate:time-oscillate-1},\eqref{estimate:time-oscillate-mid},\eqref{estimate:time-oscillate-2},\eqref{estimate:dt-leading-1} and \eqref{estimate:dt-leading-2}, together with the fact that the operators commute in the weighted norm:
\begin{equation}\label{estimate:commutative-property-energy}
    \int\J_{\lam}^{-1}w_{2m}w_{2m-2}=\int\J_{\lam}^{-1}w_{2m-1}^2,\quad\int\J_{\lam}^{-1}w_{2m}^2=\int\J_{\lam}^{-1}w_{2m-1}\A_{\lam}w_{2m},
\end{equation}
we obtain the following energy identity:
\begin{align}
    &\frac{1}{2}\frac{d}{dt}\left[(\J_{\lam}^{-1}w_{2m},w_{2m})-\int\frac{b_1(\Lambda\J)_{\lam}}{\lam^2\J_{\lam}^2}w_{2m}w_{2m-2}-2\int\frac{b_1}{\lam^2\J_{\lam}}w_{2m-1}^2\right]\\
         &=-\int\J_{\lam}^{-1}|\A_{\lam}w_{2m}|^2+\frac{b_1}{\lam^2}(\J_{\lam}^{-1}w_{2m},w_{2m})+\frac{1}{2}\int\frac{(\Lambda\J)_{\lam}}{\lam^2\J_{\lam}^{2}}\left(\frac{\lam_s}{\lam}+b_1\right)w_{2m}^2\\
         &\quad+\frac{1}{2}\frac{b_1}{\lam^2}\int \J_{\lam}^{-1}\A_{\lam}w_{2m}\A_{\lam}\left(\frac{(\Lambda \J)_{\lam}+2\J_{\lam}}{\J_{\lam}}w_{2m-2}\right)+(\J_{\lam}^{-1}w_{2m},\frac{1}{\lam^2}\Li_{\lam}^m\F_{\lam})+(\J_{\lam}^{-1}w_{2m},[\partial_t,\Li_{\lam}^m]w)\\
         &\quad-\frac{1}{2}\int\frac{b_1(\Lambda\J)_{\lam}}{\lam^2\J_{\lam}^2}w_{2m}\left[\Li_{\lam}^{m-1}\left(\frac{1}{\lam^2}\F_{\lam}\right)+[\partial_t,\Li_{\lam}^{m-1}]w\right]\\
         &\quad-\frac{1}{2}\int\frac{b_1(\Lambda\J)_{\lam}}{\lam^2\J_{\lam}^2}w_{2m-2}\left[\Li_{\lam}^{m}\left(\frac{1}{\lam^2}\F_{\lam}\right)+[\partial_t,\Li_{\lam}^{m}]w\right]\\
         &\quad-2\frac{b_1}{\lam^2}\int \J_{\lam}^{-1}w_{2m-1}\left[\A_{\lam}\Li^{m-1}_{\lam}\left(\frac{1}{\lam^2}\F_{\lam}\right)+[\partial_t,\A_{\lam}\Li^{m-1}_{\lam}]w\right]\\
         &\quad+O\left(\int\frac{|\eps_{2m}\eps_{2m-
    2}|+|\eps_{2m-1}|^2}{\lam^{4m+4-d}\J(1+y^2)}(|(b_1)_s|+b_1|\frac{\lam_s}{\lam}|)\right).\label{equation:energy-identity-final}
\end{align}
In the following steps, we estimate all the terms in \eqref{equation:energy-identity-final}. Throughout the argument, we repeatedly use the coercivity estimate from Lemma~\ref{lemma:coercive-bound-eps} together with the improved decay \eqref{equation:cancellation-J}.

\step{2} Quadratic terms with dissipation. Recalling the modulation estimates \eqref{estimate:modulation-bound-L-1} and the coercivity bound in  Lemma~\ref{lemma:coercive-bound-eps}, we obtain
\begin{equation}
    \begin{aligned}
        &\int\frac{|\eps_{2m}\eps_{2m-
    2}|+|\eps_{2m-1}|^2}{\lam^{4m+4-d}\J(1+y^2)}(|b_{1s}|+b_1|\frac{\lam_s}{\lam}|)\\
    \lesssim& \frac{b_1^2}{\lam^{4m+4-d}}\left(\left(\int\M\eps_{2m}^2\right)^{\frac{1}{2}}\left(\int\M\frac{\eps_{2m-2}^2}{1+y^4}\right)^{\frac{1}{2}}+\int\M\frac{\eps_{2m-1}^2}{1+y^2}\right)\lesssim\frac{b_1^2}{\lam^{4m+4-d}}\E_{2m}.
    \end{aligned}
\end{equation}
Then the degeneracy \eqref{equation:cancellation-J} gives the estimate
\begin{align}
    &\frac{b_1}{\lam^2}\left|\int \J_{\lam}^{-1}\A_{\lam}w_{2m}\A_{\lam}\left(\frac{(\Lambda \J)_{\lam}+2\J_{\lam}}{\J_{\lam}}w_{2m-2}\right)\right|\\
        \leq&\frac{1}{100}\int\J_{\lam}^{-1}|\A_{\lam}w_{2m}|^2+C\frac{b_1^2}{\lam^4}\int\J_{\lam}^{-1}\left|\A_{\lam}\left(\frac{(\Lambda \J)_{\lam}+2\J_{\lam}}{\J_{\lam}}w_{2m-2}\right)\right|^2\\
        \leq&\frac{1}{100}\int\J_{\lam}^{-1}|\A_{\lam}w_{2m}|^2+C\frac{b_1^2}{\lam^{4m+4-d}}\int\M\left|\A\left(\frac{\Lambda \J+2\J}{\J}\eps_{2m-2}\right)\right|^2\\
        \leq&\frac{1}{100}\int\J_{\lam}^{-1}|\A_{\lam}w_{2m}|^2+C\frac{b_1^2}{\lam^{4m+4-d}}\int\M\left(\frac{1}{1+y^4}\eps_{2m-1}^2+\frac{1}{(1+y^6)}\eps_{2m-2}^2\right)\\
        \leq&\frac{1}{100}\int\J_{\lam}^{-1}|\A_{\lam}w_{2m}|^2+C\frac{b_1^2}{\lam^{4m+4-d}}\E_{2m}.
\end{align}
Here we use the Leibniz rule in Lemma~\ref{lemma:lebniz-rule} by choosing $\phi=(\Lambda\J+2\J)/\J$. Therefore the desired bound follows from our choice of initial data, which ensures that
\begin{equation}
    b_1(s) \ll \frac{1}{M^{2\delta}} .
\end{equation}

\step{3} Quadratic terms with commutators. We first claim the following bound:
\begin{equation}\label{estimate:bound-on-commutator-general}
    \begin{gathered}
        \lam^2\int\J_{\lam}^{-1}(1+y^4)|[\partial_t,\Li^m_{\lam}]w|^2+\int\J_{\lam}^{-1}(1+y^2)|[\partial_t,\A_{\lam}\Li^{m-1}_{\lam}]w|^2+\frac{1}{\lam^2}\int\J_{\lam}^{-1}|[\partial_t,\Li^{m-1}_{\lam}]w|^2\\
        \lesssim\frac{b_1^2}{\lam^{4m+4-d}}\E_{2m}.
    \end{gathered}
\end{equation}
We only deal with the first term in \eqref{estimate:bound-on-commutator-general} since the second and the third term are estimated similarly. By induction and the definition \eqref{defeq:linear-operator-lambda} \eqref{equation:commute-energy}, we obtain
\begin{equation}
    [\partial_t,\Li_{\lam}^m]w=\sum_{k=0}^{m-1}\Li_{\lam}^k([\partial_t,\Li_{\lam}]\Li_{\lam}^{m-1-k}w)=\sum_{k=0}^{m-1}\Li_{\lam}^k\left(\left(r\partial_tV^1_{\lam}\partial_r+\partial_tV^2_{\lam}\right)\Li^{m-1-k}_{\lam}w\right).
\end{equation}
Then direct computation shows
\begin{equation}
    \partial_tV_{\lam}^1(r)=-\frac{\lam_s}{\lam}\frac{\Lambda V^1(y)
    }{\lam^4},\quad \partial_tV_{\lam}^2(r)=-\frac{\lam_s}{\lam}\frac{\Lambda V^2(y)
    }{\lam^4}.
\end{equation}
We now proceed by making the change of variables and using the modulation estimate \eqref{estimate:modulation-bound-L-1}. Therefore we have
\begin{equation}
    \begin{aligned}
        \lam^2\int\J_{\lam}^{-1}(1+y^4)|[\partial_t,\Li^m_{\lam}]w|^2&=\left|\frac{\lam_s}{\lam}\right|^2\frac{1}{\lam^{4m+4-d}}\int\M(1+y^4)\left|\sum_{k=0}^{m-1}\Li^k(y\Lambda V^1\partial_y+\Lambda V^2)\Li^{m-1-k}\eps\right|^2\\
        &\lesssim\frac{b_1^2}{\lam^{4m+4-d}}\sum_{k=0}^{m-1}\int\M(1+y^4)\left|\Li^k(y\Lambda V^1\partial_y+\Lambda V^2)\Li^{m-1-k}\eps\right|^2.
    \end{aligned}
\end{equation}
We recall \eqref{property:decay-V} together with the asymptotic behavior in Lemma~\ref{lemma:asymp-coeficients}. This leads to the following estimate
\begin{equation}
    |\Lambda V^1(y)|+|\Lambda V^2(y)|\lesssim\frac{1}{1+y^4}.
\end{equation}
Using the Leibniz rule again in Lemma~\ref{lemma:lebniz-rule} and coercivity Lemma~\ref{lemma:coercive-bound-eps}, we get
\begin{equation}
    \begin{aligned}
        \sum_{k=0}^{m-1}\int\M(1+y^4)\left|\Li^k(y\Lambda V^1\partial_y+\Lambda V^2)\Li^{m-1-k}\eps\right|^2\lesssim\sum_{k=0}^{2m-1}\int\M\frac{\eps_{2m-1-k}^2}{1+y^{2k+2}}\lesssim\E_{2m}.
    \end{aligned}
\end{equation}
And this completes the proof of \eqref{estimate:bound-on-commutator-general}.

With \eqref{estimate:bound-on-commutator-general} we are able to estimate the following term as
\begin{align}
    \left|\int\frac{b_1(\Lambda\J)_{\lam}}{\lam^2\J_{\lam}^2}w_{2m-2}[\partial_t,\Li_{\lam}^{m}]w\right|&\lesssim\lam^2\int\J_{\lam}^{-1}(1+y^4)|[\partial_t,\Li^m_{\lam}]w|^2+\frac{b_1^2}{\lam^6}\int\frac{(\Lambda\J)_{\lam}^2}{\J_{\lam}^3(1+y^4)}w_{2m-2}^2\\
    &\lesssim\frac{b_1^2}{\lam^{4m+4-d}}\E_{2m}+\frac{b_1^2}{\lam^{4m+4-d}}\int\M\frac{\eps_{2m-2}^2}{1+y^4}\\
    &\lesssim\frac{b_1^2}{\lam^{4m+4-d}}\E_{2m},
\end{align}
and similarly we have
\begin{equation}
    \begin{aligned}
        \left|\int\frac{b_1(\Lambda\J)_{\lam}}{\lam^2\J_{\lam}^2}w_{2m}[\partial_t,\Li_{\lam}^{m-1}]w\right|+\left|\int\frac{b_1}{\lam^2}\J_{\lam}^{-1}w_{2m-1}[\partial_t,\A_
        {\lam}\Li_{\lam}^{m-1}]w\right|\lesssim\frac{b_1^2}{\lam^{4m+4-d}}\E_{2m}.
    \end{aligned}
\end{equation}
It remains to estimate one final term
\begin{equation}
    (\J_{\lam}^{-1}w_{2m},[\partial_t,\Li_{\lam}^m]w).
\end{equation}
To this end, we introduce the following notation
\begin{equation}
    \eps_{2m-2}^*(y)=\frac{1}{y^{d-1}\T_Q(y)}\int_0^yy_0^{d-1}\T_Q(y_0)y_0\Lambda V^B(y_0)\eps_{2m-1}(y_0)dy_0,\quad w^*_{2m-2}=\frac{1}{\lam^{2m-2}}(\eps^*_{2m-2})_{\lam}.
\end{equation}
Then it follows from \eqref{def:operator-A-B} that
\begin{equation}
    \B^*\eps_{2m-2}^*=y\Lambda V^B\eps_{2m-1},\quad \B_{\lam}^*w_{2m-2}^*=\lam r(\Lambda V^B)_{\lam}w_{2m-1}.
\end{equation}
Therefore we can compute
\begin{align}
    (\J_{\lam}^{-1}w_{2m},[\partial_t,\Li_{\lam}^m]w)&=(\J_{\lam}^{-1}w_{2m},[\partial_t,\B^*_{\lam}]w_{2m-1})+(\J_{\lam}^{-1}w_{2m},\B_{\lam}^*[\partial_t,\A_{\lam}\Li_{\lam}^{m-1}]w)\\
        &=\int\J_{\lam}^{-1}w_{2m}r\partial_tV_{\lam}^Bw_{2m-1}+(\J_{\lam}^{-1}\A_{\lam}w_{2m},[\partial_t,\A_{\lam}\Li_{\lam}^{m-1}]w)\\
        &=-\frac{\lam_s}{\lam}\frac{1}{\lam^2}\int\J_{\lam}^{-1}w_{2m}r(\Lambda V^B)_{\lam}w_{2m-1}+(\J_{\lam}^{-1}\A_{\lam}w_{2m},[\partial_t,\A_{\lam}\Li_{\lam}^{m-1}]w)\\
        &=-\frac{\lam_s}{\lam}\frac{1}{\lam^3}\int\J_{\lam}^{-1}\A_{\lam}w_{2m}w^*_{2m-2}+(\J_{\lam}^{-1}\A_{\lam}w_{2m},[\partial_t,\A_{\lam}\Li_{\lam}^{m-1}]w).
\end{align}
For the first term we have
\begin{align}
    \left|
        \frac{\lam_s}{\lam}\frac{1}{\lam^3}
        \int \J_{\lam}^{-1}\A_{\lam}w_{2m}w^*_{2m-2}
        \right|
        &\leq \frac{1}{100}\int \J_{\lam}^{-1}|\A_{\lam}w_{2m}|^2
        + C\frac{b_1^2}{\lam^6}
        \int \J_{\lam}^{-1}(w^*_{2m-2})^2  \\
        &= \frac{1}{100}\int \J_{\lam}^{-1}|\A_{\lam}w_{2m}|^2
        + C\frac{b_1^2}{\lam^{4m+4-d}}
        \int \J^{-1}(\eps^*_{2m-2})^2 \\
        &\leq \frac{1}{100}\int \J_{\lam}^{-1}|\A_{\lam}w_{2m}|^2
        + C\frac{b_1^2}{\lam^{4m+4-d}}\E_{2m}.
\end{align}
Here, in the last inequality we used the following by noting the full coercivity of $\B^*$ in Lemma~\ref{lemma:coercivity-Bstar}
\begin{equation}
    \begin{aligned}
        \int \J^{-1}(\eps^*_{2m-2})^2
    \lesssim
    \int \J^{-1}(1+y^2)(\B^*\eps^*_{2m-2})^2
    &=
    \int \J^{-1}(1+y^2)(y\Lambda  V^B)^2\eps_{2m-2}^2\\
    &\lesssim\int\J^{-1}\frac{1}{1+y^4}\eps_{2m-2}^2\lesssim \E_{2m}.
    \end{aligned}
\end{equation}
For the second term we simply recall the estimate \eqref{estimate:bound-on-commutator-general} to get
\begin{equation}
    \begin{aligned}
        |(\J_{\lam}^{-1}\A_{\lam}w_{2m},[\partial_t,\A_{\lam}\Li_{\lam}^{m-1}]w)|&\leq\frac{1}{100}\int \J_{\lam}^{-1}|\A_{\lam}w_{2m}|^2
        + C\int\J_{\lam}^{-1}|[\partial_t,\A_
        {\lam}\Li_{\lam}^{m-1}]w|^2\\
        &\leq\frac{1}{100}\int \J_{\lam}^{-1}|\A_{\lam}w_{2m}|^2+C\frac{b_1^2}{\lam^{4m+4-d}}\E_{2m}.
    \end{aligned}
\end{equation}
Therefore we conclude the estimates of all quadratic terms.

\step{4} Further use of dissipation on error terms. What we left are those terms with error:
\begin{equation}
    \begin{gathered}
        (\J_{\lam}^{-1}w_{2m},\frac{1}{\lam^2}\Li_{\lam}^m\F_{\lam}) -\frac{1}{2}\int\frac{b_1(\Lambda\J)_{\lam}}{\lam^2\J_{\lam}^2}w_{2m}\Li_{\lam}^{m-1}\left(\frac{1}{\lam^2}\F_{\lam}\right)\\
    -\frac{1}{2}\int\frac{b_1(\Lambda\J)_{\lam}}{\lam^2\J_{\lam}^2}w_{2m-2}\Li_{\lam}^{m}\left(\frac{1}{\lam^2}\F_{\lam}\right) -2\frac{b_1}{\lam^2}\int \J_{\lam}^{-1}w_{2m-1}\A_{\lam}\Li^{m-1}_{\lam}\left(\frac{1}{\lam^2}\F_{\lam}\right).
    \end{gathered}
\end{equation}
And we can use the following identity
\begin{equation}
    \int \J_{\lam}^{-1}w_{2m-1}\A_{\lam}\Li^{m-1}_{\lam}\left(\frac{1}{\lam^2}\F_{\lam}\right)=\int \J_{\lam}^{-1}w_{2m-2}\Li^{m}_{\lam}\left(\frac{1}{\lam^2}\F_{\lam}\right)=\int \J_{\lam}^{-1}w_{2m}\Li^{m-1}_{\lam}\left(\frac{1}{\lam^2}\F_{\lam}\right)
\end{equation}
to rewrite the error terms as
\begin{equation}\label{equation:left-error-terms}
    \begin{gathered}
        (\J_{\lam}^{-1}w_{2m},\frac{1}{\lam^2}\Li_{\lam}^m\F_{\lam})\\-\frac{1}{2}\int\frac{b_1(\Lambda\J+2\J)_{\lam}}{\lam^2\J_{\lam}^2}w_{2m}\Li_{\lam}^{m-1}\left(\frac{1}{\lam^2}\F_{\lam}\right)
    -\frac{1}{2}\int\frac{b_1(\Lambda\J+2\J)_{\lam}}{\lam^2\J_{\lam}^2}w_{2m-2}\Li_{\lam}^{m}\left(\frac{1}{\lam^2}\F_{\lam}\right).
    \end{gathered}
\end{equation}
Now our goal turns to the proof of \eqref{equation:left-error-terms}. We first introduce the decomposition of error $\F$
\begin{equation}\label{defeq:decompose_error}
    \F=\partial_s\xi_L+\F_0+\F_1,\quad \F_0=-\tilde{\Psi}_b-\widetilde{\text{Mod}},\quad \F_1=-\Hi(\eps)+\Ni(\eps),
\end{equation}
with the function defined as
\begin{equation}\label{defeq:xi_l_and_tilde_Mod}
    \xi_L=\frac{\langle\M\Li^L\eps,\chi_{B_0}\Lambda Q\rangle}{\langle\M\Lambda Q,\chi_{B_0}\Lambda Q\rangle}\tilde{T}_L,\quad\widetilde{\text{Mod}}=\widehat{\text{Mod}}+\partial_s\xi_L,
\end{equation}
here $\tilde{\Psi}_b,\widehat{\text{Mod}},\Hi(\eps)$ and $\Ni(\eps)$ are defined in \eqref{equation:profile-equation-approximate-local}, \eqref{defeq:small-linear-H}, \eqref{defeq:nonlinear-N} and \eqref{equation:def-tMod}. Therefore we can compute the first line in \eqref{equation:left-error-terms} as
\begin{align}
    (\J_{\lam}^{-1}w_{2m},\frac{1}{\lam^2}\Li_{\lam}^m\F_{\lam})=&\frac{1}{\lam^{4m+4-d}}(\J^{-1}\eps_{2m},\Li^m(\partial_s\xi_L))+\frac{1}{\lam^{2}}(\J_{\lam}^{-1}w_{2m},\Li_{\lam}^m(\F_0)_{\lam}+\Li_{\lam}^m(\F_1)_{\lam})\\
        \leq&\frac{1}{\lam^{4m+4-d}}(\J^{-1}\eps_{2m},\Li^m(\partial_s\xi_L))+\frac{\sqrt{\E_{2m}}}{\lam^{4m+4-d}}\left(\int\M|\Li^m\F_0|^2\right)^{\frac{1}{2}}\\
        &+\frac{1}{100}\int\J_{\lam}^{-1}|\A_{\lam}w_{2m}|^2+C\frac{1}{\lam^{4m+4-d}}\int\M|\A\Li^{m-1}\F_1|^2.\label{equation:left-error-high-order}
\end{align}
And we can estimate the second line in \eqref{equation:left-error-terms} with degeneracy property \eqref{equation:cancellation-J} and coercivity Lemma~\ref{lemma:coercive-bound-eps} to get
\begin{align}
    &-\frac{1}{2}\int\frac{b_1(\Lambda\J+2\J)_{\lam}}{\lam^2\J_{\lam}^2}w_{2m}\Li_{\lam}^{m-1}\left(\frac{1}{\lam^2}\F_{\lam}\right)
    -\frac{1}{2}\int\frac{b_1(\Lambda\J+2\J)_{\lam}}{\lam^2\J_{\lam}^2}w_{2m-2}\Li_{\lam}^{m}\left(\frac{1}{\lam^2}\F_{\lam}\right)\\
    =&\quad-\frac{1}{2}\frac{b_1}{\lam^{4m+4-d}}\int\frac{\Lambda\J+2\J}{\J^2}(\eps_{2m}\Li^{m-1}\partial_s\xi_L+\eps_{2m-2}\Li^m\partial_s\xi_L)\\
    &-\frac{1}{2}\frac{b_1}{\lam^{4m+4-d}}\int\frac{\Lambda\J+2\J}{\J^2}(\eps_{2m}\Li^{m-1}(\F_0+\F_1)+\eps_{2m-2}\Li^{m}(\F_0+\F_1))\\
    \leq&\quad-\frac{1}{2}\frac{b_1}{\lam^{4m+4-d}}\int\frac{\Lambda\J+2\J}{\J^2}(\eps_{2m}\Li^{m-1}\partial_s\xi_L+\eps_{2m-2}\Li^m\partial_s\xi_L)\\
    &+C\frac{b_1}{\lam^{4m+4-d}}\sqrt{\E_{2m}}\left(\int\frac{\M(|\Li^{m-1}\F_0|^2+|\Li^{m-1}\F_1|^2)}{1+y^4}+\int\M|\Li^{m}\F_0|^2\right)^{\frac{1}{2}}\\
    &-\frac{1}{2}\frac{b_1}{\lambda^{4m+4-d}}\int\J^{-1}\mathcal{A}(\frac{\Lambda\J+2\J}{\J}\eps_{2m-2})\mathcal{A}\Li^{m-1}\F_1\\
    \leq&\quad-\frac{1}{2}\frac{b_1}{\lam^{4m+4-d}}\int\frac{\Lambda\J+2\J}{\J^2}(\eps_{2m}\Li^{m-1}\partial_s\xi_L+\eps_{2m-2}\Li^m\partial_s\xi_L)\label{equation:left-error-low-order}\\
    &+C\frac{b_1}{\lam^{4m+4-d}}\sqrt{\E_{2m}}\left(\int\frac{\M(|\Li^{m-1}\F_0|^2+|\Li^{m-1}\F_1|^2)}{1+y^4}+\int\M|\Li^{m}\F_0|^2+\int\M|\A\Li^{m-1}\F_1|^2\right)^{\frac{1}{2}}.
\end{align}
We now decompose the error terms as in \eqref{defeq:decompose_error}. In the remaining steps, we estimate all the contributions appearing in \eqref{equation:left-error-high-order} and \eqref{equation:left-error-low-order}.

\step{5} $\Psi_b$ terms. It is straightforward from Proposition~\ref{proposition:localize-Qb} that we have the estimate:
\begin{equation}\label{estimate:claim-bound-psi}
    \left(\int\M|\Li^m\tilde{\Psi}_b|^2\right)^{\frac{1}{2}}+\left(\int\frac{\M|\Li^{m-1}\tilde{\Psi}_b|^2}{1+y^4}\right)^{\frac{1}{2}}\lesssim b_1^{L+1+(1-\delta)(1+\eta)}.
\end{equation}

\step{6} Mod terms. We claim the following bound:
\begin{equation}\label{estimate:claim-bound-mod}
    \left(\int\M|\Li^m\TMod|^2\right)^{\frac{1}{2}}+\left(\int\frac{\M|\Li^{m-1}\TMod|^2}{1+y^4}\right)^{\frac{1}{2}}\lesssim b_1\left(\frac{\sqrt{\E_{2m}}}{M^{2\delta}}+b_1^{\eta(1-\delta)}\sqrt{\E_{2m}}+b_1^{L+1+(1-\delta)(1+\eta)}\right).
\end{equation}
As before we only deal with the first term and the second term can be estimated similarly. Let us write
\begin{align}
    \TMod=&-\left(\frac{\lam_s}{\lam}+b_1\right)\Lambda\tilde{Q}_b+\sum_{i=1}^{L-1}[(b_i)_s+(2i-\gamma)b_1b_i-b_{i+1}]\tT_i\\
        &+\sum_{i=1}^{L}[(b_i)_s+(2i-\gamma)b_1b_i-b_{i+1}]\chi_{B_1}\sum_{j=i+1}^{L+2}\frac{\partial S_j}{\partial b_i}\\
        &+\left[(b_L)_s+(2L-\gamma)b_1b_L+\frac{d}{ds}\left\{\frac{\langle\M\Li^L\eps,\chi_{B_0}\Lambda Q\rangle}{\langle\M\Lambda Q,\chi_{B_0}\Lambda Q\rangle}\right\}\right]\tT_L+\frac{\langle\M\Li^L\eps,\chi_{B_0}\Lambda Q\rangle}{\langle\M\Lambda Q,\chi_{B_0}\Lambda Q\rangle}\partial_s\tT_L\label{equation:write-tmod}
\end{align}
by recalling \eqref{equation:Mod-term-1}. Since we have the asymptotic behavior in Lemma~\ref{lemma:construction-Tk}, part (ii) of Proposition~\ref{proposition:approximate-profile} and $|b_j|\lesssim b_1^j,\Li(\Lambda Q)=0$, we estimate
\begin{align}
    &\int\M|\Li^m\Lambda\tilde{Q}_b|^2\\
        &\lesssim\sum_{i=1}^Lb_i^2\int\M|\Li^m\Lambda \tT_i|^2+\sum_{i=2}^{L+2}\int\M|\Li^m\Lambda\tS_i|^2\\
        &\lesssim \sum_{i=1}^Lb_1^{2i}\int_{y\leq2B_1}\frac{(1+y^{d-1+4})dy}{1+y^{4(m-i+1)+2\gamma}}+\sum_{i=2}^{L+1}\int_{y\leq2B_1}\frac{(1+y^{d-1+4})dy}{1+y^{4(m-i+2)+2\gamma}}+b_1^{2L+4}\int_{y\leq2B_1}\frac{(1+y^{d-1+4})dy}{1+y^{4h+2\gamma+4}}\\
        &\lesssim b_1^2
\end{align}
where we used the algebra $4(m-L+1)+2\gamma-d+1-4=5-4\delta>1$. Again we note that $\Li^mT_i=0$ for $1\leq i\leq L$ and the asymptotic behavior in Lemma~\ref{lemma:construction-Tk}, we obtain
\begin{equation}
    \sum_{i=1}^{L-1}\int\M|\Li^m\tT_i|^2\lesssim\sum_{i=1}^{L-1}\int_{B_1\leq y\leq2B_1}y^{4(i-m-1)-2\gamma+d-1
    +4}dy\lesssim b_1^{(2+2(1-\delta))(1+\eta)}\lesssim b_1^2,
\end{equation}
and
\begin{equation}\label{estimate:T_L-energy}
    \int\M|\Li^m\tT_L|^2\lesssim\int_{B_1\leq y\leq2B_1}y^{4(L-m-1)-2\gamma+d-1+4}dy\lesssim b_1^{2(1-\delta)(1+\eta)}.
\end{equation}
Then the homogeneity of $S_j$ shows us
\begin{equation}
    \sum_{j=i+1}^{L+2}\int\M\left|\Li^m\left(\chi_{B_1}\frac{\partial S_j}{\partial b_i}\right)\right|^2\lesssim\sum_{j=i+1}^{L+2}b_1^{2(j-i)}\int_{B_1\leq y\leq2B_1}y^{4(j-2-m)-2\gamma+d-1+4}dy\lesssim b_1^2
\end{equation}
provided that $\eta<\frac{1}{\delta}-1$. We recall from \eqref{estimate:M-LambdaQ-LambdaQ} and \eqref{estimate:M-Lieps-LambdaQ} to get
\begin{equation}\label{estimate:xi_L-coefficient-energy}
    \left|\frac{\langle\M\Li^L\eps,\chi_{B_0}\Lambda Q\rangle}{\langle\M\Lambda Q,\chi_{B_0}\Lambda Q\rangle}\right|\lesssim B_0^{2(1-\delta)}\sqrt{\E_{2m}}=b_1^{-(1-\delta)}\sqrt{\E_{2m}}.
\end{equation}
We also have the following
\begin{equation}
    \int\M|\Li^m(\partial_s\tT_L)|^2\lesssim b_1^2\int_{B_1\leq y\leq2 B_1}\frac{y^{d-1+4}}{y^{4(m-L+1)+2\gamma}}\lesssim b_1^2b_1^{2(1-\delta)(1+\eta)}.
\end{equation}
Now we collect all the above bounds together with the refined estimate in Lemma~\ref{lemma:modulation-bound-improve} to conclude as
\begin{align}
    \left(\int\M|\Li^m\TMod|^2\right)^{\frac{1}{2}}&\lesssim b_1\left(\frac{\sqrt{\E_{2m}}}{M^{2\delta}}+b_1^{L+1+(1-\delta)(1+\eta)}\right)\\
        &\quad+b_1^{(1-\delta)(1+\eta)}b_1^{\delta}\left(C(M)\sqrt{\E_{2m}}+b_1^{L+1+(1-\delta)(1+\eta)}\right)\\
        &\quad+b_1^{-(1-\delta)}\sqrt{\E_{2m}}b_1b_1^{(1-\delta)(1+\eta)}\\
        &\lesssim b_1\left(\frac{\sqrt{\E_{2m}}}{M^{2\delta}}+b_1^{\eta(1-\delta)}\sqrt{\E_{2m}}+b_1^{L+1+(1-\delta)(1+\eta)}\right).
\end{align}

\step{7} Small linear term $\Hi(\eps)$. We claim
\begin{equation}\label{estimate:claim-bound-small-linear}
    \int\M|\A\Li^{m-1}\Hi(\eps)|^2+\int\frac{\M|\Li^{m-1}\Hi(\eps)|^2}{1+y^4}\lesssim b_1^2\E_{2m}.
\end{equation}
We only deal with the first part. Let us split the small linear term into two parts:
\begin{equation}
    \begin{gathered}
        \Hi_1(\eps):=-(y\partial_y\tilde{\Theta}_b+2d\tilde{\Theta}_b)\eps,\quad\Hi_2(\eps):=-\tilde{\Theta}_by\partial_y\eps
    \end{gathered}
\end{equation}
where we use
\begin{equation}
    \tilde{\Theta}_b=\sum_{i=1}^Lb_i\tT_i+\sum_{i=2}^{L+2}\tS_i(b,y).
\end{equation}
Next we observe from \eqref{defeq:linear-operator-in-V} to obtain
\begin{equation}
    y\partial_y\eps=-y\A\eps+y^2V^A\eps.
\end{equation}
Using the admissibility of $T_i$ in Lemma~\ref{lemma:construction-Tk} and the homogeneity of $S_i$ in Proposition~\ref{proposition:approximate-profile}, we apply the Leibniz rule in Lemma~\ref{lemma:lebniz-rule} to write
\begin{equation}
    \begin{gathered}
        \A\Li^{m-1}\Hi_1(\eps)=\sum_{i=0}^{m-1}[\eps_{2i+1}U_{1,2i+1}+\eps_{2i}U_{1,2i}],\\
        \A\Li^{m-1}\Hi_2(\eps)=\sum_{i=1}^{m}[\eps_{2i}U_{2,2i}+\eps_{2i-1}U_{2,2i-1}]+\eps_0U_{2,0}
    \end{gathered}
\end{equation}
with the decay estimate
\begin{equation}
    |U_{1,k}|+|U_{2,k}|\lesssim \frac{b_1}{1+y^{m-k-1+\gamma}},\quad 0\leq k\leq m.
\end{equation}
Hence we estimate with the coercivity in Lemma~\ref{lemma:coercive-bound-eps} to observe
\begin{equation}
    \int\M|\eps_{k}|^2|U_{i,k}|^2\lesssim\int\M\frac{b_1^2|\eps_{k}|^2}{1+y^{2(m-k-1+\gamma)}}\lesssim b_1^2\E_{2m},\quad i=1,2,\quad0\leq k\leq m.
\end{equation}
and hence reach the desired result.

\step{8} Nonlinear term $\mathcal{N}(\eps)$. We first decompose the nonlinear error as
\begin{equation}\label{defeq:split-nonlinear}
    \begin{gathered}
        \Ni(\eps)=\Ni_1(\eps)+d\Ni_2(\eps),\\
        \Ni_1(\eps)=y\eps\partial_y\eps,\quad \Ni_2(\eps)=\eps^2.
    \end{gathered}
\end{equation}
We now claim the following
\begin{equation}\label{estimate:claim-bound-nonlinear}
    \int\M|\A\Li^{m-1}\Ni(\eps)|^2+\int\frac{\M|\Li^{m-1}\Ni(\eps)|^2}{1+y^4}\lesssim  C(K)b_1^{2L+1+2(1-\delta)(1+\eta)+C\eta}\lesssim b_1^{2L+1+2(1-\delta)(1+\eta)}
\end{equation}
given the smallness of $b_1$ chosen with respect to $K$. We deal with the first part and the second part can follow the exact same path. It suffices to prove
\begin{equation}\label{estimate:claim-nonlinear}
    \int\M|\A\Li^{m-1}\Ni_1(\eps)|^2+\int\M|\A\Li^{m-1}\Ni_2(\eps)|^2\lesssim b_1^{2L+1+2(1-\delta)(1+\eta)}.
\end{equation}
We note that the first nonlinear term $\Ni_1$ contains one derivative, and therefore requires a more careful treatment in the estimates below. To begin we have the following decomposition
\begin{equation}\label{defeq:decompose-N1}
\begin{gathered}
    \A\Li^{m-1}\Ni_1(\eps)=y\eps\Li^m\eps+\Ni_1^m(\eps),\\
    \Ni_1^m(\eps):=\A\Li^{m-1}(y\eps\partial_y\eps)-y\eps \Li^m\eps.
\end{gathered}
\end{equation}
\indent\textit{I. Estimate for $y<1$}: Recall the decomposition for $\Ni_1$ in \eqref{defeq:decompose-N1} and the coercivity we can get
\begin{equation}
    \int_{y<1}\M |y\eps|^2|\Li^m\eps|^2\lesssim \E_{2m}^2.
\end{equation}
To treat $\Ni_1^m$, we note in Lemma~\ref{lemma:inter-bound-origin} we have the asymptotic behavior near $y=0$ for $\eps$, and this gives
\begin{equation}
    \begin{gathered}
        \mathcal{N}_1=\left(\sum_{i=0}^{m-1} c_iT_i(y)+r_{\eps}(y)\right)\left(\sum_{i=0}^{m-1} c_iy\partial_yT_i(y)+y\partial_yr_{\eps}(y)\right)\\
    =\sum_{i=0}^{m-1}\tilde{c}_iy^{2i}+\tilde{r}_{\eps}+\eps y\partial_y r_{\eps},
    \end{gathered}
\end{equation}
with
\begin{equation}\label{estimate:pointw-ci-treps}
    |\tilde{c}_i|\lesssim\E_{2m},\quad |\partial_y^j\tilde{r}_{\eps}(y)|\lesssim|\ln(y)|^m y^{2m-1-\frac{d}{2}-j}\E_{2m},\quad  0\leq j\leq2m-1.
\end{equation}
Therefore we have
\begin{equation}
    \Ni_1^m=\A\Li^{m-1}\left(\sum_{i=0}^{m-1}\tilde{c}_iy^{2i}+\tilde{r}_{\eps}\right)+\left(\A\Li^{m-1}(\eps y\partial_y r_{\eps})-\eps y\Li^mr_{\eps}\right)=\Ni_{1,1}^m+\Ni_{1,2}^m.
\end{equation}
Using the structure of linear operator, together with the pointwise estimate of $\eps,r_{\eps}$ and  $\tilde{r}_{\eps}$ in Lemma~\ref{lemma:inter-bound-origin} and \eqref{estimate:pointw-ci-treps}, we obtain
\begin{equation}
\begin{gathered}
    |\Ni_{1,1}^m|\lesssim\sum_{i=0}^{m-1}\tilde{c}_i+\sum_{i=0}^{2m-1}\frac{|\partial_y^i\tilde{r}_{\eps}|}{y^{2m-1-i}}\lesssim\E_{2m}+\E_{2m}|\ln(y)|^my^{-\frac{d}{2}}\lesssim\E_{2m}|\ln(y)|^my^{-\frac{d}{2}},\\
    |\Ni_{1,2}^m|\lesssim\sum_{i+j\leq2m-1}\frac{|\partial_y^i\eps||\partial_y^j r_{\eps}|}{y^{2m-1-i-j}}\lesssim\sum_{i+j\leq2m-1}\frac{\max\{1,y^{2m-1-\frac{d}{2}-i}\}y^{2m-1-\frac{d}{2}-j}}{y^{2m-1-i-j}}\E_{2m}\lesssim\E_{2m}|\ln(y)|^my^{-\frac{d}{2}}
\end{gathered}
\end{equation} 
Here we use the fact that $\A(1)=O(y)$ for $y<1$ by recalling asymptotic behavior \eqref{property:decay-V}. Therefore we can conclude that
\begin{equation}
    \int_{y<1}\M|\Ni_1^m|^2\lesssim\int_{y<1}|\ln(y)|^{2m}y^2y^{-d}y^{d-1}dy\E_{2m}^2\lesssim\E_{2m}^2
\end{equation}
which shows
\begin{equation}
    \int_{y<1}\M|\A\Li^{m-1}\Ni_1(\eps)|^2\lesssim\E_{2m}^2.
\end{equation}
The same estimate holds for $\Ni_2$ if we write
\begin{equation}       \Ni_2=\left(\sum_{i=0}^{m-1} c_iT_i(y)+r_{\eps}(y)\right)\left(\sum_{i=0}^{m-1} c_iT_i(y)+r_{\eps}(y)\right)
    =\sum_{i=0}^{m-1}\hat{c}_iy^{2i}+\hat{r}_{\eps}
\end{equation}
with the same bounds
\begin{equation}
    |\hat{c}_i|\lesssim\E_{2m},\quad |\partial_y^j\hat{r}_{\eps}(y)|\lesssim y^{2m-1-\frac{d}{2}-j}|\ln(y)|^m\E_{2m},\quad  0\leq j\leq2m-1,
\end{equation}
and then follow the steps for $\Ni_1$. We close the case for $y<1$ by applying the bootstrap assumption to get
\begin{equation}
    \E_{2m}^2\lesssim C(K)b_1^{4L+4(1-\delta)(1+\eta)}\lesssim b_1^{2L+2(1-\delta)(1+\eta)+1}.
\end{equation}
\indent\textit{II. Estimate for $y>1$}: We are able to write the following estimate by using the structure of linear operators in \eqref{defeq:linear-operator-in-V} again:
\begin{equation}\label{estimate:nonlinear-yge1-firststep}
    \int_{y>1}\M|\A\Li^{m-1}\mathcal{N}(\eps)|^2\lesssim\int_{y>1}\M|y\eps|^2|\Li^m\eps|^2+\sum_{i+j\leq 2m-1}\int_{y>1}\frac{\M|\partial_y^i\eps|^2|\partial_y^j\eps|^2}{y^{4m-2-2i-2j}}.
\end{equation}
For the first term on the right-hand side of \eqref{estimate:nonlinear-yge1-firststep}, we use the pointwise bound in Lemma~\ref{lemma:inter-bound-far} to obtain the following for $y>1$
\begin{equation}\label{estimate:pointwise-yeps}
    |y\eps|^2=\frac{y^{d+2}\eps^2}{y^d}\lesssim\left\{
    \begin{aligned}
         & \E_{2h+2}^{\frac{1}{4}}\E_{2h+4}^{\frac{3}{4}}, && d=4h+5,\\
       & \E_{2h+4}, && d=4h+6,\\
        & \E_{2h+4}^
        {\frac{3}{4}}\E_{2h+6}^{\frac{1}{4}}, && d=4h+7,\\
         &\E_{2h+4}^
        {\frac{1}{2}}\E_{2h+6}^{\frac{1}{2}}, && d=4h+8,\\
        & \E_{2h+4}^{\frac{1}{4}}\E_{2h+6}^{\frac{3}{4}}, && d=4h+9.
    \end{aligned}
    \right.
\end{equation}
We note that \(d = 4h + 4\delta + 2\gamma\). In the case \(d = 11\), we have \(h = 1\) and \(d = 4h + 7\). For \(d \geq 12\), we have \(2 < \gamma < 3\) and \(\delta < 1\), which implies \(d \leq 4h + 9\). Hence, the cases categorized in \eqref{estimate:pointwise-yeps} cover all values of \(d\). When $l\geq3$, we apply the bootstrap assumption to show that in all cases in \eqref{estimate:pointwise-yeps} we have
\begin{equation}
    |y\eps|^2\lesssim C(K)b_1^{\frac{l}{2l-\gamma}(d+2-d)}=C(K)b_1^{\frac{2l}{2l-\gamma}}.
\end{equation}
This shows
\begin{equation}\label{estimate:nonlinear-y-rps-Li-eps}
    \int_{y>1}\M|y\eps|^2|\Li^m\eps|^2\lesssim  C(K)b_1^{\frac{2l}{2l-\gamma}}\E_{2m}\lesssim C(K)b_1^{\frac{2l}{2l-\gamma}}b_1^{2L+2(1-\delta)(1+\eta)}\lesssim b_1^{2L+2(1-\delta)(1+\eta)+1}.
\end{equation}
For $l=2$, put $S=d-4h=4\delta+2\gamma$, $a=2/(4-\gamma)$,
$q=6-2\delta$ and $P=2L+2(1-\delta)(1+\eta)$.
The low-energy range stops at $k=h+2$:
\begin{equation}\label{estimate:critical-pointwise-1}
    \mathcal{E}_{2h+4}\lesssim_K b_1^{a(8-S)},
\qquad
\mathcal{E}_{2h+6}\lesssim b_1^{q-K\eta}.
\end{equation}
Substitution in the unchanged \eqref{estimate:pointwise-yeps} gives
$\|y\varepsilon\|_{L^\infty(y>1)}^2\lesssim_K b_1^{e_S-CK\eta}$,
where
\begin{equation}\label{estimate:critical-pointwise-2}
    \begin{array}{c|cccc}
S   & 5,6 & 7 & 8 & 9 \\ \hline
e_S & 2a
    & \frac{3}{4}a+\frac{1}{4}q
    & \frac{1}{2}q
    & -\frac{1}{4}a+\frac{3}{4}q
\end{array}.
\end{equation}
Note all $e_S>1$, thus the final bound in \eqref{estimate:nonlinear-y-rps-Li-eps} holds.

For the second term in \eqref{estimate:nonlinear-yge1-firststep}, we can assume $i\geq j$ and split it into three cases.

\case
{I} $l\geq 3$ and $4\delta+2\gamma=5,7,9$. Thus $4\delta+2\gamma=2\mu+1$, $\mu\in\N$. We use again the asymptotic behavior of $\M$ to derive the following estimate for $i+j\leq2m-1$
\begin{align}
    \int_{y>1}\frac{\M|\partial_y^i\eps|^2|\partial_y^j\eps|^2}{y^{4m-2-2i-2j}}&\lesssim\int_{y>1}\frac{y^4|\partial_y^i\eps|^2|\partial_y^j\eps|^2y^{d-1}}{y^{4m-2-2i-2j}}dy\lesssim \int_{y>1}\frac{\left(y^{d+2}|\partial_y^i\eps|^2\right)\left(y^{d+2}|\partial_y^j\eps|^2\right)}{y^{4m-2i-2j+4h+4\delta+2\gamma-1}}dy\\
        &\lesssim\left\|\left(\frac{y^{d+2}|\partial_y^i\eps|^2}{y^{2I_1-2i}}\right)\left(\frac{y^{d+2}|\partial_y^j\eps|^2}{y^{2J_1-2j}}\right)\right\|_{L^{\infty}(y\geq1)}\int_{1<y<b_1^{-\alpha}}y^{9-4\delta-2\gamma}dy\\
        &\quad+\left\|\left(\frac{y^{d+2}|\partial_y^i\eps|^2}{y^{2I_2-2i}}\right)\left(\frac{y^{d+2}|\partial_y^j\eps|^2}{y^{2J_2-2j}}\right)\right\|_{L^{\infty}(y\geq1)}\int_{y>b_1^{-\alpha}}\frac{1}{y^2}dy\label{estimate:spatial-decomposition}\\
        &\lesssim B_{i,j,I_1,J_1}b_1^{\alpha(4\delta+2\gamma-10)}+B_{i,j,I_2,J_2}b_1^{\alpha}\label{estimate:nonlinear-pointwise-yge1}
\end{align}
where we use the following
\begin{equation}
    \begin{gathered}
        I_1+J_1=2m+2h+4,\quad I_1\geq i,\quad J_1\geq j,\\
        I_2+J_2=2m+2h+\mu-1,\quad I_2\geq i,\quad J_2\geq j,
    \end{gathered}
\end{equation}
and the truncation value $b_1^{-\alpha}$ will be chosen later. 

We first estimate $B_{i,j,I_1,J_1}$.

\noindent $\bullet$ If $j\leq2h+5$, then we choose
\begin{equation}
    I_1=2m-1,\quad J_1=2h+5.
\end{equation}
Now we apply the pointwise bound in Lemma~\ref{lemma:inter-bound-far} and the bootstrap assumptions to obtain
\begin{align}
    B_{i,j,I_1,J_1}&\lesssim \E_{2m}\E_{2h+6}\lesssim C(K)b_1^{2L+2(1-\delta)(1+\eta)+\frac{l}{2l-\gamma}(4h+12-d)}\\
    &=C(K)b_1^{2L+2(1-\delta)(1+\eta)+1+\frac{1}{2}(10-4\delta-2\gamma)+\frac{\gamma}{2(2l-\gamma)}(12-4\delta-2\gamma)}\\
    &\lesssim C(K)b_1^{2L+2(1-\delta)(1+\eta)+1+\frac{1}{2}(10-4\delta-2\gamma)+8C\eta}\label{estimate:I1J1-1}
\end{align}
if we choose sufficiently small $\eta$.

\noindent $\bullet$ If $j\geq2h+6$, then we choose
\begin{equation}
    I_1=2m+2h+4-J_1,\quad J_1=j.
\end{equation}
This leads to
\begin{equation}
\label{estimate:I1J1-mid}
    B_{i,j,I_1,J_1}\lesssim\left\{
    \begin{aligned}
         & \E_{2m+2h+5-J_1}\E_{J_1+1}, && \text{if $J_1$ is  odd}\\
       & \sqrt{\E_{2m+2h+4-J_1}\E_{2m+2h+6-J_1}}\sqrt{\E_{J_1}\E_{J_1+2}}, && \text{if $J_1$ is  even}.
    \end{aligned}
    \right. 
\end{equation}
We now apply the bootstrap assumptions to \eqref{estimate:I1J1-mid}. If $J_1\leq 2h+2l-1$, then
\begin{align}
    B_{i,j,I_1,J_1}&\lesssim C(K)b_1^{(2m+2h+5-J_1-2h-2)+2(1-\delta)-C\eta+\frac{l}{2l-\gamma}(2J_1+2-d)}\\
        &=C(K)b_1^{2L+2(1-\delta)(1+\eta)+1+\frac{1}{2}(10-4\delta-2\gamma)+\frac{\gamma}{2(2l-\gamma)}(2J_1+2-d)-C\eta}\\
        &\lesssim C(K)b_1^{2L+2(1-\delta)(1+\eta)+1+\frac{1}{2}(10-4\delta-2\gamma)+8C\eta}\label{estimate:I1J1-2}
\end{align}
where in the last step we used $J_1\geq 2h+6$, which gives
\begin{equation}
    \frac{\gamma}{2(2l-\gamma)}(2J_1+2-d)\geq10C\eta.
\end{equation}
On the other hand, if $J_1\geq 2h+2l+1$, then
\begin{align}
    B_{i,j,I_1,J_1}&\lesssim C(K)b_1^{(2m+2h+5-J_1-2h-2)+2(1-\delta)-C\eta+(J_1+1-2h-2)+2(1-\delta)-C\eta}\\
        &=C(K)b_1^{2L+2(1-\delta)(1+\eta)+1+\frac{1}{2}(10-4\delta-2\gamma)+\gamma-2C\eta}\\
        &\lesssim C(K)b_1^{2L+2(1-\delta)(1+\eta)+1+\frac{1}{2}(10-4\delta-2\gamma)+8C\eta}.\label{estimate:I1J1-3}
\end{align}
Finally if $J_1=2h+2l$, we already have the estimate for $\E_{2m+6-2l}\E_{2h+2l}$ and $\E_{2m+4-2l}\E_{2h+2l+2}$ in the above two cases, so we can conclude
\begin{equation}\label{estimate:I1J1-4}
    B_{i,j,I_1,J_1}\lesssim C(K)b_1^{2L+2(1-\delta)(1+\eta)+1+\frac{1}{2}(10-4\delta-2\gamma)+8C\eta}.
\end{equation}

We then estimate $B_{i,j,I_2,J_2}$ following the same steps.

\noindent $\bullet$ If $j\leq2h+\mu$, we choose
\begin{equation}
    I_2=2m-1,\quad J_2=2h+\mu.
\end{equation}
Then it follows
\begin{equation}
    B_{i,j,I_2,J_2}\lesssim\left\{
    \begin{aligned}
         & \E_{2m}\E_{2h+\mu+1}, && \text{if $\mu$ is  odd}\\
       & \E_{2m}\sqrt{\E_{2h+\mu}\E_{2h+\mu+2}}, && \text{if $\mu$ is  even}.
    \end{aligned}
    \right. 
\end{equation}
Therefore
\begin{equation}\label{estimate:I2J2-1}
    \begin{aligned}
        B_{i,j,I_2,J_2}\lesssim C(K)b_1^{2L+2(1-\delta)(1+\eta)+\frac{l}{2l-\gamma}(4h+2\mu+2-d)}&=C(K)b_1^{2L+2(1-\delta)(1+\eta)+1-\frac{1}{2}+\frac{\gamma}{2(2l-\gamma)}}\\
        &\lesssim C(K)b_1^{2L+2(1-\delta)(1+\eta)+1-\frac{1}{2}+C\eta}.
    \end{aligned}
\end{equation}

\noindent $\bullet$ If $j\geq2h+\mu+1$, we choose
\begin{equation}
    I_2=2m+2h+\mu-1-J_2,\quad J_2=j.
\end{equation}
This leads to
\begin{equation}
    B_{i,j,I_2,J_2}\lesssim\left\{
    \begin{aligned}
         & \E_{I_2+1}\E_{J_2+1}, && \text{if $I_2$ odd, $J_2$ odd},\\
       & \E_{I_2+1}\sqrt{\E_{J_2}\E_{J_2+2}}, && \text{if $I_2$ odd, $J_2$ even},\\
       & \sqrt{\E_{I_2}\E_{I_2+2}}\E_{J_2+1}, && \text{if $I_2$ even, $J_2$ odd},\\
       &\sqrt{\E_{I_2}\E_{I_2+2}}\sqrt{\E_{J_2}\E_{J_2+2}},&&\text{if $I_2$ even, $J_2$ even}.
    \end{aligned}
    \right. 
\end{equation}
Thus in the case $J_2\leq2h+2l-1$ we obtain
\begin{align}
    B_{i,j,I_2,J_2}&\lesssim C(K)b_1^{(I_2+1-2h-2)+2(1-\delta)-C\eta+\frac{l}{2l-\gamma}(2J_2+2-d)}\\
        &=C(K)b_1^{2L+2(1-\delta)(1+\eta)+1-\frac{1}{2}+\frac{\gamma}{2(2l-\gamma)}(2J_2+2-d)-C\eta}\\
        &\lesssim C(K)b_1^{2L+2(1-\delta)(1+\eta)+1-\frac{1}{2}+C\eta}.\label{estimate:I2J2-2}
\end{align}
Here we use
\begin{equation}
    2J_2+2-d\geq4h+2\mu+2+2-d=3>0.
\end{equation}
Then in the case $J_2\geq 2h+2l+1$ we have
\begin{equation}\label{estimate:I2J2-3}
    \begin{aligned}
        B_{i,j,I_2,J_2}&\lesssim C(K)b_1^{(I_2+1-2h-2)+2(1-\delta)-C\eta+(J_2+1-2h-2)+2(1-\delta)-C\eta}\\
        &=C(K)b_1^{2L+2(1-\delta)(1+\eta)+1-\frac{1}{2}+\gamma-2C\eta}\lesssim C(K)b_1^{2L+2(1-\delta)(1+\eta)+1-\frac{1}{2}+C\eta}.
    \end{aligned}
\end{equation}
And for the case $J_2=2h+2l$ the same bound holds.

Combining the estimates in all cases \eqref{estimate:I1J1-1}, \eqref{estimate:I1J1-2}, \eqref{estimate:I1J1-3}, \eqref{estimate:I1J1-4}, \eqref{estimate:I2J2-1}, \eqref{estimate:I2J2-2} and \eqref{estimate:I2J2-3} we can bound \eqref{estimate:nonlinear-pointwise-yge1} by choosing $\alpha=\frac{1}{2}$ as
\begin{equation}
    B_{i,j,I_1,J_1}b_1^{\frac{1}{2}(4\delta+2\gamma-10)}+B_{i,j,I_2,J_2}b_1^{\frac{1}{2}}\lesssim C(K)b_1^{2L+2(1-\delta)(1+\eta)+1+C\eta}\lesssim b_1^{2L+2(1-\delta)(1+\eta)+1}
\end{equation}
which is the desired result.

\case{II} $l\geq 3$ and $4\delta+2\gamma=6,8$. Thus $4\delta+2\gamma=2\mu$, $\mu\in\N$.  We proceed as  in case I to obtain
\begin{align}
    \int_{y>1}\frac{\M|\partial_y^i\eps|^2|\partial_y^j\eps|^2}{y^{4m-2-2i-2j}}&\lesssim \int_{y>1}\frac{\left(y^{d+2}|\partial_y^i\eps|^2\right)\left(y^{d+2}|\partial_y^j\eps|^2\right)}{y^{4m-2i-2j+4h+4\delta+2\gamma-1}}dy\\
        &\lesssim\left\|\left(\frac{y^{d+2}|\partial_y^i\eps|^2}{y^{2I_3-2i}}\right)\left(\frac{y^{d+2}|\partial_y^j\eps|^2}{y^{2J_3-2j}}\right)\right\|_{L^{\infty}(y\geq1)}\int_{1<y<b_1^{-\alpha}}y^{9-4\delta-2\gamma}dy\\
        &\quad+\left\|\left(\frac{y^{d+2}|\partial_y^i\eps|^2}{y^{2I_4-2i}}\right)\left(\frac{y^{d+2}|\partial_y^j\eps|^2}{y^{2J_4-2j}}\right)\right\|_{L^{\infty}(y\geq1)}\int_{y>b_1^{-\alpha}}\frac{1}{y^3}dy\\
        &\lesssim B_{i,j,I_3,J_3}b_1^{\alpha(4\delta+2\gamma-10)}+B_{i,j,I_4,J_4}b_1^{2\alpha}.\label{estimate:nonlinear-pointwise-yge1-caseII}
\end{align}
where
\begin{equation}
    \begin{gathered}
        I_3+J_3=2m+2h+4,\quad I_3\geq i,\quad J_3\geq j,\\
        I_4+J_4=2m+2h+\mu-2,\quad I_4\geq i,\quad J_4\geq j.
    \end{gathered}
\end{equation}

We choose $I_3=I_1,J_3=J_1$ and claim the following bound for $B_{i,j,I_3,J_3}$
\begin{equation}\label{estimate:I3J3}
    B_{i,j,I_3,J_3}\lesssim C(K)b_1^{2L+2(1-\delta)(1+\eta)+1+\frac{l}{2l-\gamma}(10-4\delta-2\gamma)+C\eta}.
\end{equation}
In fact, we can prove \eqref{estimate:I3J3} by revising the estimate in the last step of \eqref{estimate:I1J1-1}, \eqref{estimate:I1J1-2} and \eqref{estimate:I1J1-3}. In \eqref{estimate:I1J1-1} we use
\begin{equation}
    \frac{\gamma}{2(2l-\gamma)}(12-4\delta-2\gamma)\geq\frac{\gamma}{2(2l-\gamma)}(10-4\delta-2\gamma)+C\eta.
\end{equation}
And in \eqref{estimate:I1J1-2} we have
\begin{equation}
\begin{aligned}
    \frac{\gamma}{2(2l-\gamma)}(2J_1+2-d)-C\eta&\geq \frac{\gamma}{2(2l-\gamma)}(2(2h+6)+2-d)-C\eta\\&=\frac{\gamma}{2(2l-\gamma)}(14-4\delta-2\gamma)-C\eta
    \geq\frac{\gamma}{2(2l-\gamma)}(10-4\delta-2\gamma)+C\eta.
\end{aligned}
\end{equation}
Finally in \eqref{estimate:I1J1-3} we have
\begin{equation}
    \begin{aligned}
        \gamma-2C\eta&=\frac{\gamma}{2(2l-\gamma)}(10-4\delta-2\gamma)+\frac{\gamma}{2(2l-\gamma)}(4l+4\delta-10)-2C\eta\\
        &\geq\frac{\gamma}{2(2l-\gamma)}(10-4\delta-2\gamma)+C\eta,
    \end{aligned}
\end{equation}
whree we use $l\geq3$ to give the estimate.

We now estimate $B_{i,j,I_4,J_4}$.

\noindent $\bullet$ If $j\leq2h+\mu-1$, we choose
\begin{equation}
    I_4=2m-1,\quad J_4=2h+\mu-1.
\end{equation}
Then 
\begin{equation}
    B_{i,j,I_4,J_4}\lesssim\left\{
    \begin{aligned}
         & \E_{2m}\E_{2h+\mu}, && \text{if $\mu$ is  even}\\
       & \E_{2m}\sqrt{\E_{2h+\mu-1}\E_{2h+\mu+1}}, && \text{if $\mu$ is  odd}.
    \end{aligned}
    \right. 
\end{equation}
Thus
\begin{equation}\label{estimate:I4J4-1}
    \begin{aligned}
        B_{i,j,I_4,J_4}\lesssim C(K)b_1^{2L+2(1-\delta)(1+\eta)+\frac{l}{2l-\gamma}(4h+2\mu-d)}&=C(K)b_1^{2L+2(1-\delta)(1+\eta)}.
    \end{aligned}
\end{equation}

\noindent $\bullet$ If $j\geq2h+\mu$, we choose
\begin{equation}
    I_4=2m+2h+\mu-2-J_4,\quad J_4=j.
\end{equation}
This leads to
\begin{equation}
    B_{i,j,I_4,J_4}\lesssim\left\{
    \begin{aligned}
         & \E_{I_4+1}\E_{J_4+1}, && \text{if $I_4$ odd, $J_4$ odd},\\
       & \E_{I_4+1}\sqrt{\E_{J_4}\E_{J_4+2}}, && \text{if $I_4$ odd, $J_4$ even},\\
       & \sqrt{\E_{I_4}\E_{I_4+2}}\E_{J_4+1}, && \text{if $I_4$ even, $J_4$ odd},\\
       &\sqrt{\E_{I_4}\E_{I_4+2}}\sqrt{\E_{J_4}\E_{J_4+2}},&&\text{if $I_4$ even, $J_4$ even}.
    \end{aligned}
    \right. 
\end{equation}
Then we can estimate in the case $J_4\leq2h+2l-1$:
\begin{align}
    B_{i,j,I_4,J_4}&\lesssim C(K)b_1^{(I_4+1-2h-2)+2(1-\delta)-C\eta+\frac{l}{2l-\gamma}(2J_4+2-d)}\\
        &=C(K)b_1^{2L+2(1-\delta)(1+\eta)+\frac{\gamma}{2(2l-\gamma)}(2J_4+2-d)-C\eta}\\
        &\lesssim C(K)b_1^{2L+2(1-\delta)(1+\eta)}.\label{estimate:I4J4-2}
\end{align}
Here we use
\begin{equation}
    \frac{\gamma}{2(2l-\gamma)}(2J_4+2-d)-C\eta\geq \frac{\gamma}{2(2l-\gamma)}(4h+2\mu+2-d)-C\eta=\frac{\gamma}{2l-\gamma}-C\eta>0.
\end{equation}
In the case $J_4\geq2h+2l+1$ we obtain
\begin{equation}\label{estimate:I4J4-3}
    \begin{aligned}
        B_{i,j,I_4,J_4}&\lesssim C(K)b_1^{(I_4+1-2h-2)+2(1-\delta)-C\eta+(J_4+1-2h-2)+2(1-\delta)-C\eta}\\
        &=C(K)b_1^{2L+2(1-\delta)(1+\eta)+\gamma-2C\eta}\lesssim C(K)b_1^{2L+2(1-\delta)(1+\eta)}.
    \end{aligned}
\end{equation}
And for the case $J_4=2h+2l$ the same estimate holds.

Now we combine \eqref{estimate:I3J3}, \eqref{estimate:I4J4-1}, \eqref{estimate:I4J4-2} and \eqref{estimate:I4J4-3} together and choose $\alpha=\frac{l}{2l-\gamma}$ to obtain
\begin{equation}
    \begin{aligned}
        B_{i,j,I_3,J_3}b_1^{\alpha(4\delta+2\gamma-10)}+B_{i,j,I_4,J_4}b_1^{2\alpha}&\lesssim C(K)b_1^{2L+2(1-\delta)(1+\eta)+1+C\eta}+C(K)b_1^{2L+2(1-\delta)(1+\eta)+\frac{2l}{2l-\gamma}}\\
        &\lesssim b_1^{2L+2(1-\delta)(1+\eta)+1}.
    \end{aligned}
\end{equation}

\case{III} $l=2$. As shown in \eqref{estimate:critical-pointwise-1} and \eqref{estimate:critical-pointwise-2}, for the $l=2$ case we need further computations given $\E_{2h+6}$ is no longer available for low-energy bound. 
Use again
\[
S=d-4h=4\delta+2\gamma,
\qquad
a=\frac{2}{4-\gamma},
\qquad
q=6-2\delta,
\qquad
P=2L+2(1-\delta)(1+\eta),
\]
and we take
\[
\alpha=\frac12
\quad\text{for } S=5,7,9,
\qquad
\alpha=\frac12+\kappa,
\quad 0<\kappa<\frac14,
\quad\text{for } S=6,8.
\]
Lemma~\ref{lemma:inter-bound-far} bounds each pointwise factor by \(\mathcal E_{J+1}\) when
\(J\) is odd, and by
\[
(\mathcal E_J\mathcal E_{J+2})^{1/2}
\]
when \(J\) is even. At the transition \(J=2h+4\), we recall the estimate in \eqref{estimate:critical-pointwise-1} and \eqref{estimate:critical-pointwise-2} to get
\[
(\mathcal E_{2h+4}\mathcal E_{2h+6})^{1/2}
\lesssim C(K)
b_1^{[a(8-S)+q]/2-CK\eta}.
\]
For \(L\) sufficiently large, both the fixed and variable near-field
allocations give
\[
B_{i,j,I,J}
=
B_{i,j,J_1,J_2}
\lesssim C(K)
b_1^{P+q-CK\eta}.
\]
Multiplying by the near-field weight \(b_1^{\alpha(S-10)}\), and
performing the analogous calculation in the far field with
\(b_1^{1/2}\) for odd \(S\) and \(b_1^{1+2\kappa}\) for even \(S\),
gives the following lower bounds for the powers beyond \(P-CK\eta\):
\[
\begin{array}{c|cccc}
& S=5,7 & S=9 & S=6 & S=8 \\ \hline
\text{near field}
& 1+\gamma
& 1+\gamma
& 1+\gamma-4\kappa
& 1+\gamma-2\kappa \\[1mm]
\text{far field, fixed}
& a+\frac12
& \frac{q-a+1}{2}
& 1+2\kappa
& 1+2\kappa \\[2mm]
\text{far field, variable}
& a+\frac12
& 1+\gamma
& 2a+2\kappa
& \frac q2+2\kappa
\end{array}
\]

The variable cases use \(J_2=j\) or \(J_4=j\), as above, including
the admissible indices \(2h+3\) and \(2h+4\). In particular, for
\(S=6\), the transition gain satisfies
\[
a+\frac q2-1+2\kappa
\ge 2a+2\kappa,
\]
while for \(S=8\) it is
\[
\frac q2+2\kappa.
\]
Thus every gain displayed above is strictly larger than one. Thus we can choose $\eta$ small enough to control \eqref{estimate:nonlinear-yge1-firststep} by
\[
b_1^{P+1}.
\]
This concludes the estimate of the nonlinear
terms.

\step{9} Time oscillations. In this step, we treat all the terms with time derivatives. First we bound the left hand side of \eqref{equation:energy-identity-final}. It is straightforward to estimate using dissipation and coercivity:
\begin{align}
    \int\frac{b_1(\Lambda\J)_{\lam}}{\lam^2\J_{\lam}^2}w_{2m}w_{2m-2}+2\int\frac{b_1}{\lam^2\J_{\lam}}w_{2m-1}^2&=\frac{b_1}{\lam^2}\int\frac{(\Lambda\J)_{\lam}+2\J_{\lam}}{\J_{\lam}^2}w_{2m}w_{2m-2}\\
        &\lesssim\frac{b_1}{\lam^{4m+2-d}}\left(\int\M w_{2m}^2\right)^{\frac{1}{2}}\left(\frac{\M w_{2m-2}^2}{1+y^4}\right)^{\frac{1}{2}}\\
        &\lesssim\frac{b_1}{\lam^{4m+2-d}}\E_{2m}.
\end{align}
Then we need to estimate the time oscillations with $\xi_L$ in \eqref{equation:left-error-high-order} and \eqref{equation:left-error-low-order}. Let us write the terms in \eqref{equation:left-error-high-order} including $\xi_L$ as
\begin{align}
    \frac{1}{\lam^{4m+4-d}}(\J^{-1}\eps_{2m},\Li^m(\partial_s\xi_L))=&\frac{d}{ds}\left\{\frac{1}{\lam^{4m+4-d}}\left[\int\M\Li^m\eps\Li^m\xi_L-\frac{1}{2}\int\M|\Li^m\xi_L|^2\right]\right\}\\
    &+\frac{4m+4-d}{\lam^{4m+4-d}}\frac{\lam_s}{\lam}\left[\int\M\Li^m\eps\Li^m\xi_L-\frac{1}{2}\int\M|\Li^m\xi_L|^2\right]\\
    &-\frac{1}{\lam^{4m+4-d}}\int\M\Li^m(\partial_s\eps-\partial_s\xi_L)\Li^m\xi_L.\label{equation:rewrite-time-oscillation-1}
\end{align}
From \eqref{estimate:T_L-energy} and \eqref{estimate:xi_L-coefficient-energy} we have
\begin{equation}
    \int\M|\Li^m\xi_L|^2\leq b_1^{2\eta(1-\delta)}\E_{2m}.
\end{equation}
This implies
\begin{equation}
    \left|\int\M\Li^m\eps\Li^m\xi_L\right|\lesssim b_1^{\eta(1-\delta)}\E_{2m}.
\end{equation}
Since $dt/ds=\lam^2$, we then obtain the following estimate
\begin{equation}\label{estimate:time-oscillation-1}
\begin{aligned}
    \frac{d}{ds}\left\{\frac{1}{\lam^{4m+4-d}}\left[\int\M\Li^m\eps\Li^m\xi_L-\frac{1}{2}\int\M|\Li^m\xi_L|^2\right]\right\}=&\frac{d}{dt}\left(\frac{\E_{2m}}{\lam^{4m+2-d}}O(b_1^{\eta(1-\delta)})\right)\\
    -2&\frac{\lam_s}{\lam}\frac{1}{\lam^{4m+4-d}}\left[\int\M\Li^m\eps\Li^m\xi_L-\frac{1}{2}\int\M|\Li^m\xi_L|^2\right].
\end{aligned}
\end{equation}
We now note that $|\lam_s/\lam|\lesssim b_1$ to get
\begin{equation}\label{estimate:time-oscillation-2}
    \left|\frac{\lam_s}{\lam}\left[\int\M\Li^m\eps\Li^m\xi_L-\frac{1}{2}\int\M|\Li^m\xi_L|^2\right]\right|\lesssim b_1b_1^{\eta(1-\delta)}\E_{2m}.
\end{equation}
Using this estimate we can give the bound for the second term in the right hand side of \eqref{equation:left-error-high-order} and the last term in \eqref{estimate:time-oscillation-1}.
For the last term in \eqref{equation:rewrite-time-oscillation-1} we recall the decomposition \eqref{defeq:decompose_error} to get
\begin{equation}\label{equation:rewrite-time-estimate-2}
    \begin{aligned}
        &\int\M\Li^m(\partial_s\eps-\partial_s\xi_L)\Li^m\xi_L\\
        &=-\int\M\Li^m\eps\Li^{m+1}\xi_L-\frac{\lam_s}{\lam}\int\M\Lambda\eps\Li^{2m}\xi_L+\int\M\Li^m(-\tilde{\Psi}_b-\TMod-\Hi\eps+\Ni\eps)\Li^m\xi_L.
    \end{aligned}
\end{equation}
Using \eqref{estimate:xi_L-coefficient-energy} and the decay rate of $T_L$, we have the estimate
\begin{align}
    \int\M|\Li^{m+1}\xi_L|^2&\lesssim \left|\frac{\langle\M\Li^L\eps,\chi_{B_0}\Lambda Q\rangle}{\langle\M\Lambda Q,\chi_{B_0}\Lambda Q\rangle}\right|^2\int\M|\Li^{m+1}(1-\chi_{B_1})T_L|^2\\
        &\lesssim b_1^{-2(1-\delta)}\E_{2m}\int_{y\geq B_1}y^4y^{2(2(L-1)-\gamma-2(m+1))}y^{d-1}dy\\
        &\lesssim b_1^{-2(1-\delta)}\E_{2m}b_1^{(4-2\delta)(1+\eta)}\lesssim b_1^2b_1^{2\eta(1-\delta)}\E_{2m}
\end{align}
from which we obtain
\begin{equation}
    \left|\int\M\Li^m\eps\Li^{m+1}\xi_L\right|\lesssim b_1b_1^{\eta(1-\delta)}\E_{2m}.
\end{equation}
Similarly we have 
\begin{equation}
    \int(1+y^{4m})\M|\Li^{2m}\xi_L|^2\lesssim b_1^{-2(1-\delta)}\E_{2m}\int_{y\geq B_1}y^{4m+4}y^{2(2(L-1)-\gamma-4m)}y^{d-1}dy\lesssim b_1^{2\eta(1-\delta)}\E_{2m},
\end{equation}
hence from coercivity we obtain
\begin{equation}
\begin{aligned}
    \left|\frac{\lam_s}{\lam}\int\M\Lambda\eps\Li^{2m}\xi_L\right|\lesssim b_1\left(\int\M\frac{|\partial_y\eps|^2}{1+y^{4m-2}}+\int\M\frac{|\eps|^2}{1+y^{4m}}\right)^{\frac{1}{2}}\left(\int(1+y^{4m})\M|\Li^{2m}\xi_L|^2\right)^{\frac{1}{2}}\\
    \lesssim b_1b_1^{\eta(1-\delta)}\E_{2m}.
\end{aligned}
\end{equation}
Now we use \eqref{estimate:claim-bound-psi} and \eqref{estimate:claim-bound-mod} to have
\begin{equation}
    \begin{aligned}
        \left|\int\M\Li^m(\tilde{\Psi}_b+\TMod)\Li^m\xi_L\right|&\lesssim\left(\int\M|\Li^m\xi_L|^2\right)^{\frac{1}{2}}\left(\int\M|\Li^m(\tilde{\Psi}_b+\TMod)|^2\right)^{\frac{1}{2}}\\
        &\lesssim b_1b_1^{\eta(1-\delta)}\E_{2m}+b_1b_1^{L+(1-\delta)(1+\eta)}\sqrt{\E_{2m}}.
    \end{aligned}
\end{equation}
We also have
\[
    \int \mathcal{M}\mathcal{L}^{m-1}(\Hi(\eps)-\Ni(\eps))\,
    \mathcal{L}^{m+1}\xi_L
    =
    \int \mathcal{M}A\mathcal{L}^{m-1}(\Hi(\eps)-\Ni(\eps))\,
    A\mathcal{L}^{m}\xi_L.
\]
Since $\mathcal{L}^mT_L=0$, $A\mathcal{L}^m\widetilde{T}_L$ is supported on
$B_1\leq y\leq 2B_1$. Admissibility and \eqref{estimate:xi_L-coefficient-energy} yield
\[
    \int \mathcal{M}\left|A\mathcal{L}^{m}\widetilde{T}_L\right|^2
    \lesssim B_1^{4\delta-6},
    \qquad
    \int \mathcal{M}\left|A\mathcal{L}^{m}\xi_L\right|^2
    \lesssim
    b_1^{1+(3-2\delta)\eta}\mathcal{E}_{2m}.
\]
Thus \eqref{estimate:claim-bound-small-linear} and \eqref{estimate:claim-bound-nonlinear} imply
\begin{equation}
    \begin{aligned}
        \left|
        \int \mathcal{M}\mathcal{L}^{m-1}(\Hi(\eps)-\Ni(\eps))\,
        \mathcal{L}^{m+1}\xi_L
        \right|
        &\lesssim
        \left(
            b_1\sqrt{\mathcal{E}_{2m}}
            +b_1^{L+(1-\delta)(1+\eta)+\frac{1}{2}}
        \right)
        b_1^{\frac12+\frac{(3-2\delta)\eta}{2}}
        \sqrt{\mathcal{E}_{2m}}
        \\
        &\lesssim
        b_1^{1+\eta(1-\delta)}\mathcal{E}_{2m}
        +b_1^{L+(1-\delta)(1+\eta)+1}\sqrt{\mathcal{E}_{2m}}.
    \end{aligned}
\end{equation}
Collecting the above bounds yields
\begin{equation}\label{estimate:time-oscillation-3}
    \left|\int\M\Li^m(\partial_s\eps-\partial_s\xi_L)\Li^m\xi_L\right|\lesssim b_1b_1^{\eta(1-\delta)}\E_{2m}+b_1b_1^{L+(1-\delta)(1+\eta)}\sqrt{\E_{2m}}.
\end{equation}
Applying \eqref{estimate:time-oscillation-1}, \eqref{estimate:time-oscillation-2} and \eqref{estimate:time-oscillation-3} we conclude the estimate for \eqref{equation:rewrite-time-oscillation-1}.

Finally, for the term including $\xi_L$ in \eqref{equation:left-error-low-order} we can write
\begin{align}
    &\frac{b_1}{\lam^{4m+4-d}}\int\frac{\Lambda\J+2\J}{\J^2}(\eps_{2m}\Li^{m-1}\partial_s\xi_L+\eps_{2m-2}\Li^m\partial_s\xi_L)\\
        &=\frac{d}{ds}\left\{\frac{b_1}{\lam^{4m+4-d}}\left[\int\frac{\Lambda\J+2\J}{\J^2}(\Li^m\eps\Li^{m-1}\xi_L+\Li^{m-1}\eps\Li^{m}\xi_L-\Li^{m-1}\xi_L\Li^{m}\xi_L)\right]\right\}\\
        &\quad-\frac{d}{ds}\left(\frac{b_1}{\lam^{4m+4-d}}\right)\left[\int\frac{\Lambda\J+2\J}{\J^2}(\Li^m\eps\Li^{m-1}\xi_L+\Li^{m-1}\eps\Li^{m}\xi_L-\Li^{m-1}\xi_L\Li^{m}\xi_L)\right]\\
        &\quad-\frac{b_1}{\lam^{4m+4-d}}\int\frac{\Lambda\J+2\J}{\J^2}(\Li^m(\partial_s\eps-\partial_s\xi_L)\Li^{m-1}\xi_L+\Li^{m-1}(\partial_s\eps-\partial_s\xi_L)\Li^{m}\xi_L)\label{equation:rewrite-time-oscillation-2}
\end{align}
Following the same argument as in the estimate of \eqref{equation:rewrite-time-oscillation-1} above, we can give the exact same bound. Therefore we conclude the proof of Proposition \ref{proposition:energy-control}.
\end{proof}

\section{Proof of the main theorem}\label{sec:proof-theorem}
In this section, we complete the proof of Theorem~\ref{theorem:main-theorem}. We first close the bootstrap argument by means of a topological argument, showing that there exists a choice of initial data for which the corresponding solution remains in the bootstrap regime throughout its evolution. We begin with the following proposition, which reduces the closure of the bootstrap to a finite-dimensional problem for the unstable modes.
\begin{proposition}\label{proposition:reduce-finite-dimension}
    There exists a sufficiently large
constant  $K$ such that: Suppose the solution $(b(s),\eps(s))$ remains in $\Si_K(s)$ for all $s\in[s_0,s_1]$ with large enough $s_0$, namely, that the bootstrap assumptions in  Definition~\ref{definition:bootstrap-assumption} hold, and assume that $(b(s_1),\eps(s_1))\in\partial\Si_K(s_1)$. Then the following holds:
    \begin{enumerate}
        \item [(i)](Unstable modes exit)
        \begin{equation}\label{equation:unstable-mode-exist}
            (\V_2(s_1),\dots,\V_l(s_1))\in\partial\left[-\frac{100}{s_1^{\frac{\eta}{2}(1-\delta)}},\frac{100}{s_1^{\frac{\eta}{2}(1-\delta)}}\right]^{l-1}.
        \end{equation}
        \item [(ii)](Transverse crossing) For $2\leq k\leq l$, if at time $s^*$
        \begin{equation}
            |(s^*)^{\frac{\eta}{2}(1-\delta)}\V_k(s^*)|\geq1.
        \end{equation}
        Then we have
        \begin{equation}\label{equation:transverse-crossing}
            \left.\frac{d}{ds}\left(|s^{\frac{\eta}{2}(1-\delta)}\V_k(s)|^2\right)\right|_{s=s^*}>0.
        \end{equation}
    \end{enumerate}
\end{proposition}
\begin{proof}
    To give the proof of \eqref{equation:unstable-mode-exist}, we follow the traditional bootstrap procedure. In particular, we claim that
    \begin{align}
        &|\V_1(s)|\leq s^{-\frac{\eta}{2}(1-\delta)}, \label{estimate:bootstrap-improve-V1}\\
        &|b_k(s)|\lesssim s^{-(k+\eta(1-\delta))},&& l+1\leq k\leq L, \label{estimate:bootstrap-improve-bk}\\
        &\E_{2m}(s)
    \leq \frac{K}{2} s^{-(2L+2(1-\delta)(1+\eta))} \label{estimate:bootstrap-improve-E2m}\\
        &\E_{2k}(s)
    \leq
    \left\{
    \begin{aligned}
        &\frac{K}{2}s^{-\frac{l}{2l-\gamma}(4k-d)},\\
        &\frac{1}{2}s^{-2(k-h-1)-2(1-\delta)+K\eta},
    \end{aligned}
    \right.
    &&
    \begin{aligned}
        &h+1\leq k\leq l+h,\\
        &l+h+1\leq k\leq m-1
    \end{aligned} \label{estimate:bootstrap-improve-E2k}
    \end{align}
    hold for all $s\in[s_0,s_1]$. Recalling Definition~\ref{definition:bootstrap-assumption}, the estimates above show that all the bootstrap bounds, except those associated with the unstable modes, are strictly improved. Consequently, the solution can reach the boundary of the bootstrap region only through the unstable modes, which proves \eqref{equation:unstable-mode-exist}. We now turn to the proof of \eqref{estimate:bootstrap-improve-V1}, \eqref{estimate:bootstrap-improve-bk}, \eqref{estimate:bootstrap-improve-E2m} and \eqref{estimate:bootstrap-improve-E2k}.

    We first estimate the parameter $\lam$. The bootstrap assumption shows
    \begin{equation}\label{estimate:b-bound-inbootstrap}
        b_1(s)=\frac{c_1}{s}+\frac{\U_1}{s}=\frac{l}{(2l-\gamma)s}+O\left(\frac{1}{s^{1+c\eta}}\right).
    \end{equation}
    Recalling \eqref{estimate:modulation-bound-L-1} we get
    \begin{equation}\label{estimate:lam-s-bound-inbootstrap}
        -\frac{\lam_s}{\lam}=\frac{l}{(2l-\gamma)s}+O(\frac{1}{s^{1+c\eta}}).
    \end{equation}
    This implies
    \begin{equation}\label{estimate:lam-bound-inbootstrap}
        s_0^{-\frac{l}{2l-\gamma}}\lesssim \frac{s^{-\frac{l}{2l-\gamma}}}{\lam(s)}\lesssim s_0^{-\frac{l}{2l-\gamma}}.
    \end{equation}
    \textit{Proof of \eqref{estimate:bootstrap-improve-E2m}}:
        We recall the energy control \eqref{equation:energy-control-1}, multiply the time derivative term with $\lam^2$. Then we use the control for $\lambda$ in \eqref{estimate:b-bound-inbootstrap}, \eqref{estimate:lam-s-bound-inbootstrap} and \eqref{estimate:lam-bound-inbootstrap} to obtain
        \begin{align}
            &\frac{d}{dt}\left\{\frac{\E_{2m}}{\lambda^{4m-d}}[1+O(b_1^{\eta(1-\delta)})]\right\}\\
            &=\lam^2\frac{d}{dt}\left\{\frac{\E_{2m}}{\lambda^{4m-d+2}}[1+O(b_1^{\eta(1-\delta)})]\right\}+2\frac{\lam_s}{\lam}\frac{\E_{2m}}{\lambda^{4m-d+2}}[1+O(b_1^{\eta(1-\delta)})]\\
            &=\lam^2\frac{d}{dt}\left\{\frac{\E_{2m}}{\lambda^{4m-d+2}}[1+O(b_1^{\eta(1-\delta)})]\right\}-2(b_1+O(b_1^{1+\eta(1-\delta)}))\frac{\E_{2m}}{\lambda^{4m-d+2}}[1+O(b_1^{\eta(1-\delta)})]\\
            &\lesssim \frac{b_1}{\lambda^{4m+2-d}}[(b_1^{\eta(1-\delta)}+1/M^{2\delta})\E_{2m}+b_1^{L+(1-\delta)(1+\eta)}\sqrt{\E_{2m}}+b_1^{2L+2(1-\delta)(1+\eta)}].
        \end{align}
        Then we integrate in time by using the bootstrap assumption in Definition~\ref{definition:bootstrap-assumption} and $\lam(s_0)=1$ to obtain
        \begin{equation}
            \E_{2m}(s)\leq C\lambda(s)^{4m-d}\left[\E_{2m}(s_0)+\left(\frac{K}{s_0^{\eta(1-\delta)}}+\frac{K}{M^{2\delta}}+\sqrt{K}+1\right)\int_{s_0}^s\frac{\tau^{-1-2(L+(1-\delta)(1+\eta))}}{\lambda(\tau)^{4m-d}}d\tau\right].
        \end{equation}
        Thus we can estimate the term inside the integral by applying the bound for $\lam$ in \eqref{estimate:lam-bound-inbootstrap}:
        \begin{equation}
        \begin{aligned}
             \lam(s)^{4m-d}\int_{s_0}^s\frac{\tau^{-1-2(L+(1-\delta)(1+\eta))}}{\lambda(\tau)^{4m-d}}d\tau&\lesssim s^{-\frac{l(4m-d)}{2l-\gamma}}\int_{s_0}^s\tau^{\frac{l(4m-d)}{2l-\gamma}-1-2(L+(1-\delta)(1+\eta))}d\tau\\
             &\lesssim s^{-2(L+(1-\delta)(1+\eta))},             
        \end{aligned}
        \end{equation}
        where we use the fact that the exponent inside the integral admits
        \begin{equation}
            \frac{l(4m-d)}{2l-\gamma}-[2L+1+2(1-\delta)(1+\eta)]=\frac{2\gamma L}{2l-\gamma}+O(1)\gg-1
        \end{equation}
       for $L\gg1$ sufficiently large. Using again \eqref{estimate:lam-bound-inbootstrap} we get
        \begin{equation}
            \lam(s)^{4m-d}\E_{2m}(s_0)\lesssim \left(\frac{s_0}{s}\right)^{\frac{l(4m-d)}{2l-\gamma}}s_0^{-\frac{100Ll}{2l-\gamma}}\lesssim s^{-2(L+(1-\delta)(1+\eta))}
        \end{equation}
        for large enough $L$. Thus we can conclude there exists large enough $K$
        \begin{equation}
            \E_{2m}(s)\leq C\left(\frac{K}{s_0^{\eta(1-\delta)}}+\frac{K}{M^{2\delta}}+\sqrt{K}+1\right)s^{-2(L+(1-\delta)(1+\eta))}\leq\frac{K}{2}s^{-2(L+(1-\delta)(1+\eta))}
        \end{equation}
        by choosing the initial data $s_0$ and $M$ large enough.

        \medskip\noindent\textit{proof of \eqref{estimate:bootstrap-improve-E2k}}: We consider two cases for the control of $\E_{2k}$:
        
        \case{1} $h+1\leq k\leq l+h$. Following the same procedure as in the proof of \eqref{estimate:bootstrap-improve-E2m}, we can use the energy control \eqref{equation:energy-control-2} to arrive at
        \begin{equation}
            \frac{d}{dt}\left\{\frac{\E_{2k}}{\lambda^{4k-d}}[1+O(b_1)]\right\}\lesssim \frac{b_1}{\lambda^{4k+2-d}}\left[b_1^{\eta(1-\delta)}\E_{2k}+b_1^{k-h-1+(1-\delta)-C\eta}\sqrt{\E_{2k}}+b_1^{2(k-h-1)+2(1-\delta)-C\eta}\right].
        \end{equation}
        Integrate this we obtain
        \begin{equation}
            \begin{aligned}
                 \E_{2k}(s)\leq C\lambda(s)^{4k-d}
    \Bigg[&\E_{2k}(s_0)+K\int_{s_0}^s\frac{\tau^{-1-\eta(1-\delta)-\frac{l}{2l-\gamma}(4k-d)}}{\lambda(\tau)^{4k-d}}d\tau\\&+\sqrt{K}\int_{s_0}^s\frac{\tau^{-\frac{l}{2l-\gamma}(2k-\frac{d}{2})-(k-h+1-\delta-C\eta)}}{\lambda(\tau)^{4k-d}}d\tau+\int_{s_0}^s\frac{\tau^{-(2k-2h-1+2(1-\delta)-C\eta)}}{\lambda(\tau)^{4k-d}}d\tau\Bigg].
            \end{aligned}
        \end{equation}
        The choice of the initial data and \eqref{estimate:lam-bound-inbootstrap} shows
        \begin{equation}
            \lam(s)^{4k-d}\E_{2k}(s_0)\lesssim s^{-\frac{l}{2l-\gamma}(4k-d)},
        \end{equation}
        and
        \begin{equation}
        \begin{aligned}
            \lam(s)^{4k-d}K\int_{s_0}^s\frac{\tau^{-1-\eta(1-\delta)-\frac{l}{2l-\gamma}(4k-d)}}{\lambda(\tau)^{4k-d}}d\tau&\lesssim Ks^{-\frac{l(4k-d)}{2l-\gamma}}\int_{s_0}^s
\tau^{-1-\eta(1-\delta)}d\tau\\
&\lesssim \frac{K}{s_0^{\eta(1-\delta)}}s^{-\frac{l(4k-d)}{2l-\gamma}}.
\end{aligned}
        \end{equation}
        For the remaining two integral parts, we compute the exponent of $\tau$ inside the integral. The first exponent is
        \begin{equation}
            \begin{aligned}
                \frac{l}{2l-\gamma}(2k-\frac{d}{2})-(k-h+1-\delta-C\eta)&=-\frac{\gamma}{2}-1+C\eta+\frac{\gamma}{2l-\gamma}(k-h-\delta-\frac{\gamma}{2})\\
    &\leq-1-\frac{\gamma\delta}{2l-\gamma}+C\eta<-1,
            \end{aligned}
        \end{equation}
        and the second exponent is
        \begin{equation}
            \begin{aligned}
                \frac{l}{2l-\gamma}(4k-d)-(2k-2h-1+2(1-\delta)-C\eta)&=-\gamma-1+C\eta+\frac{\gamma}{2l-\gamma}(2k-2h-2\delta-\gamma)\\
    &\leq -1-\frac{2\gamma\delta}{2l-\gamma}+C\eta<-1.
            \end{aligned}
        \end{equation}
        This shows that these two integrals are bounded. Therefore we obtain
        \begin{equation}
            \E_{2k}(s)\leq C\left(\frac{K}{s_0^{\eta(1-\delta)}}+\sqrt{K}+1\right)s^{-\frac{l}{2l-\gamma}(4k-d)}\leq\frac{K}{2}s^{-\frac{l}{2l-\gamma}(4k-d)}.
        \end{equation}
        for $K$ large enough.

        \case{2} $l+h+1\leq k\leq m-1$.  Proceeding as in the previous cases we get
        \begin{equation}
        \begin{aligned}
            \E_{2k}(s)\leq C\lam(s)^{4k-d}\Bigg[\E_{2k}(s_0)+\int_{s_0}^s\frac{\tau^{-1-2(k-h-1+1-\delta)+\left(C+\frac{K}{2}\right)\eta-\eta(1-\delta)}}{\lambda(\tau)^{4k-d}}d\tau\\
            +\int_{s_0}^s\frac{\tau^{-1-2(k-h-1+1-\delta)+\left(C+\frac{K}{2}\right)\eta}}{\lambda(\tau)^{4k-d}}d\tau\Bigg].
        \end{aligned}
        \end{equation}
        The first part inside the integral is strictly smaller than the second part. From the identity
        \begin{equation}
        \begin{aligned}
            &\frac{l}{2l-\gamma}(4k-d)-(2k-2h-1+2(1-\delta))+\left(C+\frac{K}{2}\right)\eta\\
            &=-\gamma-1+\frac{\gamma}{2l-\gamma}(2k-2h-2\delta-\gamma)+\left(C+\frac{K}{2}\right)\eta\\
    &\geq-1+\frac{2\gamma(1-\delta)}{2l-\gamma}+\left(C+\frac{K}{2}\right)\eta>-1,
        \end{aligned}
        \end{equation}
        we obtain
        \begin{align}
            \lam(s)^{4k-d}\int_{s_0}^s\frac{\tau^{-1-2(k-h-1+1-\delta)+\left(C+\frac{K}{2}\right)\eta}}{\lambda(\tau)^{4k-d}}d\tau&\lesssim s^{-\frac{l(4k-d)}{2l-\gamma}}\int_{s_0}^s\tau^{\frac{l}{2l-\gamma}(4k-d)-(2k-2h-1+2(1-\delta))+\left(C+\frac{K}{2}\right)\eta}d\tau\\
            &\lesssim s^{-(2k-2h-1+2(1-\delta))+\left(C+\frac{K}{2}\right)\eta+1}\\&\leq\frac{1}{4}s^{-(2k-2h-2+2(1-\delta))+K\eta}.
        \end{align}
        Using the choice of initial data we get
        \begin{equation}
            C\lam(s)^{4k-d}\E_{2k}(s_0)\lesssim s^{-\frac{l(4k-d)}{2l-\gamma}}\lesssim s^{-(2k-2h-2+2(1-\delta))+\left(C+\frac{K}{2}\right)\eta}\leq \frac{1}{4}s^{-(2k-2h-2+2(1-\delta))+K\eta}.
        \end{equation}
        This concludes the proof of \eqref{estimate:bootstrap-improve-E2k}.

        \medskip\noindent\textit{Proof of \eqref{estimate:bootstrap-improve-bk}}: We close the improved bound for the stable modes $(b_{l+1},\dots,b_L)$. It is equivalent to prove the bound
        \begin{equation}\label{estimate:induction-bound-bk}
            |b_k|\lesssim b_1^{k+\eta(1-\delta)}
        \end{equation}
        for $l+1\leq k\leq L$. We argue by descending induction. For the base case $k=L$, we set
        \begin{equation}
            \tilde{b}_L=b_L+\frac{\langle\M\Li^L\eps,\chi_{B_0}\Lambda Q\rangle}{\langle\M\Lambda Q,\chi_{B_0}\Lambda Q\rangle}.
        \end{equation}
        Recalling the estimate \eqref{estimate:xi_L-coefficient-energy} and \eqref{estimate:bootstrap-improve-E2m} we obtain
        \begin{equation}
            |\tilde{b}_L-b_L|\lesssim b_1^{-(1-\delta)}\sqrt{\E_{2m}}\lesssim b_1^{L+\eta(1-\delta)}.
        \end{equation}
        Using again the improved modulation bound \eqref{estimate:modulation-bound-improve} we get
        \begin{equation}
            |(\tilde{b}_L)_s+(2L-\gamma)b_1\tilde{b}_L|\lesssim b_1|\tilde{b}_L-b_L|+\frac{1}{B_0^{2\delta}}[C(M)\sqrt{\E_{2m}}+b_1^{L+1+(1-\delta)-C\eta}]\lesssim b_1^{L+1+\eta(1-\delta)},
        \end{equation}
        and this implies
        \begin{equation}
            \left|\frac{d}{ds}\left\{\frac{\tilde{b}_L}{\lam^{2L-\gamma}}\right\}\right|\lesssim \frac{b_1^{L+1+\eta(1-\delta)}}{\lam^{2L-\gamma}}.
        \end{equation}
        Integrating this from $s_0$ we obtain
        \begin{equation}
            \tilde{b}_L(s)\lesssim \lam(s)^{2L-\gamma}\left(\tilde{b}_L(s_0)+\int_{s_0}^s\frac{b_1^{L+1+\eta(1-\delta)}(\tau)}{\lam^{2L-\gamma}(\tau)}d\tau\right).
        \end{equation}
        We note the following estimate
        \begin{equation}
            \begin{gathered}
                \lam(s)^{2L-\gamma}\tilde{b}_L(s_0)\lesssim s^{-\frac{l(2L-\gamma)}{2l-\gamma}}\lesssim s^{-L-\eta(1-\delta)},\\
                \lam(s)^{2L-\gamma}\int_{s_0}^s\frac{b_1^{L+1+\eta(1-\delta)}(\tau)}{\lam^{2L-\gamma}(\tau)}d\tau\lesssim s^{-\frac{l(2L-\gamma)}{2l-\gamma}}\int_{s_0}^s\tau^{\frac{l(2L-\gamma)}{2l-\gamma}-L-1-\eta(1-\delta)}d\tau\lesssim s^{-L-\eta(1-\delta)},
            \end{gathered}
        \end{equation}
        and hence we have
        \begin{equation}
            |b_L(s)|\lesssim |\tilde{b}_L(s)|+|b_L(s)-\tilde{b}_L(s)|\lesssim s^{-L-\eta(1-\delta)}\lesssim b_1^{L+\eta(1-\delta)}.
        \end{equation}
        
        Now we assume $l+1\leq k\leq L-1$ and \eqref{estimate:induction-bound-bk} holds for $k+1$. The modulation bound \eqref{estimate:modulation-bound-L-1} implies
        \begin{equation}
            \left|(b_k)_s-(2k-\gamma)\frac{\lam_s}{\lam}b_k\right|\lesssim b_1^{L+1}+|b_{k+1}|\lesssim b_1^{k+1+\eta(1-\delta)},
        \end{equation}
        and
        \begin{equation}
            \left|\frac{d}{ds}\left\{\frac{b_k}{\lam^{2k-\gamma}}\right\}\right|\lesssim \frac{b_1^{k+1+\eta(1-\delta)}}{\lam^{2k-\gamma}}.
        \end{equation}
        Then integrating this by using the relation $k\geq l+1$, we obtain
        \begin{equation}
            b_k(s)\lesssim \lam(s)^{2k-\gamma}\left(b_k(s_0)+\int_{s_0}^s\frac{b_1^{k+1+\eta(1-\delta)}(\tau)}{\lam^{2k-\gamma}(\tau)}d\tau\right)\lesssim s^{-k-\eta(1-\delta)}.
        \end{equation}
        This concludes the proof of \eqref{estimate:bootstrap-improve-bk}.

        \medskip\noindent\textit{Proof of \eqref{estimate:bootstrap-improve-V1}}: Recall from \eqref{equation:perturb-bk} and \eqref{defeq:definition-Vl} we have
        \begin{equation}
            b_k = b_k^e+\frac{\U_k}{s^k},\quad 1\leq k\leq l,\quad \V=P_l\U,
        \end{equation}
        and $P_l$ diagonalizes the matrix $A_l$ with spectrum \eqref{defeq:diag-Al}. Then we use the relation \eqref{equation:linearized-dynamic-system}, the modulation bound \eqref{estimate:modulation-bound-L-1}, the improved bound \eqref{estimate:bootstrap-improve-bk} and the priori bound in Definition~\ref{definition:bootstrap-assumption} to get for $1\leq k\leq l-1$,
        \begin{equation}
            |s(\U_k)_s-(A_l\U)_k|\lesssim s^{k+1}|(b_k)_s+(2k-\gamma)b_1b_k-b_{k+1}|+|\U|^2\lesssim s^{-L+k}+|\U|^2\lesssim s^{-\eta(1-\delta)},
        \end{equation}
        and
        \begin{equation}
            |s(\U_l)_s-(A_l\U)_l|\lesssim s^{l+1}(|(b_l)_s+(2l-\gamma)b_1b_l-b_{l+1}|+|b_{l+1}|)+|\U|^2\lesssim s^{-\eta(1-\delta)}+|\U|^2\lesssim s^{-\eta(1-\delta)}.
        \end{equation}
        This implies
        \begin{equation}\label{equation:relation-sVs}
            s\V_s=D_l\V+O(s^{-\eta(1-\delta)}).
        \end{equation}
        Since the first eigenvalue of $D_l$ is exactly $-1$, we obtain the control for the stable mode $\V_1$:
        \begin{equation}
            |(s\V_1)_s|\lesssim s^{-\eta(1-\delta)},
        \end{equation}
        and this implies
        \begin{equation}
            |s^{\eta(1-\delta)}\V_1(s)|\leq\left(\frac{s_0}{s}\right)^{1-\eta(1-\delta)}s_0^{\eta(1-\delta)}\V_1(s_0)+1\lesssim 1.
        \end{equation}
        Therefore \eqref{estimate:bootstrap-improve-V1} holds for large enough $s_0$.

        \medskip We now prove \eqref{equation:transverse-crossing} to end the proof of Proposition~\ref{proposition:reduce-finite-dimension}. Using \eqref{equation:relation-sVs} and \eqref{defeq:diag-Al}, we compute at the time $s=s^*$:
        \begin{align}
            \left.\frac{1}{2}\frac{d}{ds}\left(|s^{\frac{\eta}{2}(1-\delta)}\V_k(s)|^2\right)\right|_{s=s^*}&=\left.\left(s^{\eta(1-\delta)-1}\left[\frac{\eta}{2}(1-\delta)\V_k^2(s)+s\V_k(\V_k)_s\right]\right)\right|_{s=s^*}\\
            &=\left.\left(s^{\eta(1-\delta)-1}\left[\left[\frac{\eta}{2}(1-\delta)+\frac{k\gamma}{2l-\gamma}\right]\V_k^2(s)+O\left(\frac{1}{s^{\frac{3}{2}\eta(1-\delta)}}\right)\right]\right)\right|_{s=s^*}\\
            &\geq \frac{1}{s^*}\left[c(d,l)|(s^*)^{\frac{\eta}{2}(1-\delta)}\V_k(s^*)|^2+O\left(\frac{1}{(s^*)^{\frac{\eta}{2}(1-\delta)}}\right)\right]\\
            &\geq\frac{1}{s^*}\left[c(d,l)+O\left(\frac{1}{(s^*)^{\frac{\eta}{2}(1-\delta)}}\right)\right]>0.
        \end{align}
        Therefore we conclude the proof of Proposition~\ref{proposition:reduce-finite-dimension}.
\end{proof}
We are now ready to prove that there exists a solution trapped in $\Si_K(s)$ and close the proof of Theorem~\ref{theorem:main-theorem}.
\begin{proposition}[Existence of the solution through topological argument]\label{proposition:exist-solution-trapped}
    There exist suffciently large
constants $K,s_0$ and $M$, such that there exists initial data for the unstable modes
    \begin{equation}
        (\V_2(s_0),\dots,\V_{l}(s_0))\in[-s_0^{-\frac{\eta}{2}(1-\delta)},s_0^{-\frac{\eta}{2}(1-\delta)}]^{l-1}
    \end{equation}
    such that the corresponding solution $(b(s),\eps(s))$ is trapped in $\Si_K(s)$ for all $s\geq s_0$.
\end{proposition}
\begin{proof}
    Suppose any choice of the initial data will lead to the exit in $\Si_K(s)$. We define
    \begin{equation}
        \Z_k(s)=s^{\frac{\eta}{2}(1-\delta)}\V_k(s),\quad \Z(s)=(\Z_2(s),\dots,\Z_l(s)).
    \end{equation}
    Fix the other initial parameters and parameterize the family by
$\xi \in \mathcal{C}=[-1,1]^{l-1}$ with $Z(s_0;\xi)=\xi$. Define
\[
    s^*(\xi)
    =
    \sup\left\{
        s\geq s_0:
        (b(\sigma;\xi),\epsilon(\sigma;\xi))
        \in \mathcal{S}_K(\sigma)
        \text{ for every } \sigma\in[s_0,s]
    \right\},
\]
\[
    \Upsilon(\xi)
    =
    \frac{1}{100}Z(s^*(\xi);\xi)
    \in \partial\mathcal{C}.
\]
The local flow in Section~\ref{sec:bootstrap} and Proposition~\ref{proposition:reduce-finite-dimension} make this first-exit map
continuous: continuous dependence keeps nearby trajectories inside before exit,
and strict outward crossing on any active face forces exit soon afterward.
This also applies at corners, so $s^*$ and $\Upsilon$ are continuous.

    We claim that
    \begin{equation}
        \Upsilon(\xi)\neq-\xi,\quad \forall \xi\in\partial[-1,1]^{l-1}.
    \end{equation}
    For $\xi\in\partial\mathcal{C}$, choose $k$ with $|\xi_k|=1$. By Proposition~\ref{proposition:reduce-finite-dimension}, $\xi_k=1$ implies $Z_k(s;\xi)\geq 1$, whereas $\xi_k=-1$ implies $Z_k(s;\xi)\leq -1$, for $s_0\leq s\leq s^*(\xi)$. Hence
    \begin{equation}
        \xi_k\Z_k(s;\xi)\geq 1,\quad s_0\leq s\leq s_*.
    \end{equation}
    Therefore
    \begin{equation}
        \xi_k\Upsilon_k(\xi)=\frac{1}{100}\xi_k\Z_k(s_*;\xi)\geq\frac{1}{100},
    \end{equation}
    which is a contradiction to 
    \begin{equation}
       \xi_k\Upsilon_k(\xi)=\xi_k(-\xi_k)=-1. 
    \end{equation}

     We now define
     \begin{equation}
         \overline{\Upsilon}(\xi)=-\Upsilon(\xi).
     \end{equation}
     Then $\overline{\Upsilon}:[-1,1]^{l-1}\rightarrow[-1,1]^{l-1}$ is continuous. By using Brouwer fixed point theorem, there exist $\xi_*$ such that $\xi_*=\overline{\Upsilon}(\xi_*)$. Since $\overline{\Upsilon}([-1,1]^{l-1})\subset\partial[-1,1]^{l-1}$, we have $\xi_*\in\partial[-1,1]^{l-1}$. This  leads to
     \begin{equation}
         \Upsilon(\xi_*)=-\xi_*,
     \end{equation}
     which is a contradiction to our claim. Thus we concludes the proof of Proposition~\ref{proposition:exist-solution-trapped}.
\end{proof}
\begin{proof}[Proof of Theorem~\ref{theorem:main-theorem}:]
    Given the solution in Proposition~\ref{proposition:exist-solution-trapped}, we estimate from \eqref{estimate:lam-s-bound-inbootstrap}:
    \begin{equation}
        -\lam\lam_t=c(u_0)\lam^{\frac{2l-\gamma}{l}}(1+o(1)),
    \end{equation}
    which implies
    \begin{equation}
        \lam(t)=c(u_0)(1+o(1))(T-t)^{\frac{l}{\gamma}}.
    \end{equation}

    To control the local $L^{\infty}$ norm, we apply part (i) of Lemma~\ref{lemma:inter-bound-origin} and the asymptotic behaviour of $T_k$  in Lemma~\ref{lemma:construction-Tk} to get
    \begin{equation}
        |\eps(y)| = \left|\sum_{i=1}^{m}c_iT_{m-i}(y)+r_{\eps}(y)\right|\lesssim \sqrt{\E_{2m}}\rightarrow0,\quad \text{as $s\rightarrow+\infty$},\quad \forall y<1.
    \end{equation}
    Recalling part (ii) of Lemma~\ref{lemma:inter-bound-far}, we can obtain the following with a fixed truncation $B>1$:
    \begin{equation}
        |\eps(y)|\lesssim y^m\sqrt{\E_{2m}}\rightarrow0,\quad \text{as $s\rightarrow+\infty$},\quad \forall1\leq y\leq B.
    \end{equation}

    For $h+2\leq k\leq m-1$, put $\tilde{k}=2k$. Applying \eqref{equation:appendix-weight-bound-1} with derivative order $i$ and weight index $j=\tilde{k}-i+2$ gives
\[
    \int_{y>1}\left|\partial_y^{2k}\varepsilon\right|^2
    \lesssim
    \sum_{i=1}^{\tilde{k}} \int_{y>1} y^{-2(\tilde{k}-i)}
    \left|\partial_y^i\varepsilon\right|^2
    \lesssim \mathcal{E}_{2k+2},
\]
which controls all terms in $\nabla^{\tilde{k}}\varepsilon$ outside the unit ball. Near zero, Lemma~\ref{lemma:inter-bound-origin} separates smooth radial profiles from a remainder satisfying
\[
    \left|\nabla^{\tilde{k}} r_\varepsilon\right|
    \lesssim
    \sqrt{\mathcal{E}_{2m}}
    (1+|\log y|)^m
    y^{2m-1-d/2-\tilde{k}}.
\]
Its squared radial integral has power $y^{4m-3-2\tilde{k}}$, integrable for $\tilde{k}\leq 2m-2$; the profiles and the remaining compact annulus are also controlled by $\mathcal{E}_{2m}$. Thus
\[
    \|\varepsilon\|_{\dot H^{2k}(\mathbb{R}^d)}^2
    \lesssim
    \mathcal{E}_{2m}+\mathcal{E}_{2k+2}
    \longrightarrow 0.
\]
These energy levels decay: in the low range $4(k+1)-d\geq 12-(d-4h)>0$, and the high-range exponents are positive for small $\eta$. Since $2h+4<d/2+2$, interpolation covers $\sigma\in[d/2+2,2m-2]$.

    Note the total remainder can be written as
    \begin{equation}
        \eps_{tot}(t,y)=\eps(s(t),y)+\tQ_{b(t)}(y)-Q(y).
    \end{equation}
Using Lemma~\ref{lemma:construction-Tk} and Proposition~\ref{proposition:approximate-profile}, for every integer
$\lfloor d/2+2\rfloor\le j\le2m-2$,
\begin{equation}
    \|\tQ_b-Q\|_{\dot H^j}^2
\lesssim
b_1^2+b_1^{j+\gamma+2-\frac{d}{2}-C_{L,j}\eta}
\longrightarrow0,
\end{equation}
provided $\eta$ is sufficiently small.
Moreover, $\tQ_b-Q\to0$ uniformly on every fixed
compact set. Combining these bounds with the preceding
estimates for $\epsilon$ proves the asserted convergence
of $\epsilon_{\mathrm{tot}}$.
\end{proof}

\noindent
{\bf Acknowledgments}. 
This research was supported in part by the National Science Foundation under Grant No. DMS-2508463. The project was initiated during Zirui Wang’s Summer Undergraduate Research Fellowship (SURF) at Caltech in the summer of 2025. The authors also gratefully acknowledge the generous support of the Choi Family Gift Fund and the Mike Yan Gift Fund.

\appendix

\section{Existence of the steady state solution}
\label{app:steady-state}
We now state the asymptotic behavior of the steady state, which plays a crucial role in our estimates.
\begin{lemma}[Existence and asymptotic behaviour of $Q$]\label{lemma:appendix-profile-Q}
    Let $d\geq10$. Then there exists a unique solution $Q$ to \eqref{equation:steady-state} satisfying
    \begin{equation}\label{condition:boundary-Q}
        Q(0)=1, \quad y^2Q(y)\rightarrow2,\quad y\rightarrow+\infty.
    \end{equation}
    Every profile is obtained from $Q$ by the scaling symmetry
    \begin{equation}
        Q_A(y)=AQ(\sqrt{A}y),\quad A>0.
    \end{equation}
    Moreover we have the following asymptotic behaviour.
    \begin{enumerate}
        \item [(i)] Asymptotic behaviour at the origin:
        \begin{equation}
            Q_A(y)=A-\frac{dA^2}{2(d+2)}y^2+O(y^4).
        \end{equation}
        \item [(ii)] Asymptotic behaviour at infinity: We distinguish several cases according to dimension $d$.
        \begin{enumerate}
            \item If $d>11$, then there exists a constant $c_{A,d}>0$ depending on $A$ and $d$ such that
            \begin{equation}
                Q_A(y)=y^{-2}(2-c_{A,d}\,y^{-\gamma}+O(y^{-2\gamma})),\quad y\rightarrow+\infty.
            \end{equation}
            Here $\gamma,\Gamma$ are defined in \eqref{defeq:gamma-Gamma}.
            \item If $d=11$, then $\gamma=3,\Gamma=6$, and there exists $c_A>0$ such that
            \begin{equation}
                Q_A(y)=y^{-2}(2-c_{A,d}\,y^{-3}+2c_{A,d}^2(\log y)y^{-6}+O(y^{-6})),\quad y\rightarrow+\infty.
            \end{equation}
            \item If $d=10$, then $\gamma=\Gamma=4$, and there exists $c_A>0$ such that
            \begin{equation}
                Q_A(y)=y^{-2}(2-c_{A,d}(\log y)y^{-4}+O(y^{-4})),\quad y\rightarrow+\infty.
            \end{equation}
        \end{enumerate}
    \end{enumerate}
\end{lemma}
\begin{proof}
    To begin, we make the change of variables. Let
    \begin{equation}
        x=\log y,\quad q(x)=y^2Q(y)=e^{2x}Q(e^x).
    \end{equation}
    The profile equation for $q$ can be written as
    \begin{equation}\label{equation:profile-q}
        q''+(d-4+q)q'+(d-2)q(q-2)=0.
    \end{equation}

    We now proceed in two steps. First, we prove the existence and uniqueness of the solution to \eqref{equation:profile-q} based on an analysis of the phase portrait in the $(q,q')$-plane. This phase-plane approach has also been used in \cite{Biernat2015}.

    \step{1} Solution to \eqref{equation:profile-q}: Recalling \eqref{condition:boundary-Q}, we can rewrite the boundary conditions for \eqref{equation:profile-q} as
    \begin{equation}
        q(-\infty)=0,\quad q(+\infty)=2.
    \end{equation}
    We also note here that in \eqref{equation:profile-q}, the scaling symmetry of
the original equation becomes translation in $x$:
\begin{equation}
    \lam^{-2}Q(x/\lam)\quad\iff\quad q(x-\log\lam).
\end{equation}
This shows in $(q,q')$-plane the solution orbit of $(q,q')$ does not change. Thus, the uniqueness of the heteroclinic orbit connecting $(0,0)$ to $(2,0)$ corresponds precisely to the uniqueness of the original profile $Q$ up to the scaling symmetry.

We now define the vector field
\begin{equation}\label{def:intergral-curve-F}
    F(q,q')=(q',-(d-4+q)q'+(d-2)q(2-q)),
\end{equation}
and a trapping region
\begin{equation}
    \mathcal{S}=\left\{(q,p):0<q<2,\quad q(2-q)<p<\frac{\gamma}{2}q(2-q)\right\},
\end{equation}
which includes critical points $(0,0)$ and $(2,0)$. We need to show no integral curve of $F$ starting in $\mathcal{S}$ will leave $\mathcal{S}$.

Indeed, for $\alpha>0$ we set
\begin{equation}
    G_{\alpha}(q,p)=p-\alpha q(2-q).
\end{equation}
Therefore we can compute on the integral curve of $F$ as defined in \eqref{def:intergral-curve-F},
\begin{equation}\label{equation:compute-integral-curve}
    \begin{aligned}
        \frac{d}{dx}G_{\alpha}(q,p)&=p'-\alpha q'(2-2q)\\
        &=q(2-q)\left[d-2-\alpha(d-4+q)-2\alpha^2(1-q)\right].
    \end{aligned}
\end{equation}
Then for the lower boundary $\alpha=1$, we have
\begin{equation}
    \frac{d}{dx}G_{1}(q,p)=q^2(2-q)>0
\end{equation}
for $q\in(0,2)$. Therefore the vector field points into the region $p>q(2-q)$. For the upper boundary $\alpha=\gamma/2$, we use the identity $\gamma^2-(d-2)\gamma+2(d-2)=0$ and compute
\begin{equation}
    \frac{d}{dx}G_{\gamma/2}(q,p)=-\frac{\gamma(\gamma-1)}{2}q(2-q)^2<0.
\end{equation}
This shows the vector field points into the region $p<(\gamma/2)q(2-q)$. This concludes the claim that every integral curve of \eqref{def:intergral-curve-F} enters $\mathcal{S}$ remains in $\mathcal{S}$.

Finally we observe that near the origin, we have
\begin{equation}
    q(x)=Ae^{2x}-\frac{dA^2}{2(d+2)}e^{4x}+O(e^{6x}).
\end{equation}
Thus we obtain
\begin{equation}
    p-q(2-q)=\frac{2A^2}{d+2}e^{4x}+O(e^{6x})>0,\quad x\ll-1 ,
\end{equation}
and
\begin{equation}
    \frac{\gamma}{2}q(2-q)-p=(\gamma-2)Ae^{2x}+O(e^{4x})>0.
\end{equation}
This implies that the integral curve starts at $(0,0)$ cannot escape $\mathcal{S}$, which completes the proof.

\step{2} Asymptotics at infinity: Now we compute the asymptotic behaviour of the solution as $x\rightarrow+\infty$. Let
\begin{equation}
    \tilde{q}(x)=q(x)-2.
\end{equation}
Then $\tilde{q}<0$ for all $x$, $\tilde{q}(x)\rightarrow0$ as $x\rightarrow+\infty$, and the equation of $\tilde{q}$ becomes
\begin{equation}\label{equation:tilde-q}
    \tilde{q}''+(d-2)\tilde{q}'+2(d-2)\tilde{q}=-\tilde{q}\tilde{q}'-(d-2)\tilde{q}^2.
\end{equation}
The linear characteristic polynomial for \eqref{equation:tilde-q} is
\begin{equation}
    P(r)=r^2+(d-2)r+2(d-2),
\end{equation}
with roots $-\gamma$ and $-\Gamma$ when $d>10$, and a double root $-4$ when $d=10$.

\case{I} $d>11$: In this case $\Gamma>2\gamma$, therefore classical asymptotic integration for analytic autonomous systems near a hyperbolic equilibrium
gives
\begin{equation}
    \tilde{q}(x)=ae^{-\gamma x}+O(e^{-2\gamma x}),\quad x\rightarrow+\infty.
\end{equation}
We need to show $a\neq0$. Indeed from the upper boundary of $\mathcal{S}$,
\begin{equation}\label{equation:relation-p-q}
    p=q'<\frac{\gamma}{2}q(2-q)
\end{equation}
Since $0<q<2$, the trapping inequality gives
$p<\gamma(2-q)$. Hence
\begin{equation}
    \frac{d}{dx}\bigl[e^{\gamma x}(2-q(x))\bigr]
=e^{\gamma x}\bigl[\gamma(2-q)-p\bigr]>0.
\end{equation}
Thus $e^{\gamma x}(2-q(x))$ is positive and increasing.
Its limit is $-a$, so $a<0$. Setting $c_A=-a>0$ we have
\begin{equation}
    Q_A(y)=y^{-2}(2-c_Ay^{-\gamma}+O(y^{-2\gamma})),\quad y\rightarrow+\infty.
\end{equation}

\case{II} $d=11$: In this case $\Gamma=2\gamma=6$, the resonance of $\Gamma=2\gamma$ leads to the existence of logarithmic factors. In particular, we have
\begin{equation}
    \tilde{q}(x)=ae^{-3x}+(Bx+b)e^{-6x}+O(x^2e^{-9x}).
\end{equation}
The same monotonicity argument as in Case I yields $a<0$. Substituting the above ansatz into \eqref{equation:tilde-q} we obtain
\begin{equation}
    B=2a^2.
\end{equation}
This shows
\begin{equation}
    Q_A(y)=y^{-2}(2-c_Ay^{-3}+2c_A^2(\log y)y^{-6}+O(y^{-6})),\quad y\rightarrow+\infty.
\end{equation}

\case{III} $d=10$: In this case $\gamma=\Gamma=4$, thus we have
\begin{equation}
    \tilde{q}(x)=(ax+b)e^{-4x}+O(x^2e^{-8x}).
\end{equation}
We claim $a\neq0$. If $a=0$, then the orbit lies on the one-dimensional invariant eigencurve tangent to $p=4\xi$ at $(\xi,p) = (0,0)$,
where $\xi=2-q$ and $p=q'$. Writing this curve as $p=\psi(\xi)$, applying \eqref{equation:profile-q} we obtain
\begin{equation}
    \frac{d\psi}{d\xi}=8-\xi-\frac{8(2-\xi)\xi}{\psi}.
\end{equation}
This gives the expansion
\begin{equation}
    \psi(\xi)=4\xi+\xi^2+O(\xi^3).
\end{equation}
However in dimension $10$ the upper boundary of the trapping region is
\begin{equation}
    p<\frac{\gamma}{2}(2-q)q=4\xi-2\xi^2,
\end{equation}
which is incompatible with the previous expansion for all sufficiently small $\xi>0$. Therefore $a\neq0$ and we have
\begin{equation}
     Q_A(y)=y^{-2}(2-c_A(\log y)y^{-4}+O(y^{-4})),\quad y\rightarrow+\infty.
\end{equation}
\end{proof}

\section{Coercivity of operators}
\label{app:coercivity}
In this section we derive the coercivity estimate of the linear operator $\Li$ under orthogonality condition \eqref{relation:orthogonality-condition}, which is used repeatedly in the preceding
estimates. Throughout this section, we let $d\geq11$ and $f\in\mathcal{D}_{\text{rad}}$ where
    \begin{equation}
        \mathcal{D}_{\text{rad}}=\{\text{$f\in C_c^{\infty}(\R^d)$ with radial symmetry}\}.
    \end{equation}
    We first recall the following standard Hardy-type estimate.
\begin{lemma}[Hardy-type inequalities]\label{lemma:Hardy-inequality}   
    We adopt the simplified notation for the integral in $\R^d$ introduced in \eqref{defeq;integrate-f}. Then we have: 
    \begin{enumerate}
        \item [(i)] Hardy inequality near the origin:
        \begin{equation}
            \int_{y\leq 1}\frac{|\partial_y f|^2}{y^{2i}}\geq\frac{(d-2-2i)^2}{4}\int_{y\leq 1}\frac{f^2}{y^{2+2i}}-C(d)f^2(1),\quad i=0,1,2.
        \end{equation}
        \item [(ii)] Hardy away from the origin: Let $\alpha>0,\alpha\neq\frac{1}{2}(d-2)$, then
        \begin{equation}
            \int_{y\geq 1}\frac{|\partial_y f|^2}{y^{2\alpha}}\geq\left(\frac{d-(2\alpha+2)}{2}\right)^2\int_{y\geq 1}\frac{f^2}{y^{2+2\alpha}}-C(\alpha,d)f^2(1).
        \end{equation}
    \end{enumerate}
\end{lemma}
\begin{proof}
    The proof can be found in \cite{MerleRaphaelRodnianski2015}.
\end{proof}
We then prove the coercivity property of $\B^*$ from the Hardy-type inequalities.
\begin{lemma}[Full coercivity of $B^*$]\label{lemma:coercivity-Bstar}
   Suppose $\alpha>0$. Then there exists $c_{\alpha}>0$ such that
   \begin{equation}\label{estimate:coercivity-Bstar}
       \int\frac{\M|\mathcal{B}^*f|^2}{y^{2i}(1+y^{2\alpha})}\geq c_{\alpha}\left(\int\frac{\M|\partial_yf|^2}{y^{2i}(1+y^{2\alpha})}+\int\frac{\M f^2}{y^{2i+2}(1+y^{2\alpha})}\right),\quad i=0,1,2.
   \end{equation}
\end{lemma}
\begin{proof}
    The proof follows the standard two-step argument:
    
    \step{1} Subcoercive estimate for $\B^*$: for $i=0,1,2$ and $\alpha\geq0$ we claim   \begin{equation}\label{estimate:subcoercivity-Bstar}
        \int\frac{\M|\mathcal{B}^*f|^2}{y^{2i}(1+y^{2\alpha})}\gtrsim\int\frac{\M|\partial_yf|^2}{y^{2i}(1+y^{2\alpha})}+\int\frac{\M f^2}{y^{2i+2}(1+y^{2\alpha})}-f^2(1)-\int\frac{\M f^2}{1+y^{2i+2\alpha+4}}.
    \end{equation}
    Indeed the structure of $\B^*$ in \eqref{defeq:linear-operator-in-V} and \eqref{property:decay-V} leads to the following by integrating by parts
    \begin{align}
        \int_{y\leq1}\frac{\M|\mathcal{B}^*f|^2}{y^{2i}(1+y^{2\alpha})}&\gtrsim\int_{y\leq1}\frac{y^2}{y^{2i}}\left|\partial_yf+\frac{d+1}{y}f+O(|yf|)\right|^2\\
            &\gtrsim \int_{y\leq1}\frac{|\partial_yf|^2}{y^{2i-2}}+(d+1)\int_{y\leq1}\frac{\partial_y(f^2)}{y^{2i-1}}+(d+1)^2\int_{y\leq1}\frac{f^2}{y^{2i}}-O\left(\int_{y\leq1}\frac{f^2}{y^{2i-4}}\right)\\
            &\gtrsim\int_{y\leq1}\frac{|\partial_yf|^2}{y^{2i-2}}+2i(d+1)\int_{y\leq1}\frac{f^2}{y^{2i}}+(d+1)f^2(1)-O\left(\int_{y\leq1}\frac{f^2}{y^{2i-4}}\right)\\
            &\gtrsim\int_{y\leq1}\frac{\M|\partial_yf|^2}{y^{2i}}+\int_{y\leq1}\frac{\M|f|^2}{y^{2i+2}}-O\left(\int_{y\leq1}y^2\M f^2\right).
    \end{align}
    Then for $y$ away from the origin, we use again \eqref{property:decay-V} to obtain
    \begin{equation}
        \int_{y\geq1}\frac{\M|\mathcal{B}^*f|^2}{y^{2i}(1+y^{2\alpha})}\gtrsim\int_{y\geq1}\frac{1}{y^{2i+2\alpha-4}}\left(\partial_yf+\frac{d+1-\gamma}{y}f\right)^2-\int_{y\geq1}\frac{f^2}{y^{2i+2\alpha}}.
    \end{equation}
    Applying Lemma \ref{lemma:Hardy-inequality} we compute
    \begin{equation}
        \begin{aligned}
            \int_{y\geq1}\frac{1}{y^{2i+2\alpha-4}}\left(\partial_yf+\frac{d+1-\gamma}{y}f\right)^2&=\int_{y\geq1}\frac{|\partial_y(y^{d+1-\gamma}f)|^2}{y^{2i+2\alpha-4+2(d+1-\gamma)}}=\int_{y\geq1}\frac{|\partial_yg|^2}{y^{2i+2\alpha-4+2(d+1-\gamma)}}\\
            &\gtrsim\int_{y\geq1}\frac{g^2}{y^{2i+2\alpha+2d-2\gamma}}-g^2(1)\gtrsim\int_{y\geq1}\frac{\M f^2}{y^{2i+2\alpha+2}}-f^2(1),
        \end{aligned}
    \end{equation}
    where we use $g=y^{d+1-\gamma}f$. Gathering the above bounds yields \eqref{estimate:subcoercivity-Bstar}.

    \step{2} Coercivity of $\B^*$: We now argue by contradiction to show the coercivity of $\B^*$. Suppose that \eqref{estimate:coercivity-Bstar} fails. Then there exists a sequence $f_n\in\mathcal{D}_{\text{rad}}$ such that
    \begin{equation}\label{estimate:sequence-Bstar-1}
        \int\frac{\M|\partial_yf_n|^2}{y^{2i}(1+y^{2\alpha})}+\int\frac{\M f_n^2}{y^{2i+2}(1+y^{2\alpha})}=1,\quad  \int\frac{\M|\mathcal{B}^*f_n|^2}{y^{2i}(1+y^{2\alpha})}\leq\frac{1}{n}.
    \end{equation}
    Recalling \eqref{estimate:subcoercivity-Bstar} we get
    \begin{equation}\label{estimate:sequence-Bstar-2}
        f_n^2(1)+\int\frac{\M f_n^2}{1+y^{2i+2\alpha+4}}\gtrsim1.
    \end{equation}
    We note from \eqref{estimate:sequence-Bstar-1} the sequence $f_n$ is bounded in  $H_{\text{loc}}^1$.  Therefore we use a standard diagonal extraction argument to obtain that there exists $f_{\infty}\in H_{\text{loc}}^1$ such that up to a subsequence,
    \begin{equation}
        f_n\rightharpoonup f_{\infty} \quad \text{in $H_{\text{loc}}^1$},
    \end{equation}
    and from the local compactness of one-dimensional Sobolev embeddings
    \begin{equation}
        f_n\rightarrow f_{\infty} \quad\text{in $L_{\text{loc}}^2$},\quad f_n(1)\rightarrow f_{\infty}(1).
    \end{equation}
    Using again \eqref{estimate:sequence-Bstar-1} and \eqref{estimate:sequence-Bstar-2} we show
    \begin{equation}\label{estimate:sequence-Bstar-3}
        f_{\infty}^2(1)+\int\frac{\M f_{\infty}^2}{1+y^{2i+2\alpha+4}}\gtrsim1\quad\text{and}\quad\int\frac{\M f_{\infty}^2}{y^{2i+2}(1+y^{2\alpha})}\lesssim1,
    \end{equation}
    which means $f_{\infty}\neq0$. On the other hand, from \eqref{estimate:sequence-Bstar-1} and the lower semicontinuity of norms for the
weak topology, we have
\begin{equation}
    \B^*f_{\infty}=0.
\end{equation}
Now we need to recall the structure of $\B^*$ in \eqref{relation:kernel-of-A-B}, this shows that there exists $\beta\neq0$ such that
\begin{equation}
    f_{\infty}=\frac{\beta}{y^{d-1}\T_Q},
\end{equation}
with asymptotic behaviour $\T_Q\sim y^2$ near the origin. Therefore
\begin{equation}
    \int_{y\leq1}\frac{\M f_{\infty}^2}{y^{2i+2}}\gtrsim\int_{y\leq1}\frac{y^4y^{-2d-2}}{y^{2i+2}}dy=+\infty,
\end{equation}
which contradicts \eqref{estimate:sequence-Bstar-3}.
\end{proof}
Now we present the coercivity of $\A$ under orthogonality conditions.
\begin{lemma}[Coercivity of $\A$]\label{lemma:coercivity-A}
    Let $p\geq0$ and $i=0,1,2$. Further assume that
    \begin{equation}
        \scl{\M f}{\Phi_M}=0\quad \text{when}\quad 2i+2p>d-2\gamma-2,
    \end{equation}
    where $\Phi_M$ is defined in \eqref{defeq:PhiM}. Then we have the coercivity of $\A$
    \begin{equation}\label{estimate:coercivity-A}
        \int\frac{\M|\A  f|^2}{y^{2i}(1+y^{2p})}\gtrsim\int\frac{\M|\partial_yf|^2}{y^{2i}(1+y^{2p})}+\int\frac{\M f^2}{y^{2i}(1+y^{2p+2})}.
    \end{equation}
\end{lemma}
\begin{proof}
   We follow the same two-step argument as in Lemma \ref{lemma:coercivity-Bstar}.

    \step{1} Subcoercivity of $\A$: We claim:
    \begin{equation}\label{estimate:subcoercivity-A}
        \int\frac{\M|\A f|^2}{y^{2i}(1+y^{2p})}\gtrsim\int\frac{\M|\partial_yf|^2}{y^{2i}(1+y^{2p})}+\int\frac{\M f^2}{y^{2i}(1+y^{2p+2})}-f^2(1)-\int\frac{\M f^2}{1+y^{2i+2p+4}}.
    \end{equation}
    Recall \eqref{defeq:linear-operator-in-V} and \eqref{property:decay-V} again we compute near the origin as
    \begin{align}
        \int_{y\leq1}\frac{\M|\A f|^2}{y^{2i}(1+y^{2p})}&\gtrsim\int_{y\leq1}\frac{y^2}{y^{2i}}\left|-\partial_yf+O(|yf
        |)\right|^2\\
        &\gtrsim\int_{y\leq1}\frac{|\partial_yf|^2}{y^{2i-2}}+\int_{y\leq1}\frac{f^2}{y^{2i-2}}-\int_{y\leq1}\left(\frac{1}{y^{2i-2}}+\frac{1}{y^{2i-4}}\right)f^2\\
        &\gtrsim \int_{y\leq1}\frac{\M|\partial_yf|^2}{y^{2i}}+\int_{y\leq1}\frac{\M |f|^2}{y^{2i}}-\int_{y\leq1}\M|f|^2.
    \end{align}
    Away from the origin, we have the following estimate by recalling \eqref{property:decay-V}
    \begin{align}
        \int_{y\geq1}\frac{\M|\A f|^2}{y^{2i}(1+y^{2p})}&\gtrsim\int_{y\geq1}\frac{1}{y^{2i+2p-4}}\left(\partial_yf+\frac{\gamma+2}{y}f\right)^2-\int_{y\geq1}\frac{f^2}{y^{2i+2p}}\\
            &=\int_{y\geq1}\frac{|\partial_y(y^{\gamma}f)|^2}{y^{2i+2p-4+2\gamma}}-\int_{y\geq1}\frac{f^2}{y^{2i+2p}}\\
            &\gtrsim\int_{y\geq1}\frac{f^2}{y^{2i+2p-2}}-f^2(1)-\int_{y\geq1}\frac{f^2}{y^{2i+2p}}\\
            &\gtrsim \int\frac{\M f^2}{y^{2i}(1+y^{2p+2})}-f^2(1)-\int\frac{\M f^2}{1+y^{2i+2p+4}},
    \end{align}
    where we apply Lemma \ref{lemma:Hardy-inequality} again. This shows \eqref{estimate:subcoercivity-A}.

    \step{2} We now follow the steps in the proof of \eqref{estimate:coercivity-Bstar}. This implies the existence of $f_{\infty}\neq0$ such that
    \begin{equation}
        \int\frac{\M f_{\infty}^2}{y^{2i}(1+y^{2p+2})}\lesssim1,\quad \A f_{\infty}=0.
    \end{equation}
    Then it follows from \eqref{relation:kernel-of-A-B} that
    \begin{equation}
        f_{\infty}=\beta \Lambda Q,\quad \beta\neq0.
    \end{equation}
    If $2i+2p>d-2\gamma-2$ we apply the orthogonality condition to obtain
    \begin{equation}
        0=\langle\M f_{\infty},\Phi_M\rangle=\beta M^{d-2\gamma}>0,
    \end{equation}
    which is a contradiction. Otherwise if $2i+2p\leq d-2\gamma-2$ we integrate over the region away from the origin to get
    \begin{equation}
        \int_{y\geq1}\frac{\M f_{\infty}^2}{y^{2i}(1+y^{2p+2})}\gtrsim\beta^2\int_{y\geq1}\frac{y^4y^{d-1}y^{-2\gamma-4}}{y^{2i+2p+2}}dy\gtrsim\int_{y\geq1}y^{-1}dy=+\infty,
    \end{equation}
    which is also a contradiction.
\end{proof}

We are now ready to provide the coercivity of the iterate of $\Li$.
\begin{lemma}[Coercivity of $\Li^k$]\label{lemma:coercivity-L}
    Let $k\in\N$. Assume we have the orthogonality condition
    \begin{equation}
        \scl{\M f}{\Li^i\Phi_M}=0,\quad0\leq i\leq k-h,
    \end{equation}
    where $h$ is defined in \eqref{defeq:h-delta}. Then it follows
    \begin{equation}
    \begin{aligned}
        \E_{2k+2}(f)&=\int\M |\Li^{k+1}f|^2\\
        &\gtrsim\int\frac{\M|\A\Li^kf|^2}{y^2}+\sum_{i=0}^k\int\frac{\M|\Li^if|^2}{y^2(1+y^{4(k-i)+2})}+\sum_{i=0}^{k-1}\int\frac{\M|\A\Li^if|^2}{y^4(1+y^{4(k-i)-2})}.
    \end{aligned}
    \end{equation}
\end{lemma}
\begin{proof}
    We prove this by induction on $k$. First when $k=0$, apply Lemma \ref{lemma:coercivity-Bstar} and Lemma \ref{lemma:coercivity-A} to obtain
    \begin{equation}
        \int\M|\Li f|^2=\int\M|\B^*\A f|^2\gtrsim\int\frac{\M|\A f|^2}{y^2}\gtrsim\int\frac{\M|\A f|^2}{y^2}+\int\frac{\M f^2}{y^2(1+y^2)}.
    \end{equation}
    We now assume the inequality holds for $k\geq0$ and consider the case $k+1$. We have the coercivity condition
    \begin{equation}
        \scl{\M f}{\Li^i\Phi_M}=0,\quad0\leq i\leq k+1-h.
    \end{equation}
    Take $g=\Li f$, then
    \begin{equation}
        \scl{\M g}{\Li^i\Phi_M}=0,\quad0\leq i\leq k-h.
    \end{equation}
    Using the induction hypothesis we get
    \begin{align}
        \int\M|\Li^{k+2}f|^2&=\int\M|\Li^{k+1}g|^2\\
            & \gtrsim\int\frac{\M|\A\Li^kg|^2}{y^2}+\sum_{i=0}^k\int\frac{\M|\Li^ig|^2}{y^2(1+y^{4(k-i)+2})}+\sum_{i=0}^{k-1}\int\frac{\M|\A\Li^ig|^2}{y^4(1+y^{4(k-i)-2})}\\
            &=\int\frac{\M|\A\Li^{k+1}f|^2}{y^2}+\sum_{i=1}^{k+1}\int\frac{\M|\Li^if|^2}{y^2(1+y^{4(k+1-i)+2})}+\sum_{i=1}^{k}\int\frac{\M|\A\Li^if|^2}{y^4(1+y^{4(k-i)+2})}.
    \end{align}
    We then observe if $k\geq h-1$, the orthogonality condition $\scl{\M f}{\Phi_M}$ still holds. When $k\leq h-2$ we have
    \begin{equation}
        4k+6\leq 4h-2\leq4\left(\frac{d}{4}-\frac{\gamma}{2}-\delta\right)-2\leq d-2\gamma-2.
    \end{equation}
    Therefore we can apply Lemma \ref{lemma:coercivity-Bstar} and Lemma \ref{lemma:coercivity-A} to obtain
    \begin{equation}
        \int\frac{\M|\Li f|^2}{y^2(1+y^{4k+2})}\gtrsim \int\frac{\M|\A f|^2}{y^4(1+y^{4k+2})}\gtrsim  \int\frac{\M|\A f|^2}{y^4(1+y^{4k+2})}+\int\frac{\M|f|^2}{y^4(1+y^{4k+4})}.
    \end{equation}
    which concludes the proof.
\end{proof}
Finally we provide the coercivity estimate for $\eps$, which follows directly from Lemma \ref{lemma:coercivity-L}.
\begin{lemma}\label{lemma:coercive-bound-eps}
    We have the following weighted bounds for $\eps$: for $1\leq k\leq m$,
    \begin{equation}
        \int\M|\eps_{2k}|^2+\sum_{i=0}^{2k-1}\int\frac{\M|\eps_{i}|^2}{y^2(1+y^{4k-2i-2})}\leq C(M)\E_{2k}.
    \end{equation}
\end{lemma}

\section{Interpolation bounds}
\label{app:interpolation}
In this section we provide the interpolation bounds for $\eps$. We first give the bounds near the origin.
\begin{lemma}[Bounds near the origin]\label{lemma:inter-bound-origin}
    \begin{enumerate}
        \item [(i)] Expansion near the origin: We can write
        \begin{equation}
            \eps=\sum_{i=1}^mc_i T_{m-i}+r_{\eps}
        \end{equation}
        with the bounds
        \begin{equation}
            \begin{gathered}
                |c_i|\lesssim\sqrt{\E_{2m}}\\
            |\partial_y^j r_{\eps}|\lesssim y^{2m-1-\frac{d}{2}-j}|\ln(y)|^m\sqrt{\E_{2m}},\quad 0\leq j\leq 2m-1,\quad y<1.
            \end{gathered}
        \end{equation}

        \item [(ii)] Pointwise bounds: For $y\leq 1$ we have
        \begin{equation}
            \begin{gathered}
                |\eps_{2i}|+|\partial_y^{2i}\eps|\lesssim |\ln(y)|^my^{-\frac{d}{2}+1}\sqrt{\E_{2m}},\quad  0\leq i\leq m-1\\
            |\eps_{2i-1}|+|\partial_y^{2i-1}\eps|\lesssim|\ln(y)|^m y^{-\frac{d}{2}}\sqrt{\E_{2m}},\quad  1\leq i\leq m.
            \end{gathered}
        \end{equation}
    \end{enumerate}
\end{lemma}
\begin{proof}
    \smallcase{i} We claim the following for $1\leq k\leq m$
    \begin{equation}\label{equation:taylor-expand-at-origin-eps}
        \eps_{2m-2k}=\sum_{i=1}^kc_{i,k}T_{k-i}+r_{2k}
    \end{equation}
    together with the bounds
    \begin{equation}
        \begin{gathered}
            |c_{i,k}|\lesssim\sqrt{\E_{2m}}\\
        |\partial_y^jr_{2k}|\lesssim y^{2k-1-\frac{d}{2}-j}|\ln(y)|^k\sqrt{\E_{2m}},\quad 0\leq j\leq2k-1,\quad y<1.
        \end{gathered}
    \end{equation}
    
    We prove this by induction in $k$. For $k=1$ we write
    \begin{equation}
        r_1(y)=\eps_{2m-1}(y)=\frac{1}{y^{d-1}\T_Q}\int_0^y\eps_{2m}\T_Qx^{d-1}dx.
    \end{equation}
    Using Lemma~\ref{lemma:asymp-coeficients} that $\T_Q\sim y^{2},\M_Q\sim y^2$ as $y\rightarrow0$, we obtain
    \begin{equation}
        |r_1(y)|\leq\frac{1}{y^{d+1}}\left(\int\M|\eps_{2m}|^2x^{d-1}dx\right)^{\frac{1}{2}} \left(\int x^{-2}x^{4}x^{d-1}dx\right)^{\frac{1}{2}}\lesssim y^{-\frac{d}{2}}\sqrt{\E_{2m}},\quad y<1. 
    \end{equation}
    So there exists $a\in(\frac{1}{2},1)$ such that
    \begin{equation}
        |\eps_{2m-1}(a)|^2\lesssim\int_{\frac{1}{2}\leq y\leq 1}\M|\eps_{2m-1}|^2\lesssim \E_{2m}.
    \end{equation}
    Next we define
    \begin{equation}
        r_2(y)=-\Lambda Q\int_a^y\frac{r_1}{\Lambda Q}dx,
    \end{equation}
    therefore
    \begin{equation}
        |r_2(y)|\lesssim 1\cdot y^{-\frac{d}{2}}\sqrt{\E_{2m}}\int_a^y1dx\lesssim y^{-\frac{d}{2}+1}\sqrt{\E_{2m}},\quad y<1.
    \end{equation}
    It is straightforward to get
    \begin{equation}
        \A r_2=r_1=\eps_{2m-1},\quad \Li r_2=\B^*\eps_{2m-1}=\eps_{2m}=\Li \eps_{2m-2},
    \end{equation}
    and we recall from \eqref{relation:kernel-Li} that $\ker(\Li)=\text{Span}\{\Lambda Q,\Gamma_Q\}$. $\Gamma_Q$ is singular with $y^{-d}$ at the origin, this shows there exists $c_2\in\R$ such that
    \begin{equation}
        \eps_{2m-2}=c_2\Lambda Q+r_2.
    \end{equation}
    Moreover there exists $a_1\in(1/2,1)$ such that
    \begin{equation}
        |\eps_{2m-2}(a_1)|^2\lesssim\E_{2m},
    \end{equation}
    which implies
    \begin{equation}
        |c_2|\lesssim\sqrt{\E_{2m}},\quad |\partial_yr_2|\lesssim|r_1|+\frac{r_2}{y}\lesssim y^{-\frac{d}{2}}|\ln(y)|\sqrt{\E_{2m}}
    \end{equation}
    and complete the case for $k=1$.

    We then assume \eqref{equation:taylor-expand-at-origin-eps} holds for $k\geq1$ and prove for $k+1$. We define
    \begin{equation}
        r_{2k+1}(y)=\frac{1}{y^{d-1}\T_Q}\int_{a_2}^yr_{2k}\T_Qx^{d-1}dx,\quad r_{2k+2}(y)=-\Lambda Q\int_{a_3}^y\frac{r_{2k+1}}{\Lambda Q}dx.
    \end{equation}
    Here $a_2,a_3$ are chosen to be $0$ or $1$ depending on whether the term inside the integral is integrable at the origin. In fact, the above choice of $a_2,a_3$ together with induction hypothesis lead to the following estimate
    \begin{equation}
        \begin{gathered}
            |r_{2k+1}|\lesssim \frac{1}{y^{d+1}}\int_{a_2}^y|\ln(x)|^kx^{2k-1-\frac{d}{2}}x^2x^{d-1}dx\sqrt{\E_{2m}}
        \lesssim |\ln(y)|^{k+\delta_0}y^{2k-\frac{d}{2}}\sqrt{\E_{2m}},\\
        |r_{2k+2}|\lesssim \int_{a_3}^y|\ln(x)|^{k+\delta_0}x^{2k-\frac{d}{2}}dx\sqrt{\E_{2m}}
        \lesssim |\ln(y)|^{k+\delta_0+\delta_1}y^{2k+1-\frac{d}{2}}\sqrt{\E_{2m}},
        \end{gathered}
    \end{equation}
    here $\delta_0=1$ if $2k-d/2=0$ and $\delta_0=0$ otherwise, $\delta_1=1$ if $2k+1-d/2=0$ and $\delta_1=0$ otherwise. Therefore in all the cases we obtain
    \begin{equation}
        |r_{2k+2}|\lesssim|\ln(y)|^{k+1}y^{2k+1-\frac{d}{2}}\sqrt{\E_{2m}}.
    \end{equation}
    Then direct computation shows
    \begin{equation}
        \Li\eps_{2m-2(k+1)}=\eps_{2m-2k}=\sum_{i=1}^kc_{i,k}T_{k-i}+r_{2k}=-\sum_{i=1}^kc_{i,k}\Li T_{k+1-i}+\Li r_{2k+2},
    \end{equation}
    which implies
    \begin{equation}
         \eps_{2m-2(k+1)}=-\sum_{i=1}^kc_{i,k} T_{k+1-i}+c_{2k+2}\Lambda Q+ r_{2k+2}.
    \end{equation}
    Same as the case for $k=1$ we can choose $a_4\in(\frac{1}{2},1)$ such that
    \begin{equation}
        |\eps_{2m-2(k+1)}(a_4)|^2\lesssim\E_{2m}.
    \end{equation}
    Thus
    \begin{equation}
        |c_{2k+2}|\lesssim\sqrt{\E_{2m}}
    \end{equation}
    and this completes the proof.

    \smallcase{ii} The proof follows directly from \eqref{equation:taylor-expand-at-origin-eps}.
\end{proof}
\begin{lemma}[Bounds away from the origin]\label{lemma:inter-bound-far}
    \begin{enumerate}
        \item [(i)] Weighted bounds: We have for $i+j\leq2m$
        \begin{equation}\label{equation:appendix-weight-bound-1}
            \int\frac{\M|\partial_y^i\eps|^2}{1+y^{2j}}\lesssim\left\{
    \begin{aligned}
         & \E_{2k}, && \text{if $i+j=2k$, $1\leq k\leq m$},\\
       & \sqrt{\E_{2k}}\sqrt{\E_{2k+2}}, && \text{if $i+j=2k+1$, $1\leq k\leq m-1$},
    \end{aligned}
    \right. 
        \end{equation}
        and
        \begin{equation}\label{equation:appendix-weight-bound-2}
            \int\frac{\M|\partial_y^i\eps|^2}{1+y^{2j+1}}\lesssim\left\{
    \begin{aligned}
         & \E_{2k}^{\frac{3}{4}}\E_{2k+2}^{\frac{1}{4}}, && \text{if $i+j=2k$, $1\leq k\leq m-1$},\\
       & \E_{2k}^{\frac{1}{4}}\E_{2k+2}^{\frac{3}{4}}, && \text{if $i+j=2k+1$, $1\leq k\leq m-1$}.
    \end{aligned}
    \right. 
        \end{equation}
        \item[(ii)] Pointwise bound at infinity: We set $1\leq i+j\leq 2m-1$. For $y>1$ the following estimate holds
        \begin{equation}
            \frac{|\partial_y^i\eps|^2}{y^{2j}}\lesssim\left\{
    \begin{aligned}
         & \frac{1}{y^{d+2}}\E_{2k}, && \text{if $i+j+1=2k$, $1\leq k\leq m$},\\
       & \frac{1}{y^{d+2}}\sqrt{\E_{2k}}\sqrt{\E_{2k+2}}, && \text{if $i+j=2k$, $1\leq k\leq m-1$}.
    \end{aligned}
    \right. 
        \end{equation}
        Also if we set $1\leq i+j\leq 2m-2$, we have
        \begin{equation}
            \frac{|\partial_y^i\eps|^2}{y^{2j+1}}\lesssim\left\{
    \begin{aligned}
         & \frac{1}{y^{d+2}}\E_{2k}^{\frac{3}{4}}\E_{2k+2}^{\frac{1}{4}}, && \text{if $i+j+1=2k$, $1\leq k\leq m-1$},\\
       & \frac{1}{y^{d+2}}\E_{2k}^{\frac{1}{4}}\E_{2k+2}^{\frac{3}{4}}, && \text{if $i+j=2k$, $1\leq k\leq m-1$}.
    \end{aligned}
    \right. 
        \end{equation}
    \end{enumerate}
\end{lemma}
\begin{proof}
    \smallcase{i} Near the origin, recalling $\mathcal{M}(y)y^{d-1}\simeq y^{d+1}$, we apply Lemma~\ref{lemma:coercive-bound-eps} to get
\begin{equation}
    \sum_{i=0}^{2k}\int_{y<1}\mathcal M|\partial_y^i\epsilon|^2\lesssim\sum_{i=0}^{2k}\int_{y<1}\mathcal M\frac{|\partial_y^i\epsilon|^2}{1+y^{4k-2i}}
\lesssim \mathcal E_{2k}.
\end{equation}
    We observe for $i+j=2k$ with $1\leq k\leq m$, we have
    \begin{equation}
        \begin{aligned}
            \sum_{i=0}^{2k}\int\frac{\M|\partial_y^i\eps|^2}{1+y^{4k-2i}}&\lesssim\E_{2k}+\sum_{i=0}^{2k-1}\int_{y<1}y^2|\partial_y^i\eps|^2+\sum_{i=0}^{2k-1}\int_{y>1}\frac{\M|\partial_y^i\eps|^2}{y^{4k-2i}}\\
        &\lesssim \E_{2k}+\E_{2k}+\sum_{i=0}^{2k-1}\sum_{j=0}^i\int_{y>1}\frac{\M|\eps_j|^2}{y^{4k-2j}}\lesssim\E_{2k}.
        \end{aligned}
    \end{equation}
    Then for $i+j=2k+1$ we use the Cauchy-Schwarz inequality to obtain
    \begin{equation}
        \int\frac{\M|\partial_y^i\eps|^2}{1+y^{2j}}\lesssim\left(\int\frac{\M|\partial_y^i\eps|^2}{1+y^{4k-2i}}\right)^{\frac{1}{2}}\left(\frac{\M|\partial_y^i\eps|^2}{1+y^{4k-2i+4}}\right)^{\frac{1}{2}}\lesssim\sqrt{\E_{2k}}\sqrt{\E_{2k+2}},
    \end{equation}
    which implies \eqref{equation:appendix-weight-bound-1}. The proof of \eqref{equation:appendix-weight-bound-2} can be obtained in the same way by means of the Cauchy-Schwarz inequality.
    
    \smallcase{ii} We apply the weighted bounds \eqref{equation:appendix-weight-bound-1} and \eqref{equation:appendix-weight-bound-2} to obtain
    \begin{equation}
        \begin{aligned}
            \left|\frac{\partial^i_y\eps}{y^j}\right|^2&\lesssim\left|\int_y^{+\infty}\partial_x\left(\frac{(\partial_x^i\eps)^2}{x^{2j}}\right)dx\right|\lesssim\frac{1}{y^{d-1+4-1}}\left(\int_y^{+\infty}\frac{\M|\partial_x^i\eps|^2}{x^{2j+2}}dx+\int_{y}^{+\infty}\frac{\M|\partial_x^{i+1}\eps|^2}{x^{2j}}dx\right)\\
        &\lesssim\frac{1}{y^{d+2}}\left\{
    \begin{aligned}
         & \E_{2k}, && \text{if $i+j+1=2k$, $1\leq k\leq m$},\\
       & \sqrt{\E_{2k}}\sqrt{\E_{2k+2}}, && \text{if $i+j=2k$, $1\leq k\leq m-1$},
    \end{aligned}
    \right. 
        \end{aligned}
    \end{equation}
    where we use the Cauchy-Schwarz inequality. We also have
    \begin{equation}
        \begin{aligned}
            \frac{|\partial^i_y\eps|^2}{y^{2j+1}}&\lesssim\left|\int_y^{+\infty}\partial_x\left(\frac{(\partial_x^i\eps)^2}{x^{2j+1}}\right)dx\right|\lesssim\frac{1}{y^{d-1+4-1}}\left(\int_y^{+\infty}\frac{\M|\partial_x^i\eps|^2}{x^{2j+3}}dx+\int_{y}^{+\infty}\frac{\M|\partial_x^{i+1}\eps|^2}{x^{2j+1}}dx\right)\\
        &\lesssim\frac{1}{y^{d+2}}\left\{
    \begin{aligned}
         & \E_{2k}^{\frac{3}{4}}\E_{2k+2}^{\frac{1}{4}}, && \text{if $i+j+1=2k$, $1\leq k\leq m-1$},\\
       & \E_{2k}^{\frac{1}{4}}\E_{2k+2}^{\frac{3}{4}}, && \text{if $i+j=2k$, $1\leq k\leq m-1$},
    \end{aligned}
    \right. 
        \end{aligned}
    \end{equation}
    which completes the proof.
\end{proof}

We conclude the appendix by stating the Leibniz rule for $\Li^k$ as below.
\begin{lemma}[Leibniz rule for $\Li^k$]\label{lemma:lebniz-rule}
    Let $\phi$ be a smooth function and $k\in\N$, we have the following
    \begin{equation}\label{equation:lebniz-relation}
        \begin{gathered}
            \Li^{k+1}(\phi f)=\sum_{i=0}^{k+1}f_{2i}\phi_{2k+2,2i}+\sum_{i=0}^kf_{2i+1}\phi_{2k+2,2i+1},\\
            \A\Li^k(\phi f)=\sum_{i=0}^{k}f_{2i+1}\phi_{2k+1,2i+1}+\sum_{i=0}^kf_{2i}\phi_{2k+1,2i},
        \end{gathered}
    \end{equation}
    where
    \begin{equation}
        f_0=f,\quad f_{2i}=\Li^{i}f,\quad f_{2i+1}=\A\Li^{i}f.
    \end{equation}
    And we have the following recurrence relations: for $k=0$,
    \begin{gather}
        \phi_{1,0}=-\partial_y\phi,\quad\phi_{1,1}=\phi,\\
            \phi_{2,0}=-\partial_{yy}\phi-y(V^A+V^B)\partial_y\phi,\quad\phi_{2,1}=2\partial_y\phi,\quad \phi_{2,2}=\phi,
    \end{gather}
    and for $k\geq1$,
    \begin{align}
        &\phi_{2k+1,0}=-\partial_y\phi_{2k,0},\quad \phi_{2k+1,2i}=-\partial_y\phi_{2k,2i}-\phi_{2k,2i-1},\quad1\leq i\leq k,\\
            &\phi_{2k+1,2i+1}=\phi_{2k,2i}+y(V^A+V^B)\phi_{2k,2i+1}-\partial_y\phi_{2k,2i+1},\quad 0\leq i\leq k-1,\\
            &\phi_{2k+2,2k+2}=\phi_{2k+1,2k+1}=\phi_{2k,2k}=\phi,\quad\phi_{2k+2,0}=\partial_y\phi_{2k+1,0}+y(V^A+V^B)\phi_{2k+1,0},\\
            &\phi_{2k+2,2i}=\phi_{2k+1,2i-1}+\partial_y\phi_{2k+1,2i}+y(V^A+V^B)\phi_{2k+1,2i},\quad1\leq i\leq k,\\
            &\phi_{2k+2,2i+1}=-\phi_{2k+1,2i}+\partial_y\phi_{2k+1,2i+1},\quad 0\leq i\leq k.
    \end{align}
\end{lemma}
\begin{proof}
    It is clear from \eqref{defeq:linear-operator-in-V} that
    \begin{equation}
        \begin{gathered}
            \A(\phi f)=\phi\A f-\partial_y\phi f,\quad \B^*(\phi f)=\phi\B^*f+\partial_y\phi f,\\
            \A f+\B^*f = y(V^A+V^B)f.
        \end{gathered}
    \end{equation}
    Therefore we can compute
    \begin{equation}
        \begin{aligned}
            &\A(\phi f)=f_1\phi+f(-\partial_y\phi)\\
            &\Li(\phi f)=\B^*(f_1\phi+f(-\partial_y\phi))=f_2\phi+f_1(2\partial_y\phi)+f\left(-\partial_{yy}\phi-y(V^A+V^B)\partial_y\phi\right),
        \end{aligned}
    \end{equation}
    which concludes the case for $k=0$. Now we assume \eqref{equation:lebniz-relation} holds for $k$ and prove it for $k+1$. Using induction we can write
    \begin{align}
        \A\Li^{k+1}(\phi f)&=\sum_{i=0}^{k+1}\A(f_{2i}\phi_{2k+2,2i})+\sum_{i=0}^k[-\B^*+y(V^A+V^B)]f_{2i+1}\phi_{2k+2,2i+1}\\
            &=\sum_{i=0}^{k+1}[f_{2i+1}\phi_{2k+2,2i}+f_{2i}(-\partial_y\phi_{2k+2,2i})]\\
            &\quad+\sum_{i=0}^k[f_{2i+2}(-\phi_{2k+2,2i+1})+f_{2i+1}(-\partial_y\phi_{2k+2,2i+1})+f_{2i+1}y(V^A+V^B)\phi_{2k+2,2i+1}]\\
            &=\sum_{i=0}^kf_{2i+1}\left(\phi_{2k+2,2i}-\partial_y\phi_{2k+2,2i+1}+y(V^A+V^B)\phi_{2k+2,2i+1}\right)\\
            &\quad+\sum_{i=1}^kf_{2i}(-\partial_y\phi_{2k+2,2i}-\phi_{2k+2,2i-1})+f_{2k+3}\phi_{2k+2,2k+2}+f(-\partial_y\phi_{2k+2,0}),
    \end{align}
    which yields the recurrence relation for $\phi_{2k+3,j}$. Similarly we can write $\Li^{k+2}(\phi f)=\B^*(\A\Li^{k+1}(\phi f))$ and derive the recurrence relation for $\phi_{2k+4,j}$.
\end{proof}

\bibliographystyle{plain}
\bibliography{reference} 

\end{document}